\documentclass{article}
 
\usepackage{amssymb}
\usepackage{mathrsfs}
\usepackage{amsthm}

\begin{document}
 
\newtheorem{thm}{Theorem}[section]
\newtheorem{defn}[thm]{Definition}
\newtheorem{lem}[thm]{Lemma}
\newtheorem{prop}[thm]{Proposition}
\newtheorem{cor}[thm]{Corollary}
\newtheorem{rmk}[thm]{Remark}
\newtheorem{sublem}[thm]{Sublemma}
\newtheorem{convention}[thm]{Convention}
 
\renewcommand{\qed}{\hfill \mbox{\raggedright \rule{.1in}{.1in}}}
\def\co{\colon\thinspace}
 
\title{The quasi-isometry invariance of commensurizer subgroups}
\author{Diane M. Vavrichek\footnote{This research was partially supported by NSF grant DMS-0602191.}}
\date{}
\maketitle

\abstract{
We prove that commensurizers of two-ended subgroups with at least three coends in one-ended, finitely presented groups are invariant under quasi-isometries.
We discuss a variety of applications of this result.
}
 
\section{Introduction}

Any finitely generated group can be endowed with a metric, via the Cayley graph of the group with respect to any finite generating set.
The metrics we get on the group from two different finite generating sets may differ greatly, but they will induce the same ``coarse geometry'' on the group, in the sense that there will be a quasi-isometry taking the group endowed with one metric to the group endowed with the other.

The discovery and investigation of algebraic properties of groups that are invariant under quasi-isometries is an active area of geometric group theory.
Though many such properties have been found, there are few results that provide answers to the question that we are interested in here:  are there certain types of subgroups whose existence and location are invariant under quasi-isometries?

In this paper, we show that we can answer this in the affirmative in the case of certain commensurizer subgroups of one-ended, finitely presented groups.  We recall the definition of a commensurizer (also called a commensurator by some authors):

\begin{defn}
Let $G$ be a group with a subgroup $H$.  Then the {\em commensurizer of $H$ in $G$}, denoted Comm$_G(H)$, is the subgroup consisting of all $g \in G$ such that $H \cap gHg^{-1}$ is of finite index in $H$ and in $gHg^{-1}$.
\end{defn}

Our main result is the following.  We use $d_{Haus}$ to denote Hausdorff distance.
\\

\noindent {\bf Theorem \ref{Cqiinv}.} {\it
Let $f\co G \to G'$ be a $(\Lambda, K)$-quasi isometry between finitely presented, one-ended groups, and suppose that $H$ is a two-ended subgroup of $G$ with $n$ coends in $G$, for $n\geq 3$.
Then there is a two-ended subgroup $H'$ of $G'$ such that $H'$ has $n$ coends in $G'$ and there exists some constant $y=y(G, H, \Lambda, K)$ such that 
$$d_{Haus}(f({\rm Comm}_{G}(H)), {\rm Comm}_{G'}(H'))< y.$$
}

To prove this theorem, we make use of ``quasi-lines'', which were defined by Papasoglu in \cite{papa_qlines}, and were used to prove the main results in that work.
We make use of some of his results, and develop his theory further.

A variety of corollaries follow relatively easily from Theorem \ref{Cqiinv}.
First, we have the following.
\\

\noindent {\bf Corollary \ref{fn_cor}.} {\it
Let $f, G, G', H$ and $H'$ be as in Theorem \ref{Cqiinv}.  Then Comm$_G(H)$ is of type $F_n$ if and only if Comm$_{G'}(H')$ is of type $F_n$.
}
\\

This result follows from Theorem \ref{Cqiinv} and the theorem below, which we show by introducing a theory of ``coarse isometries'' (a more general notion than that of quasi-isometries), and applying Kapovich's arguments from \cite{Kap_ggt_book} that show that being of type $F_n$ is a quasi-isometry invariant.
\\

\noindent {\bf Theorem \ref{fn_thm}.} {\it
Let $f \co G \to G'$ be a quasi-isometry between finitely generated groups, and suppose that $C$ is a subgroup of $G$, $C'$ is a subgroup of $G'$, and that $d_{Haus}(f(C), C')<\infty$.
Then $C$ is of type $F_n$ if and only if $C'$ is of type $F_n$, for $n \geq 1$.
}
\\

For the next consequence of Theorem \ref{Cqiinv}, we recall that a JSJ decomposition of a group is a graph of groups decomposition that encapsulates the structure of the different splittings the group admits over two-ended subgroups, in the same way that a JSJ decomposition of a 3-manifold encapsulates the structure of the essential embeddings of annuli and tori into the manifold.  
Many different versions of these decompositions of groups have been defined (see \cite{KropJSJ}, \cite{SelaJSJ}, \cite{RipsSela}, \cite{BowCutPts}, \cite{BowJSJ}, \cite{DunSa}, \cite{DunSwen}, \cite{FujiPapa}). 
Scott and Swarup defined a version in \cite{SS}, which exists for all one-ended, finitely presented groups.

Papasoglu's work in \cite{papa_qlines} proves the invariance under quasi-isometries of many of the vertex groups of the Scott-Swarup JSJ decomposition.
Theorem \ref{Cqiinv} implies the invariance of some of the remaining vertex groups:
\\

\noindent {\bf Corollary \ref{SSJSJ_cor}.} {\it
Let $f \co G \to G'$ be a quasi-isometry between the Cayley graphs of one-ended, finitely presented groups $G$ and $G'$, and suppose that $C$ is a vertex group of commensurizer type of the Scott-Swarup JSJ decomposition of $G$.
Then there is a vertex group, $C'$, of commensurizer type of the Scott-Swarup JSJ decomposition of $G'$, which is such that
$$d_{Haus}(f(C), C')<\infty.$$
}

In the case when $G$ and $G'$ are 3-manifold groups, the Scott-Swarup JSJ decompositions are closely related to the JSJ decompositions of the associated 3-manifolds.  
In particular, Corollary \ref{SSJSJ_cor} implies the following.
\\

\noindent {\bf Corollary \ref{3man_app}.} {\it
Let $M$ and $M'$ be connected orientable Haken 3-manifolds with incompressible boundary, and with $f\co \pi_1(M) \to \pi_1(M')$ a quasi-isometry.  Suppose that $N$ is a nonexceptional Seifert fibered component of the characteristic submanifold of $M$ that meets the boundary of $M$.

Then there is a nonexceptional Seifert fibered component, $N'$, of the characteristic submanifold of $M'$ that meets the boundary of $M'$.
Moreover, if $C$ denotes the subgroup of $\pi_1(M)$ induced by the inclusion of $N$ into $M$ and $C'$ denotes the subgroup of $\pi_1(M')$ induced by the inclusion of $N'$ into $M'$, then
$$d_{Haus}(f(C), C')<\infty.$$
}

In \cite{KapLeeb_Haken}, M. Kapovich and Leeb prove a stronger result that implies Corollary \ref{3man_app}, using different methods.  See also \cite{KapLeeb_SFS} for a related result.

Two more corollaries to Theorem \ref{Cqiinv} are results about groups of quasi-isometries of groups.  
We denote the set of quasi-isometries from a metric space $X$ to a metric space $Y$, after identifying any maps of finite sup distance from one another, by $QI(X,Y)$, and we denote $QI(X,X)$ by $QI(X)$.
Of particular interest are the sets $QI(G)$, for $G$ a finitely generated group.
These sets form groups themselves, and are often quite complicated and difficult to study.

The following two results are implied immediately from the proof of Theorem \ref{Cqiinv}.
\\

\noindent {\bf Corollary \ref{qi1}.} {\it
Suppose that $G$ is a one-ended, finitely presented group such that $G = $ Comm$_G(H)$ for a two-ended subgroup $H$ of $G$ that has at least three coends in $G$.
Consider $G/H$, with the metric described in section \ref{qi(g)_section}.

Then there is a canonical map $QI(G) \to QI(G/H)$ that is surjective.
}
\\

\noindent {\bf Corollary \ref{qi2}.} {\it 
Let $G$ be a one-ended finitely presented group and let $N= \{ 3,4,5,\ldots\} \cup \{ \infty\}$.  
For any $n \in N$, let $K_n$ be a maximal collection of two-ended subgroups of $G$ with $n$ coends that have mutually infinite Hausdorff distance, and let $\{ K_n^j\}_{j \in J_n}$ be the partition of $K_n$ into sets of subgroups of quasi-isometric commensurizers.

Then there is a canonical map 
$$QI(G) \to \left\{ \prod_{n \in N} \prod_{H \in K_n} QI(M_H, M_{\sigma(H)}) \ : \ \sigma \in \prod_{n \in N} \prod_{j \in J_n} Sym(K_n^j)\right\} ,$$
where, for any two-ended subgroup $H$ of $G$, $M_H$ denotes the set Comm$_G(H)/H$, with a metric given in section \ref{qi(g)_section}, and $Sym(K_n^j)$ denotes the symmetric group on the set $K_n^j$.
}
\\

We note that Corollary \ref{qi1} is very similar in spirit to the main result in \cite{SoucheWiest}.

In section \ref{semi_section}, we prove one more corollary to Theorem \ref{Cqiinv}, which characterizes when a semidirect product of a free group with $\mathbb{Z}$ is quasi-isometric to the direct product of the free group with $\mathbb{Z}$:
\\

\noindent {\bf Corollary \ref{semi}.} {\it
Let $F_n$ be the free group on $n$ generators for any $n \geq 1$.
Then $F_n \rtimes \mathbb{Z}$ is quasi-isometric to $F_n \times \mathbb{Z}$ if and only if it is virtually $F_n \times \mathbb{Z}$.
}
\\

Other results on the topic of quasi-isometric semi-direct products and extensions have been obtained by Alonso and Bridson \cite{AlonsoBridson}, Bridson \cite{Bridson_optimal_isoper_ab_by_free}, and Bridson and Gersten \cite{BridsonGersten_optimal_isoper_torus_bundles}.
Alonso and Bridson give a sufficient condition for a group extension of two arbitrary groups to be quasi-isometric to the direct product of the groups, while the latter two works give necessary conditions for two semi-direct products (of the form $A \rtimes F$ for $A$ abelian and $F$ free in \cite{Bridson_optimal_isoper_ab_by_free} and of the form $\mathbb{Z}^n \rtimes \mathbb{Z}$ in \cite{BridsonGersten_optimal_isoper_torus_bundles}) to be quasi-isometric.
\\

\noindent {\bf Acknowledgements}  Much of the following work was done as part of the author's dissertation.
The author gratefully thanks her advisor, Peter Scott, for his help and guidance.  
The author also thanks Panos Papasoglu for several helpful discussions, and Mario Bonk, for his help with the proof of Theorem \ref{fn_thm} in the case that $n=1$.
In addition, thanks to Matt Brin and Ross Geoghegan for their encouragement and advice.



\section{Basic definitions and facts about quasi-lines in finitely presented groups}

In this section, we will introduce ``quasi-lines'', which were defined by Papasoglu in \cite{papa_qlines}, and were main objects of study in that work.
We will prove some basic properties of quasi-lines, including justifying why we may think of two-ended subgroups of finitely generated groups as quasi-lines inside the ambient groups.
We will then prove several results about complementary components of certain quasi-lines - in particular, that quasi-lines satisfy conditions $ess(m_0)$ and $iness(m_1)$ for some number $m_0$ and function $m_1$ in the settings that we are interested in.
Some of these results were used implicitly in \cite{papa_qlines}.

We will first set some basic notation and conventions.
Let $X$ be any metric space and let $x \in X$, $Y \subset X$ and $0 \leq r < \infty$.
Then $\overline{Y}$ shall denote the closure of $Y$ in $X$.
We set balls to be open and neighborhoods closed, i.e. let $B_r(x) = \{ z \in X : d(z,x)<r\}$ and $N_r(Y) = \{ z \in X : d(z,Y) \leq r\}$.

Let $Y'$ be another subset of $X$, and we shall denote by $d_{Haus} (Y,Y')$ the Hausdorff distance between $Y$ and $Y'$.  I.e.,
$$d_{Haus}(Y,Y') = \inf \{ r \geq 0 : Y \subset N_r(Y') \ {\rm and } \  Y' \subset N_r(Y)\}.$$

If $X'$ is another metric space, then a map $f \co X \to X'$ is a {\em $(\Lambda , K)$ quasi-isometry} if $\Lambda \geq 1 $ and $K \geq 0$ are such that, for any $x_1, x_2 \in X$, 
$$\frac{1}{\Lambda} d_X(x_1, x_2) -K \leq d_{X'} (f(x_1), f(x_2)) \leq \Lambda d_X(x_1, x_2) + K,$$
and $X' = N_K(f(X)).$
We say that $f$ is a quasi-isometry if $f$ is a $(\Lambda, K)$ quasi-isometry for some $\Lambda$ and $K$, and in this case we say that $X$ and $X'$ are quasi-isometric.

A map $g \co X' \to X$ is a {\em quasi-inverse} to $f$ if both $\sup_{x \in X}d(x, g \circ f (x))$ and $\sup_{x' \in X'} d(x', f \circ g (x'))$ are finite.

Suppose further that $X$ is a locally finite CW complex.  
Then the number of {\em ends} of $X$ is
$$e(X) = \sup |\{ {\rm infinite\  components\ of \ }X-K\} |,$$
where the supremum is taken over all finite subcomplexes $K$ of $X$.
Thus $e(X)$ can take on any value in $\mathbb{Z}_{\geq 0} \cup \{ \infty\}$.
The number of ends of a metric space is a quasi-isometric invariant:  if $X$ and $X'$ are quasi-isometric, then $e(X) = e(X')$.

\begin{convention}
We shall assume throughout this paper that all finitely generated groups that we deal with come equipped with a chosen finite generating set, and that all finitely presented groups come equipped with a chosen finite collection of defining relations.  
\end{convention}

If $G$ is a finitely generated group, then we shall denote by $\mathscr{C}^1(G)$ the Cayley graph for $G$ with respect to the associated finite generating set.
If $G$ is a finitely presented group, then we shall denote by $\mathscr{C}^2(G)$ the Cayley complex for $G$ with respect to the chosen finite generating set and relations.
We recall that $\mathscr{C}^2(G)$ is a simply connected, 2-dimensional CW complex with 1-skeleton equal to $\mathscr{C}^1(G)$, and that $G$ is identified with the vertex sets of $\mathscr{C}^1(G)$ and $\mathscr{C}^2(G)$.
The group $G$ acts cocompactly and by cellular isometries on $\mathscr{C}^1(G)$, and on $\mathscr{C}^2(G)$ if it exists.
Furthermore, this action is faithful, and transitive on vertices.
We will take these actions of $G$ to be on the left.

Recall that the Cayley graphs (and Cayley complexes, if they exist) of a group $G$ with respect to different finite generating sets (finite presentations, respectively) are all quasi-isometric.
Thus there is a well-defined notion of the number of ends of a finitely generated group:  if $G$ is finitely generated, then the number of ends of $G$, $e(G)$, is defined to equal $e(\mathscr{C}^1(G))$.
It is a fact that the number of ends of any finitely generated group can be one of only $0,1,2$ or $\infty$.
We have that $e(G) = 0$ if and only if $G$ is finite, $e(G) = 2$ if and only if $G$ contains a finite index infinite cyclic subgroup, and, by the work of Stallings, $e(G) = \infty$ if and only if $G$ splits over a finite subgroup and does not have a finite index infinite cyclic subgroup.
See \cite{ScottWall} for more details.

We will always take CW complexes to be metrized to have edges of length one, and to have the interiors of 2-cells be isometric to regular polygons.

Let $X$ be a CW complex and consider $\mathbb{R}$ as a graph by taking the integer points to be the vertices.
Then let $l\co \mathbb{R} \to X^{(1)}$ be continuous, injective and cellular, hence parametrized by arc length, i.e. with $length (l([s,t])) = d_\mathbb{R}(s,t)$, for all $s,t \in \mathbb{R}$.  Suppose further that $l$ is a uniformly proper map, so for every $M>0$, there exists an $N>0$ such that if $A \subset X$ with diam$(A)<M$, then diam$(l^{-1}(A))<N$.  Then we shall say that $l$ is a {\it line}, and we shall sometimes use $l$ to denote the image of $l$.  

To a line $l$, we associate the distortion function $D_l(t)\co \mathbb{R}_{\geq 0} \to \mathbb{R}_{\geq 0}$, where
$$D_l(t) = \sup \{ {\rm diam}(l^{-1}(A)) : {\rm diam}(A) \leq t \}.$$
As $l$ is parameterized by arc length, we note that $D_l(t) \geq t$ for all $t \in \mathbb{R}$.

Let $L$ be a closed, path connected subspace of $X$ containing a line $l$, with $N>0$ such that any point  in $L$ can be joined to $l$ by a path in $L$ of length less than or equal to 
$N$.  If $\phi\co \mathbb{R}_{\geq 0} \to \mathbb{R}_{\geq 0}$ is a proper, increasing function and, for all $t >0$, $D_l(t) \leq \phi(t)$, then we will say that $L$ is a {\it $(\phi, N)$ quasi-line}, or simply, a {\it quasi-line}.  We shall refer to $\phi$ and $N$ as {\em parameters for $L$}, and $l$ as the line {\em associated to $L$}.

We note that the assumption that $\phi$ be proper and increasing is not a strong one, for any line $l$ has $D_l$ bounded by some increasing function, and, as $D_l(t) \geq t$ for all $t$, any such function must be proper.

Observe that if $L$ is a $(\phi, N)$ quasi-line and $R>0$, then $N_R(L)$ is a $(\phi, N+R)$ quasi-line, and we may take the line associated to $L$ to also be associated to $N_R(L)$.


The following lemma shows that the restriction that a line be embedded is not an important one.

\begin{lem}\label{get_line}

Let $l'$ be a uniformly proper cellular map from $\mathbb{R}$ into a CW complex $X$ (taking $\mathbb{R}$ to be a simplicial complex with vertex set $\mathbb{Z}$).
Then there is a line $l \co \mathbb{R} \to X$ with Im$(l) \subset $ Im$(l')$, and such that $d_{Haus}($Im$(l),$Im$(l'))$ is finite, and bounded above by a function of $D_l$.
\end{lem}

\begin{proof}
If $l'$ is an embedding, then it suffices to take $l = l'$.  So suppose that $l'$ is not an embedding.

As $l'$ is uniformly proper, there is some maximal $n = n(l') \in \mathbb{R}$ such that the preimage of some point in Im$(l')$ has diameter $n$.
Because $l'$ is cellular,  note that $n \in \mathbb{Z}$.  
We shall induct on $n$.  
Note that $n>0$, as $l'$ is not an embedding.

Let $S$ denote a maximal disjoint set of closed intervals of size $n$ in $\mathbb{R}$ such that the endpoints of each interval are sent to the same point of $X$ by $l'$.  
Note that we can take $S$ to be such that the endpoints of each interval in $S$ are mapped by $l'$ to vertices of $X$.

Let $\iota \co \mathbb{R} \to \mathbb{R}$ denote the quotient map attained by identifying each component of $S \subset \mathbb{R}$ to a point, and define a right inverse to $\iota$, $\iota'$, to take each such point to an endpoint of its full preimage.  
Since the endpoints of each component of $S$ are identified by $l'$, there is a well-defined, continuous map $l_1: \mathbb{R} \to X$ defined by $l_1(t) = l'\circ \iota(t)$.
Intuitively, $l_1$ is the map we get by removing a disjoint collection of maximal loops from $l'$.

Clearly $l_1$ is cellular, thus is parameterized by arc length, and we note that 
$l_1$ is uniformly proper, for if $A$ is any subset of $X$, then diam$(l_1^{-1}(A)) \leq $ diam$((l')^{-1}(A))$.
Furthermore, we have that $d_{Haus}($Im$(l_1), $ Im$(l'))\leq \frac{1}{2}n(l')$.

It remains to show that $n(l_1)<n(l')$. 
To see this, let us suppose that there are $t_0, t_1 \in \mathbb{R}$ such that $|t_0-t_1| \geq n(l')$, and $l_1(t_0) = l_1(t_1)$.  
Suppose that $t_0$ is the image of a collapsed segment under $\iota$.
Then there exist two points $s, s' \in \mathbb{R}$ that are the endpoints of this segment, with $\iota (s) = \iota(s') = t_0$, $l'(s) = l'(s')$, and $|s-s'| = n(l')$.  
If $t_1$ is the image under $\iota$ of only a point, then let $s_1$ denote that point.  
If $t_1$ is the image of a segment under the map $\iota$, then let $s_1$ denote an endpoint of that segment.
Then $l'(s_1) = l'(s) = l'(s')$, and either $|s_1-s|>|s-s'|$ or $|s_1-s'|>|s-s'|$.  But $|s-s'|=n(l')$, so this contradicts the definition of the function $n$.

Thus we may suppose that $t_0$ and $t_1$ are images under $\iota$ of single points, say $s_0$ and $s_1$ respectively.  If $\iota$ collapses no segments in the interval $[s_0, s_1]$ then we reach another contradiction, for $[s_0, s_1]$ must be an interval of size $n(l')$, whose endpoints are mapped to the same vertex of $X$ by $l'$, and that is disjoint from $S$.  This contradicts the maximality of $S$.

Finally, suppose that $\iota$ collapses a segment $[s,s']$ in $[s_0, s_1]$.  Then $|s-s'| = n(l')$, so $|s_0- s_1|>n(l')$.  The endpoints $s_0$ and $s_1$ must share the same image under $l'$, and again this contradicts the definition of the function $n$.

Thus $n(l_1)<n(l')$.  If $l_1$ is not an embedding, then we can repeat this process on $l_1$, getting a map $l_2 \co \mathbb{R} \to X$ such that $n(l_2)<n(l_1)$, and so on.

Eventually we must get a map $l_k$ such that $n(l_k) = 0$, and hence $l_k=l$ is the desired line.
\end{proof}

Next, we will establish some basic facts about quasi-lines.
First, we give the connection between quasi-lines and infinite cyclic subgroups.
Recall that $G$ is identified with the vertex set of $\mathscr{C}^1(G)$.

\begin{lem}\label{2ended_subgps_are_qlines}
Let $G$ be a finitely generated group, and $H \subset G$ a two-ended subgroup.  If $R>0$ is large enough so that $N_R(H)\subset \mathscr{C}^1(G)$ is connected, then $N_R(H)$ is a quasi-line.
\end{lem}

\begin{proof}
Let $H$ be a two-ended subgroup of $G$, let $\langle h \rangle \cong \mathbb{Z}$ be a finite index subgroup of $H$, and let $R>0$ be such that $N_R(H)\subset \mathscr{C}^1(G)$ is connected.  Let $p$ be a simplicial path in $N_R(H)$ from the identity to $h$.  Let $l' \co \mathbb{R} \to N_R(H) \subset \mathscr{C}^1(G)$ be the natural map onto $\cup_{n \in \mathbb{Z}} h^n\cdot p$ that is parameterized by arc length.

Note that there is some $N>0$ depending on $\langle h \rangle$ and $R$ such that each point of $N_R(H)$ can be connected to Im$(l')$ by a path in $N_R(H)$ of length less than or equal to $N$.  If we can show that $l'$ is uniformly proper, then it will follow from Lemma \ref{get_line} that $N_R(H)$ is a quasi-line.

Fix any $M>0$, and recall that $p$ is a finite path and that $H$ acts freely on $\mathscr{C}^1(G)$.
It follows that, for any point $p_0 \in p$, the image under $(l')^{-1}$ of the vertices contained in $B_M(p_0)$ is finite, and hence $diam((l')^{-1}(B_M(p_0)))<\infty$.  
Let $D_M = \max \{ diam((l')^{-1}(B_M(p_0))) \}$, where the maximum is taken over all vertices $p_0$ of $p$.

Now suppose that $A \subset \mathscr{C}^1(G)$ has diameter less than $(M-N)$ and meets $N_R(H)$.

Then for any $a \in (A \cap N_R(H))$, certainly $A \subset B_{M-N}(a)$, and hence there is some $p_1 \in l'$ of distance no more than $N$ from $a$, and $A \subset B_M(p_1)$.
Thus there exists some $n \in \mathbb{Z}$ such that $h^nA \subset B_M(h^np_1)$ and $h^np_1$ is a vertex of $p$.  
Therefore we have that 
$$diam((l')^{-1}(A)) = diam((l')^{-1}(h^nA)) \leq diam ((l')^{-1}(B_M(h^np_1))) \leq D_M,$$
and it follows that $l'$ is uniformly proper, with $D_{l'}(M-N)$ bounded by $D_M$ for each $M>0$.
\end{proof}

The next lemma shows that quasi-lines are suitably invariant under quasi-isometries.

\begin{lem}\label{imgs_of_qlines_are_qlines}
Let $f \co \mathscr{C}^1(G) \to \mathscr{C}^1(G')$ be a $(\Lambda, K)$ quasi-isometry and let $L \subset \mathscr{C}^1(G)$ be a $(\phi, N)$ quasi-line.  Then there is some $R = R(\Lambda, K, \phi, N)> 0$ such that $N_R(f(L))$ is a quasi-line.
\end{lem}

\begin{proof}
Let $l$ be the line associated to $L$ and let $\pi$ denote nearest point projection of $\mathscr{C}^1(G')$ onto its vertex set.
Let $l'$ denote $\pi \circ f \circ l |_\mathbb{Z}$, and note that $l'$ is uniformly proper.
As $l$ is a line, for each $i \in \mathbb{Z}$, $d(l(i), l(i+1)) = 1$.
Hence $d (f \circ l(i), f\circ l(i+1)) \leq (\Lambda + K)$ so $d(l'(i), l'(i+1)) \leq (\Lambda + K + 1)$.

Let $\psi \co \mathbb{Z} \to \mathbb{R}$ be such that, for all $i\in \mathbb{Z}$, $\psi(i+1)-\psi(i) = d(l'(i), l'(i+1))$, so $\psi$ increases distances by no more than a factor of $(\Lambda + K + 1)$.
Let $l'' \co Im(\psi) \to \mathscr{C}^1(G')$ be such that $l' = l'' \circ \psi$, and it follows that $l''$ is uniformly proper.

Next, we extend the definition of $l''$ to all of $\mathbb{R}$ by mapping each $[\psi(i), \psi(i+1)]$ isometrically to a geodesic segment from $l'(i)$ to $l'(i+1)$.
Thus $l''$ is a unit speed, cellular map of $\mathbb{R}$ into $\mathscr{C}^1(G')$.
As any point in $l''(\mathbb{R})$ is of distance no more than $\frac{1}{2}(\Lambda + K + 1)$ from a point of $l''(Im(\psi))$ and $l''|_{Im(\psi)}$ is uniformly proper, it follows that $l'' \co \mathbb{R} \to \mathscr{C}^1(G')$ is uniformly proper.

Thus, by Lemma \ref{get_line}, there is a line $\hat{l} \co \mathbb{R} \to \mathscr{C}^1(G')$ such that $Im(\hat{l}) \subset Im(l'')$ and $d_{Haus}(Im(\hat{l}), Im(l''))$ is bounded by a constant depending only on $\Lambda, K, \phi$ and $N$.
There is some $R>0$ depending on the same constants such that $N_{R}(f(L))$ is connected and contains $Im(\hat{l})$, and thus is a quasi-line.
\end{proof}

\begin{rmk}\label{qlines_have_same_consts}
In the argument to prove our main result, Theorem \ref{Cqiinv}, we shall work with a quasi-isometry $f \co \mathscr{C}^1(G) \to \mathscr{C}^1(G')$, a two-ended subgroup $H$ of $G$, and its translates $\{ gH\}$.
We will also be concerned with $f(H)$ and its translates $\{ (g')f(H)\}$.
The previous results show that $H$ and $f(H)$ have neighborhoods that are quasi-lines.
Since left multiplication is an isometry in any Cayley graph, it follows that the translates of $H$ and $f(H)$ have neighborhoods that are quasi-lines, with the same parameters as those containing $H$ and $f(H)$ respectively.
\end{rmk}


We note next that quasi-lines are two-ended:

\begin{lem}\label{e(L)=2}
Let $L$ be a quasi-line contained in a locally finite CW complex $X$.  Then $e(L)=2$.
\end{lem}

\begin{proof}
Let $\phi$ and $N$ be parameters for $L$, and let $l \subset L$ be the line associated to $L$.  
Then every point in $L$ can be connected by a path of length less than or equal to $N$ to $l$, and $e(l) = 2$, so $e(L) \leq 2$.  
As the line $l$ is an injective and cellular map into the 1-skeleton of $X$, and $X$ is locally finite, we must have that $e(L) \geq 1$.

To see that $e(L)=2$, first note that if $a, b \in \mathbb{R}$ are such that $|a-b|>\phi(2N)$, then $d(l(a),l(b))>2N$.  Thus if we fix such $a$ and $b$, say with $a<b$, then, for any $q \in l((-\infty, a])$ and $q' \in l([b, \infty ))$, $d(q,q')>2N$.  Let $K$ be the set of all points $p \in L$ such that there is a path of length less than or equal to $N$ contained in $L$ that connects $p$ to $l((a,b))$.  Since $X$ is locally finite, $K$ is compact.  Thus $L-K$ contains two infinite components - one intersecting $l((-\infty, a])$ and one intersecting $l([b,\infty))$.
\end{proof}

The following definitions will be important, particularly to the remainder of this section.  

\begin{defn} \label{ess_iness_ness_def}
If $L$ is a quasi-line in a metric space $X$, then a connected component $C$ of $X-L$ is said to be {\em essential} if $C \cup L$ has one end.  Otherwise, $C$ is said to be {\em inessential}.  

If $C$ is not contained in $N_R(L)$, for any $R \geq 0$, then we shall say that $C$ is {\em nearly essential}.
\end{defn}

\begin{defn}
If there is some $m_0>0$ such that, for each $p \in L$, each essential component of the complement of $L$ intersects a vertex of $B_{m_0}(p)$, then we say that $L$ satisfies {\em ess}$(m_0)$.  
\end{defn}

\begin{defn}
If $m_1 \co \mathbb{R}_{\geq 0} \to \mathbb{R}_{\geq 0}$ is such that, for each $R \geq 0$, each inessential component of the complement of $N_R(L)$ is contained in the $m_1(R)$-neighborhood of $N_R(L)$, then we say that $L$ satisfies {\em iness}$(m_1)$.
\end{defn}

We will see shortly that all the quasi-lines in which we are interested (see Remark \ref{qlines_have_same_consts}) will satisfy $iness(m_1)$ for some $m_1$.  It follows that the components of the complements of these quasi-lines are essential if and only if they are nearly essential.

\begin{defn} $L$ is said to be {\it $n$-parting} if the complement of $L$ has at least $n$ essential components.
\end{defn}

Note that, if $L$ is $n$-parting and $R>0$, then $N_R(L)$ is also $n$-parting.

Next, we shall show that any quasi-line in a one-ended finitely presented group has only finitely many essential components in its complement, and moreover satisfies $ess(m_0)$ for some $m_0$.

\begin{defn} \label{defn_chain}
Let $L$ and $C$ denote subsets of a metric space $X$, let $n>0$, and let $x,y \in C \cap L$.  Then we shall say that $x$ and $y$ are connected by an {\em $(L,n)$-chain in $C \cap L$} if there are points $x=z_0, z_1, \ldots, z_k =y$ in $C\cap L$ such that, for each $i$, there is a path in $L$ connecting $z_i$ to $z_{i+1}$ of length less than or equal to $n$.
\end{defn}

We will begin by working in $\mathscr{C}^2(G)$.

\begin{lem}\label{CC_lem}
Let $G$ be a finitely presented group, let $L$ be a $(\phi, N)$ quasi-line in $\mathscr{C}^1(G)$, and let $0<\epsilon \ll 1$.  
Then let $L'$ be an open set in $\mathscr{C}^2(G)$ such that $L \subset L'$, each point in $L'$ can be connected to $L$ by a path in $\mathscr{C}^2(G)$ of length less than $\epsilon$, and each point in $L' \cap \mathscr{C}^1(G)$ can be connected to $L$ by a path in $\mathscr{C}^1(G)$ of length less than $\epsilon$.
Suppose in addition that both $\overline{L'}$ and $\overline{L'} \cap \mathscr{C}^1(G)$ deformation retract onto $L$. 

Then there is some $n_0 = n_0 (G, \phi, N, \epsilon)$ such that, if $C'$ denotes any component of $\mathscr{C}^2(G)-L'$, and $x, y \in C'\cap \overline{L'}$, then $x$ and $y$ are connected by an $(\overline{L'}, n_0)$-chain $\{ z_i\}$ in $C' \cap \overline{L'}$.  Moreover, the path in $\overline{L'}$ connecting any $z_i$ and $z_{i+1}$ can be taken to be in $L$, outside of an initial segment containing $z_i$ and a final segment containing $z_{i+1}$, each of length less than or equal to $\epsilon$.
\end{lem}

To prove this lemma, we first need the following.

\begin{lem}\label{L''_simp_conn}
Let $G, \mathscr{C}^2(G), L, L'$, and $\epsilon$ be as in Lemma \ref{CC_lem}.  

Let $(\mathscr{C}^2)''$ denote the union of $\mathscr{C}^2(G)$ together with a disk added at each closed edge path of $\overline{L'}$ of length less than or equal to $\phi (2(N+\epsilon)+1) + 2(N + \epsilon)+1$, and let $\overline{L''}$ denote the union of $\overline{L'}$ with these disks.
Then $\overline{L''}$ is simply connected.
\end{lem}

\begin{proof}
Let $l$ be the line associated to $L$, so that $l \subset L \subset L'$, and note that $\overline{L'}$ is a $(\phi , N+\epsilon)$ quasi-line.  
Recall that the edges of $\mathscr{C}^2(G)$ are of length one, and metrically, the 2-cells of $\mathscr{C}^2(G)$ are regular polygons.  
Let the disks added to create $(\mathscr{C}^2)''$ and $\overline{L''}$ be regular polygons as well.

Let $\gamma \co S^1 \to \overline{L''}$ be any closed curve in $\overline{L''}$.  
In the following we shall replace $\gamma$ with homotopic curves (which we shall also call $\gamma$), and we shall assume at each stage that $\gamma$ is parameterized by arc length.
Thus we shall consider copies of $S^1$ with different metrics as appropriate below.

Note first that $\gamma$ can be homotoped in $\overline{L''}$ to a cellular path that is contained in $\overline{L'} \cap [\mathscr{C}^2(G)]^{(1)} =\overline{L'} \cap ((\mathscr{C}^2)'')^{(1)}$.  Let $\gamma$ now denote the resulting curve under this homotopy.

Fix a finite collection $P$ of points in $Im(\gamma) - l$ such that the $(1/2)$-neighborhood of $P$ contains $Im(\gamma) - l$.
At each point $p_0$ in $P$, consider a spike from $p_0$ to $l$ in $\overline{L'}$ of length no more than $(N+\epsilon)$.
We shall homotope $\gamma$ to traverse the appropriate spike each time it meets a point of $P$.
We have thereby constructed a homotopy of $\gamma$ to a curve that is contained in $l$ except for finitely many segments in $\overline{L'} \cap [\mathscr{C}^2(G)]^{(1)}$ of length bounded by $2(N+\epsilon)+1$ 

Let $\Sigma$ denote the components of the complement of $\gamma^{-1}(l)$ in $S^1$, so each component of $\Sigma$ has length no more than $2(N+\epsilon) + 1$.
We will prove, by induction on $|\Sigma|$, that $\gamma$ must be null-homotopic in $\overline{L''}$.

Note that, after adding the spikes, $\gamma$ must meet $l$, so $\Sigma$ must consist of at least one segment.
If $|\Sigma| = 1$, then we shall let $\sigma$ denote the element in $\Sigma$, and let $p$ and $q$ denote the endpoints of $\sigma$.  
Then the distance between $\gamma(p)$ and $\gamma(q)$ in $\overline{L'} \cap ((\mathscr{C}^2)'')^{(1)}$ must be bounded by $2(N+\epsilon) + 1$, so $\gamma(p)$ and $\gamma(q)$ are connected in $l$ by a path of length no more than $\phi(2(N+\epsilon)+1)$.  
It follows that this path in $l$ together with $\gamma(\sigma)$ make up a circuit in $\overline{L'}$ of length no more than $\phi(2(N+\epsilon) + 1) + 2(N+\epsilon) + 1$, and thus must bound a disk in $\overline{L''}$.

Thus $\gamma$ may be homotoped in $\overline{L''}$ so that $\gamma(\sigma)$ is taken to this path in $l$, i.e. $\gamma$ may be homotoped within $\overline{L''}$ to lie entirely inside of $l$.
But $l$ is an embedding of $\mathbb{R}$, so $\gamma$ is null-homotopic in $\overline{L''}$.

Next, assume that $|\Sigma|=i>1$.  By choosing any element $\sigma$ of $\Sigma$, the same argument as was given above shows that we can homotope $\gamma$ within $\overline{L''}$ so that $\sigma$ is replaced by a path in $l$.
Thus we have homotoped $\gamma$ so that now $|\Sigma| = i-1$.
By repeating this process, we reduce to the case of $|\Sigma|=1$ above, and hence $\gamma$ must have originally been null-homotopic in $\overline{L''}$.
It follows that $\overline{L''}$ is simply connected, as desired.
\end{proof}

\begin{proof}[Proof of Lemma \ref{CC_lem}]
Again let $l$ be the line associated to $L$, and let $(\mathscr{C}^2)''$ and $\overline{L''}$ be as in Lemma \ref{L''_simp_conn}.

Recall that $\overline{L''}$ differs from $\overline{L'}$ only by disks with boundary in $\overline{L'}$, and that these disks are not contained in $\mathscr{C}^2(G)$.  
It follows that the intersection of $\overline{L'}$ with the closures of the components of $\mathscr{C}^2(G)-L'$ is the same as the intersection of $\overline{L''}$ with the closures of the components of $(\mathscr{C}^2)''-\overline{L''}$.  
Thus, if we can prove that the conclusion of Lemma \ref{CC_lem} holds for $\overline{L''}$ in $(\mathscr{C}^2)''$, then the lemma will follow for $\overline{L'}$ in $\mathscr{C}^2(G)$.

Let $\{ C_\alpha \}$ denote the components of the complement of $\overline{L''}$.
As $\overline{L''}$ is simply connected, we can apply Van Kampen's theorem to $\{ \overline{C_\alpha} \cup \overline{L''}\}$ to see that each $\overline{C_\alpha} \cup \overline{L''}$ is simply connected.  Since $\overline{L''}$ and each $\overline{C_\alpha}$ are connected, we have that $\overline{C_\alpha} \cap \overline{L''}$ is connected.  
Thus for any fixed $x,y \in \overline{C_\alpha} \cap \overline{L''}$, there exists a path $p$ from $x$ to $y$ contained in $\overline{C_\alpha} \cap \overline{L''}$.  

Without loss of generality, we can assume that the frontier of $\overline{L'}$ (which equals the frontier of $\overline{L''}$) meets any edge of $\mathscr{C}^1(G) = (\mathscr{C}^2)^{(1)} = ((\mathscr{C}^2)'')^{(1)}$ in only finitely many points.
The group $G$ has one end and $\mathscr{C}^2(G)$ is simply connected, thus each edge of $\mathscr{C}^2(G)$ is contained in a 2-cell.  The same is true for $(\mathscr{C}^2)''$, thus we can take $p$ to be transverse to $((\mathscr{C}^2)'')^{(1)}$, with $|\{ p \cap ((\mathscr{C}^2)'')^{(1)}\} |$ finite and $p$ still contained in $\overline{C_\alpha} \cap \overline{L''}$.  
Let $z_0, z_1, \ldots , z_{k}$ denote the elements of $p \cap ((\mathscr{C}^2)'')^{(1)}$, numbered in the order in which they are traversed by $p$, traveling from $x$ to $y$, with $z_0=x, z_k=y$.

Recall that $G$ is finitely presented, so the 2-cells of $\mathscr{C}^2(G)$ are of bounded perimeter.
Also the additional 2-cells added to create $(\mathscr{C}^2)''$ have diameter bounded by a function of $\phi, N$ and $\epsilon$, so there is a bound $n_0 = n_0(G, \phi, N, \epsilon)$ on the perimeters of the 2-cells of $(\mathscr{C}^2)''$.
Any component of $p-(\mathscr{C}^2)^{(1)}$ must be contained in such a cell, thus its interior can be replaced by a segment in $(\mathscr{C}^2)^{(1)}$ 
with length less than $n_0$.  

Recall that $p$ was originally contained in the frontier of $\overline{L'}$, $\overline{L'} \subset N_\epsilon (L)$ with $\epsilon \ll 1$, and both $\overline{L'}$ and $\overline{L'} \cap \mathscr{C}^1(G)$ deformation retract onto $L$.
Hence we can replace each segment of $p-(\mathscr{C}^2)^{(1)}$ with a segment in $(\mathscr{C}^2)^{(1)}$ of length less than $n_0$ that is contained in $L$ except for initial and terminal segments of length bounded by $\epsilon$.

Thus the $z_i$'s form an $(\overline{L''}, n_0)$-chain from $x$ to $y$ as desired.
\end{proof}


We can now prove that quasi-lines in finitely presented groups satisfy $ess(m_0)$.

\begin{lem}\label{compl_has_fin_many_comps} 
Let $G$ be a one-ended finitely presented group, with $L$ a $(\phi,N)$ quasi-line in $\mathscr{C}^1(G)$.  Then $\mathscr{C}^1(G)-L$ contains only finitely many essential components.  Moreover, there is some $m_0=m_0(G, \phi , N)$ such that $L$ satisfies $ess(m_0)$.
\end{lem}

\begin{proof}
We shall prove that $L$ satisfies $ess(m_0)$, for some $m_0>0$.  Since $\mathscr{C}^1(G)$ is locally finite, it will follow that the complement of $L$ contains finitely many essential components.

Let $C$ be a component of the complement of $L$ in $\mathscr{C}^1(G)$.  We shall use Lemma \ref{CC_lem} to show that there is some $n>0$ (not depending on our choice of $C$) such that any $x,y \in (\overline{C} \cap L)$  are connected by an $(L,n)$-chain in $(\overline{C} \cap L) \subset \mathscr{C}^1(G) \subset \mathscr{C}^2(G)$.  

Let $L'$ be as defined in Lemma \ref{CC_lem}, so, for some $0<\epsilon \ll 1$, $L \subset L' \subset N_\epsilon (L)$, $L'$ is open in $\mathscr{C}^2(G)$, and both $\overline{L'}$ and $\overline{L'} \cap \mathscr{C}^1(G)$ deformation retract onto $L$.
Recall that in Lemma \ref{CC_lem}, we proved a result similar to that desired now, but for $\overline{L'}\subset \mathscr{C}^2(G)$.

Fix any such $x$ and $y$, and, as $\overline{C}$ is connected, there is a simple oriented edge path $p$ in $\overline{C}$ connecting them.  
Note that our assumptions on $L'$ imply that each edge in $p$ must meet one component of $\mathscr{C}^2(G)-L'$.
Thus $p$ is a union of edge paths $p_1, p_2, \ldots , p_k$ such that, for each $i$, the terminal vertex of $p_i$ is equal to the initial vertex of $p_{i+1}$, and each $p_i$ intersects $L'$ in components of length no more than $\epsilon$ containing its initial and terminal vertices, with the rest of $p_i$ contained in some component $C'$ of $\mathscr{C}^2(G)-L'$.
Let $x_i$ and $y_i$ denote the two points of $(C' \cap \overline{L'})$, so each is within $\epsilon$ of a different endpoint of $p_i$.

By Lemma \ref{CC_lem}, each pair $x_i$ and $y_i$ can be connected by an $(\overline{L'}, n_0)$-chain $\{ z_j'\}$ in $C' \cap \overline{L'}$.  Recall that, moreover, a path of length no more than $n_0$ between any two consecutive points in the chain is in $L$, outside of initial and final segments of length no more than $\epsilon$, and that $\overline{L'} \cap \mathscr{C}^1(G)$ deformation retracts onto $L$.  
Thus, each $z_j'$ can be connected by a path of length no more than $\epsilon$ in $\overline{L'} \cap \mathscr{C}^1(G)$ to a point $z_j \in \overline{C} \cap L$ such that $\{ z_j\}$ forms an $(L, n_0 + 2\epsilon)$-chain in $\overline{C}\cap L$, connecting the endpoints of $p_i$.  Concatenating these chains, we see that, for $n=n_0+2\epsilon$, $x$ and $y$ can be connected by an $(L,n)$-chain in $\overline{C}\cap L$ as desired.

From now on, we shall work only in $\mathscr{C}^1(G)$, not $\mathscr{C}^2(G)$.

We shall now find an $m_0>0$ such that there is an $(L,n)$-chain in the frontier of each essential component $C$ of the complement of $L$ that must intersect the $m_0$-ball about any given point of $L$.

Fix any $a \in L$ and $R\gg N, n$.  As $C$ is essential, $e(C \cup L)=1$, and, from Lemma \ref{e(L)=2}, recall that $L$ must have two ends.  It follows then that $C$ must intersect both unbounded components of $L-B_R(a)$;
let $x$ be in the intersection of $\overline{C}\cap L$ with one, and $y$ in the intersection of $\overline{C}\cap L$ with the other.   
By the work above, there exists an $(L,n)$-chain, $\{ z_i\}$, from $x$ to $y$ in $\overline{C}\cap L$.

Recall that $L$ is a $(\phi,N)$ quasi-line, and let $l$ be the line associated to $L$.  Then, for each $i$, there is a path in $L$ of length less than or equal to $N$ connecting $z_i$ to some $w_i \in l$.  For each $i$, $d(z_i, z_{i+1})\leq n$, thus $d(w_i,w_{i+1})\leq n+2N$, and thus the path in $l$ between any two adjacent $w_i$'s has length less than or equal to $\phi (n+2N)$.

Let $a_0 \in l$ be of distance less than or equal to $N$ from $a \in L$.  As $R\gg 0$, $x$ and $y$ are such that there is some $i$ with $l^{-1}(w_i) \leq l^{-1}(a_0) \leq l^{-1}(w_{i\pm 1})$, and hence, for some $j$, $d(a_0, w_j)\leq \frac{1}{2}\phi (n+2N)$.  Thus
$$d(a,z_j) \leq d(a, a_0) + d(a_0, w_j) + d(w_j, z_j) \leq \frac{1}{2}\phi (n+2N)+2N.$$
Since $z_j \in C$, and $z_j$ is of distance less than 1 from a vertex of $C$, it follows that, for any $m_0\geq [\frac{1}{2}\phi (n+2N)+2N+1]$, $C$ intersects $B_{m_0}(a)$ in a vertex.  Thus $L$ satisfies $ess(m_0)$.
\end{proof}


We note that, in particular, the argument above  proves the following:

\begin{cor}\label{from_ess_proof}
Let $G$ be a one-ended finitely presented group, with $L$ a $(\phi, N)$ quasi-line in $\mathscr{C}^1(G)$ and $C$ a component of $\mathscr{C}^1(G)-L$, which need not be essential.  Let $m_0 = m_0(G, \phi, N)$ be as in Lemma \ref{compl_has_fin_many_comps}.

If $K\subset L$ is such that $K$ separates $L$ into two infinite components and $C$ meets both of those components, then $B_{m_0}(x)$ meets $C$ in a vertex, for each $x \in K$.
\end{cor}

\begin{rmk}\label{all_qlines_sat_ess_w_same_m_0}
By Lemma \ref{compl_has_fin_many_comps}, and since we saw in Remark \ref{qlines_have_same_consts} that all quasi-lines with which we are concerned in any one Cayley graph will have the same parameters, they will all satisfy $ess(m_0)$ for some fixed $m_0$.
\end{rmk}



In the remainder of this section, we will show that any quasi-line $L$ that we are concerned with satisfies $iness(m_1)$ for some $m_1$.

\begin{lem}\label{NRH_sats_iness}
Let $G$ be a one-ended finitely generated group, and let $H$ be a two-ended subgroup of $G$.  Then any neighborhood of $H$ that is a quasi-line satisfies $iness(m_1)$, for some $m_1$ depending only on $H$ and the size of the neighborhood.
\end{lem}

\begin{proof}
Fix $R>0$ such that $N_R(H)$ is connected, hence is a quasi-line. 
It will suffice to find some number $m_1(R)$ such that each inessential component of the complement of $N_R(H)$  is contained in the $m_1(R)$-neighborhood of $N_R(H)$.

We will prove this by first showing that a component $C$ of $\mathscr{C}^1(G)-N_R(H)$ is essential if and only if $C$ is nearly essential.
As $N_R(H)$ has two ends, it follows that $C$ is nearly essential if $C$ is essential.  We now prove the converse.

Suppose that $C$ is not essential.
Then $C \cup N_{R}(H)$ has more than one end, so there is a compact $K \subset (C \cup N_{R}(H))$ such that $(C \cup N_{R}(H))-K$ has more than one infinite component.  

Let $m$ denote the number of infinite components of $(C \cup N_{R}(H))-K$, and suppose that $m>2$.
Since $e(G)=1$, each of these components must meet $N_R(H)$, and
as $e(N_R(H))=2$, the intersection of $N_R(H)$ with at least $(m-2)$ of these components must be finite.  
Let $M$ be the union of $K$ with these finite regions of $N_R(H)$, and thus at least $(m-2)$ components of the complement of $M$ in $(C\cup N_R(H))$ do not intersect $N_R(H)$.  
Moreover, note that each such component, of which there is at least one, must be a component of $\mathscr{C}^1(G)-M$.  
At least one other component of $\mathscr{C}^1(G)-M$ has infinite intersection with $N_R(H)$, and hence $e(G) > 1$, a contradiction.

Thus we must have that $e(C \cup N_R(H))=2$.  We shall show that we can find a finite index subgroup of $H$ that fixes $C$.

Note that since $C$ is a component of the complement of $N_R(H)$, then, for any $g \in H$, $gC$ is also a component of the complement of $N_R(H)$.  

Let $\langle h \rangle$ be a finite index subgroup of $H$, and suppose that the $\langle h \rangle$-orbit of $C$ contains infinitely many components of the complement of $N_R(H)$.

Suppose, in addition, that $C$ does not meet $N_R(H)$ along its entire length, i.e. that there is some compact region $K' \subset N_R(H)$ and an infinite component $L_+$ of $N_R(H)-K'$ such that $C$ does not meet $L_+$.  
As $e(\mathscr{C}^1(G)) = 1$, the intersection of $C$ with $N_R(H)$ must be infinite, so $N_R(H)-K'$ must have another infinite component, call it $L_-$, and $C$ must meet $L_-$.  
Moreover, for any point $q \in N_R(H)$ and any $r>0$, $C$ must meet $L_-$ outside of $B_r(q)$.

Let $\phi$ and $N$ be parameters for $N_R(H)$ and let $m_0=m_0(G, \phi, N)$ be as in Lemma \ref{compl_has_fin_many_comps}.
Then, by Corollary \ref{from_ess_proof}, for any point $p \in L_-$ that is sufficiently far from $K'$, $C$ must meet $B_{m_0}(p)$ in a vertex.  Fix such a point $p$.

As we have assumed that $\langle h \rangle \cdot C$ consists of infinitely many components, choose $\{ n_i\}$ such that $\{ h^{n_i} C\}$ are distinct.  We can moreover choose the $\{ n_i\}$ such that $L_- \subset h^{n_i} L_-$, for all $i$.

But then each $h^{n_i} C$ must meet $B_{m_0}(p)$ in a vertex.  $\mathscr{C}^1(G)$ is finitely generated, hence there are only finitely many vertices in $B_{m_0}(p)$, but the translates $h^{n_i}\cdot C$ are disjoint, thus we have reached a contradiction.  

If instead, for any compact subset $K'$ of $N_R(H)$, $C$ meets both infinite components of $N_R(H)-K'$, then a similar argument, with $L$ taking the role of $L_-$, also gives a contradiction.  
Thus the $\langle h \rangle$-orbit of $C$ must be a finite collection of components.  

By passing to a finite index subgroup of $\langle h \rangle$ if necessary, we can assume that $\langle h \rangle$ fixes $C$.

Recall that we showed above that $(C\cup N_R(H))$ has two ends.  The subgroup $\langle h \rangle$ acts on this union by isometries, so the quotient of $(C \cup N_R(H))$ by this action must be compact.
It follows that $C$ is contained in a finite neighborhood of $N_R(H)$, hence is not nearly essential.

Thus $C$ is essential if and only $C$ is nearly essential.

It remains to argue that each inessential component of the complement of $N_R(H)$ is contained in the $m_1(R)$-neighborhood of $N_R(H)$ for some $m_1(R)$.
Note that any inessential component of $\mathscr{C}^1(G)-N_R(H)$ is not nearly essential, hence projects onto a bounded component of $H\backslash \mathscr{C}^1(G)- H\backslash N_R(H)$.
As $H \backslash N_R(H)$ is compact and $H \backslash \mathscr{C}^1(G)$ is locally finite, there are only finitely many components of $H \backslash \mathscr{C}^1(G) - H \backslash N_R(H)$, thus there is some $m_1(R)>0$ such that each bounded component is contained in the $m_1(R)$-neighborhood of $H \backslash N_R(H)$.  

It follows that any inessential component of the complement of $N_R(H)$ is contained in the $m_1(R)$-neighborhood of $N_R(H)$, as desired.
\end{proof}

We show next that the property of satisfying $iness(m_1)$ for some function $m_1$ is invariant under quasi-isometries.

\begin{lem}\label{iness_preserved_under_qi}
Let $f \co \mathscr{C}^1(G) \to \mathscr{C}^1(G')$ be a $(\Lambda, K)$ quasi-isometry between the Cayley graphs of one-ended, finitely presented groups $G$ and $G'$, and let $L\subset \mathscr{C}^1(G)$ be a quasi-line satisfying $iness(m_1)$.  

If $R'\geq 0$ is such that $L'=N_{R'}(f(L))$ is a quasi-line in $\mathscr{C}^1(G')$, then $L'$ must satisfy $iness (m_1')$, for some $m_1'$ depending on $\Lambda, K,$ $m_1$, and $R'$.
\end{lem}

\begin{proof}
Recalling Lemma \ref{imgs_of_qlines_are_qlines}, we fix $R'$ so that $L' = N_{R'}(f(L))$ is a quasi-line.
As was the case previously, it suffices to prove that there is some number $m_1'(R')>0$ (dependent on $R'$) such that the inessential components of the complement of $L'$ are contained in the $m_1'(R')$-neighborhood of $L'$.

Again we will begin by showing that any component of $\mathscr{C}^1(G')-L'$ is essential if and only if it is nearly essential.
Recall that we always have that essential implies nearly essential.

Let $f^{-1}$ be a quasi-inverse to $f$, and note that, for any $R>0$, each component of $\mathscr{C}^1(G')-L'$ gets mapped by $f^{-1}$ either into $N_R(L)$ or into the union of $N_R(L)$ with components of its complement.  
We claim that we may choose $R$ large enough that, if $C'$ is a component of $\mathscr{C}^1(G')-L'$ such that $f^{-1}(C')$ meets a component $C$ of $\mathscr{C}^1(G)-N_R(L)$, then the image under $f^{-1}$ of no other component of $\mathscr{C}^1(G')-L'$ will meet $C$.

To see this, let $\Lambda', K', \delta$ be such that $f^{-1}$ is a $(\Lambda', K')$ quasi-isometry, with $f^{-1}(L') \subset N_\delta (L)$.  
Let $\{ C_\alpha'\}$ be the components of $\mathscr{C}^1(G')-L'$, and let $R_1> \Lambda' K'$.  Note that, if $\alpha \neq \beta$, and $C_\alpha' -N_{R_1}(L')$ and $C_\beta' -N_{R_1}(L')$ are nonempty, then any points $p_\alpha \in C_\alpha' -N_{R_1}(L')$, $p_\beta \in C_\beta' -N_{R_1}(L')$ are at least a distance of $2\Lambda' K'$ apart.

Let $R> (\delta + \Lambda' R_1 + K')$, and
note that $f^{-1}(N_{R_1}(L')) \subset N_R(L)$.
Recall that $f^{-1}$ is coarsely surjective, with $N_{K'}(f^{-1}(\mathscr{C}^1(G'))) = \mathscr{C}^1(G)$.
Suppose that there is some component of $\mathscr{C}^1(G)-N_R(L)$ that is met by more than one image $f^{-1}(C_\alpha')$.  
Then there are two such, call them $f^{-1}(C_\alpha')$ and $f^{-1}(C_\beta')$, with some $p_\alpha \in C_\alpha' - N_{R_1}(L'), p_\beta \in C_\beta' - N_{R_1}(L')$, such that $d(f^{-1}(p_\alpha),f^{-1}( p_\beta))<K'$.  
But this means that 
$\frac{1}{\Lambda'}d(p_\alpha, p_\beta)-K' < K'$, i.e. that
$d(p_\alpha, p_\beta)<2\Lambda' K'$, which is a contradiction.

Thus, with $R$ chosen as above, we have that the images under $f^{-1}$ of different components of the complement of $L'$ shall not meet the same component of the complement of $N_R(L)$.  

Suppose now that $C'$ is a component of $\mathscr{C}^1(G')-L'$ that is nearly essential.  Let $C_0$ be the union of the components of $\mathscr{C}^1(G)-N_R(L)$ that are met by $f^{-1}(C')$.  
Since $L$ satisfies $iness (m_1)$, $C_0$ must contain an essential component of the complement of $N_R(L)$.
Let $C_0^e$ denote the essential components in the complement of $N_R(L)$ that are met by $f^{-1}(C')$, and now we have that $C_0^e$ is nonempty.

Observe that $(C' \cup L')$ is quasi-isometric to $C_0 \cup f^{-1}(L')$, which is quasi-isometric to $(C_0 \cup L)$.  Certainly this is quasi-isometric to $(C_0 \cup N_R(L))$, which in turn must be quasi-isometric to $(C_0^e \cup N_{R}(L))$, since the components of $C_0$ that are inessential must be contained in the $m_1(R)$-neighborhood of $N_R(L)$.  

We claim that $e(C_0^e \cup N_{R}(L))=1$.
This is immediate if $C_0^e$ contains only one component, so assume that $C_0^e = \{ C_i\}$ contains more than one, and suppose for a contradiction that $e(C_0^e \cup N_{R}(L)) >1$.

Then there is some finite subgraph $K$ of $(C_0^e \cup N_{R}(L))$ whose complement has more than one infinite component.
For each $i$, $C_i$ is essential so $(C_i \cup N_R(L))-K$ has exactly one infinite component which we shall call $D_i$, and as $\mathscr{C}^1(G)$ is locally finite, we note that $N_R(L) - (D_i \cap N_R(L))$ must therefore be finite.
On the other hand, as $G$ is finitely presented, so $N_R(L)$ satisfies $ess(m_0)$ for some $m_0$. 
Thus there are only finitely many $C_i$'s, so there must be indices $i$ and $j$ such that $D_i$ and $D_j$ are disconnected in $(C_0^e \cup N_{R}(L))$ by $K$.
As $N_R(L)$ is finite outside of $D_i$ and is finite outside of $D_j$, we have reached a contradiction.

Thus $e(C_0^e \cup N_{R}(L))=1$, so $e(C' \cup L') =1$ and $C'$ is essential.

Hence components of the complement of $L'$ are essential if and only if they are nearly essential.
It remains to conclude that $L'$ satisfies $iness(m_1)$.  

As $f^{-1}$ is coarsely surjective, we have that a finite neighborhood of the image under $f^{-1}$ of any component of the complement of $L'$ is equal to a subset of $N_R(L)$, together with a collection of components of the complement of $N_R(L)$.
As the inessential components of the complement of $N_R(L)$ are contained in the $m_1(R)$-neighborhood of $N_R(L)$, it follows that there is some $m_1'(R')>0$ such that any component $C'$ of the complement of $L'$ is either contained in the $m_1'(R')$-neighborhood of $L'$, or is contained in no finite neighborhood of $L'$.
Thus, $C'$ must be contained in $N_{m_1'(R')}(L')$ or else is nearly essential hence essential, as desired.

\end{proof}

\begin{rmk}\label{iness_preserved}
Recall that we are concerned with the following quasi-lines.
If $G$ is a one-ended finitely presented group with a two-ended subgroup $H$, then we will consider a quasi-line in $\mathscr{C}^1(G)$ of the form $N_R(H)$ and its translates under the action of $G$.
We will also consider a quasi-line that is the $R'$-neighborhood of the image of $H$ under a quasi-isometry $f \co \mathscr{C}^1(G) \to \mathscr{C}^1(G')$, and the translates of that quasi-line in $\mathscr{C}^1(G')$.

$G$ acts on $\mathscr{C}^1(G)$ by isometries, and hence it follows from Lemma \ref{NRH_sats_iness} that any collection of quasi-lines that are translates of $N_R(H)$ by elements of $G$ will all satisfy $iness(m_1)$ for the same function $m_1$.

Simiarly, and using Lemma \ref{iness_preserved_under_qi}, we have that any translates of the $R'$-neighborhood of $f(H)$ will all satisfy $iness(m_1')$ for the same function $m_1'$.
\end{rmk}


\section{The quasi-isometry invariance of two-ended subgroups with at least three coends} \label{section3}

In this section, we will prove that, up to finite Hausdorff distances, quasi-isometries take two-ended subgroups with at least three coends to other two-ended subgroups with at least three coends (Theorem \ref{existence_qi_inv}).
Coends will be defined below, and we will see that two-ended subgroups having at least three coends will be exactly those whose corresponding quasi-lines are 3-parting (Lemma \ref{coends_and_ess_comps}).

The main ingredient in the proof of the quasi-isometry invariance of two-ended subgroups with at least three coends is Proposition \ref{LnearZ}, which shows that any $3$-parting quasi-line satisfying $iness(m_1)$ is a finite Hausdorff distance from an infinite cyclic subgroup.
For this, we use the proof of a similar result from \cite{papa_qlines}.
There, Papasoglu shows that, given a 3-parting quasi-line in the Cayley graph of a one-ended, finitely presented group, either the quasi-line is a finite distance from an infinite cyclic subgroup, or a related limit of translates of quasi-lines is.
We will see below that this latter possibility can be eliminated.
\\

We first note that quasi-lines satisfying $iness(m_1)$ but that are not 3-parting need not be a finite Hausdorff distance from a copy of $\mathbb{Z}$.
For example, consider the nearest-point projection of a line $l_0$ in $\mathbb{R}^2$ with irrational slope into the Cayley graph of $\mathbb{Z}^2$, where the vertices are taken to be the integer lattice points in $\mathbb{R}^2$ and we take the standard generating set.  
Let $L$ denote a connected neighborhood in $\mathscr{C}_{\mathbb{Z}^2}$ of the projection of $l_0$.  Then $L$ is a 2-parting quasi-line in $\mathscr{C}_{\mathbb{Z}^2}$ that satisfies $iness(m_1)$ for some $m_1$.  
However, the infinite cyclic subgroups of $\mathbb{Z}^2$ correspond to lines in $\mathbb{R}^2$ with rational slope, hence $L$ is an infinite Hausdorff distance from any subgroup of $\mathbb{Z}^2$.
\\


In order to prove Proposition \ref{LnearZ}, we will need to know that 3-parting quasi-lines do not cross one another in an essential way.  
We shall say that $a,b \in \mathscr{C}^1(G)$ are {\it $K$-separated by a quasi-line $L$} if $B_K(a)$ and $B_K(b)$ are in different components of the complement of $L$.
The following is Proposition 2.1 from \cite{papa_qlines}:

\begin{thm} {\rm \cite{papa_qlines}} \    
\label{2.1}
Let $G$ be a one-ended, finitely presented group, and let $L, L_1$ be $(\phi',N')$ quasi-lines in $\mathscr{C}^1(G)$ that satisfy $iness(m_1)$.
Suppose that $L$ is 3-parting.

Then there is some $K = K(G, \phi', N', m_1)$ such that no two points $a,b \in L$ are $K$-separated by $L_1$.
\end{thm}

We include a proof of this theorem in the appendix, in order to clarify some points from the proof given in \cite{papa_qlines}.

To restate Theorem \ref{2.1}, if there is some $a \in L$ that is in an essential component $C$ of the complement of $L_1$ and is more than a distance of $K$ from $L_1$, and $b \in L$ is in a different essential component of the complement of $L_1$, then $b$ is no more than a distance of $K$ from $L_1$.  

It follows that $L$ is contained in the $K$-neighborhood of $L_1 \cup C$.
Let $m_0$ be such that $L$ and $L_1$ satisfy $ess(m_0)$ (see Lemma \ref{compl_has_fin_many_comps}), and let $K' = K + 2m_0$.
Then it follows that $L$ is contained in the $K'$-neighborhood of $C$.

Thus we have the following corollary to Theorem \ref{2.1}:

\begin{cor} \label{q_lines_are_K_disjoint}
Let $G$, $L$ and $L_1$ be as in Theorem \ref{2.1}.

  Then there is some $K'=K'(G, \phi , N, m_1)$ such that $L$ is contained in the $K'$-neighborhood of an essential component of the complement of $L_1$.
 \end{cor}


The next observation will be needed in Lemma \ref{seq_of_qlines_is_fin}.

\begin{lem}\label{get_haus}
Let $L$ and $L'$ be $(\phi, N)$ quasi-lines in a metric space $X$.
Then for any $x_2$, there is some $x_1 = x_1(\phi, N, x_2)>x_2$ such that if $L \nsubseteq N_{x_1}(L')$, then $L' \nsubseteq N_{x_2}(L)$.
\end{lem}

\begin{proof}
Given $x_2$, let $x_1 > \frac{1}{2}\phi(2(x_2+N)) + 2N + x_2$.
Suppose for a contradiction that $L \nsubseteq N_{x_1}(L')$ and that $L' \subset N_{x_2}(L)$.
Let $l$ and $l'$ be the lines associated to $L$ and $L'$ respectively, and it follows that there is some $t \in \mathbb{R}$ such that the $(x_1-N)$-ball about $l(t)$ does not meet $l'$.

As $L' \subset N_{x_2}(L)$, hence $l' \subset N_{x_2+N}(l)$, it follows that there are $t_1 < t < t_2$ such that $|t-t_i| \geq ((x_1 - N)-(x_2+N))$ for each $i$ and $d(l(t_1), l(t_2)) \leq 2(x_2+N)$.
But our assumption on $x_1$ implies that $\phi (2(x_2+N)) < 2(x_1-2N-x_2)$, a contradiction.
Thus $L' \nsubseteq N_{x_2}(L)$.
\end{proof}

We shall need the following lemma, both to prove Proposition \ref{LnearZ} and also to prove another later result.


\begin{lem}\label{seq_of_qlines_is_fin}
Let $G$ be a one-ended, finitely presented group, and let $\{ L^i\}$ be a collection of 3-parting $(\phi, N)$ quasi-lines in $\mathscr{C}^1(G)$ satisfying $iness(m_1)$.  Suppose that $\cap_i L^i$ contains a vertex.  

Then there is some constant $x_1 = x_1(G, \phi, N, m_1)$ such that if, for all $i,j$, $d_{Haus}(L^i, L^j) > x_1$,
then $\{ L^i \}$ is finite.
\end{lem}

\begin{proof}
Let $m_0 = m_0(G, \phi, N)$ be as in Lemma \ref{compl_has_fin_many_comps}, so that each $L^i$ satisfies $ess(m_0)$.
Let $K' = K'(G, \phi , N, m_1)$ be as in Corollary \ref{q_lines_are_K_disjoint}, so that, for each $i, j$, $L^i$ is contained in the $K'$-neighborhood of an essential component of the complement of $L^j$.
Furthermore, let $m_0' = m_0(G, \phi, N+K')$, so that, for any $i$, $N_{K'}(L^i)$ (which is a $(\phi, N+K')$ quasi-line) satisfies $ess(m_0')$.
Let $x_2 > \max \{ K', m_0'\}$ and let $x_1$ be from Lemma \ref{get_haus}.
Thus $d_{Haus}(L^i, L^j) > x_1$ implies that $L^i \nsubseteq N_{x_2}(L^j)$ and $L^j \nsubseteq N_{x_2}(L^i)$.

Let $\mathscr{L}_0$ denote $\{ L^i\}$, and suppose that $\mathscr{L}_0$ is infinite.
Then choose any element $L_0$ from $\mathscr{L}_0$.  
As $L_0$ satisfies $ess(m_0)$, the complement of $L_0$ has only finitely many essential components, so there is some essential component $B_0$ whose $K'$-neighborhood contains infinitely many elements of $\mathscr{L}_0$.  
Let $\mathscr{L}_1 = \{L \in [\mathscr{L}_0-L_0] : L \subset N_{K'}(B_0)\}$.  Choose $L_1$ from $\mathscr{L}_1$, and let $B_1'$ be the essential component of the complement of $L_1$ whose $K'$-neighborhood contains $L_0$.  Note that $x_2>K'$ implies that $B_1'$ is unique.

As $\mathscr{L}_1$ is infinite, there is some essential component of the complement of $L_1$ whose $K'$-neighborhood contains infinitely many elements of $\mathscr{L}_1$.  Let $B_1$ denote this component, and let $\mathscr{L}_2$ denote $\{ L \in [\mathscr{L}_1-\{L_0, L_1\} ] : L \subset N_{K'}(B_1)\}$.  Choose $L_2$ from $\mathscr{L}_2$, and continue on in this manner.  This produces an infinite sequence of quasi-lines $\{ L_i\}$ and subsets of $\mathscr{C}^1(G)$, $\{ B_i\}$ and $\{B_i'\}$, such that, for each $i$, $B_i$ is an essential component of the complement of $L_i$ such that $L_j \subset N_{K'}(B_i)$ for all $j>i$, and $B_i'$ is an essential component of the complement of $L_i$ such that $L_j \subset N_{K'}(B_i')$ for all $j<i$ (with perhaps $B_i  = B_i'$).  
Each $L_i$ is 3-parting, so we may set $D_i$ to be an essential component of the complement of $L_i$ that is not equal to $B_i$ nor $B_i'$, for each $i$.

We shall see next that the $D_i$'s are basically disjoint.  
Let $i \neq j$, and note that, since $L_i$ is not contained in the $x_2$-neighborhood of $L_j$,
 there must be some point $p \in L_i$ such that $B_{x_2}(p)$ does not intersect $L_j$.  Thus $B_{x_2}(p)$ is contained in $B_j$ or $B_j'$.  
 
 Note that, for each $i$, $D_i - N_{K'}(L_i)$ is a collection of essential and inessential components of the complement of $N_{K'}(L_i)$.
Since $D_i$ is an essential component of the complement of $L_i$, and $L_i$ satisfies $iness(m_1)$, it follows that $D_i-N_{K'}(L_i)$ must contain an essential component $E_i$ of the complement of $N_{K'}(L_i)$.

As $x_2> m_0'$, $B_{x_2}(p)$ must meet each essential component of the complement of $N_{K'}(L_i)$, so, in particular, $B_{x_2}(p)$ meets $E_i$, hence $B_{x_2}(p) \cup E_i$ is connected.  

The quasi-line $L_j$ is disjoint from $D_i-N_{K'}(L_i)$, hence does not meet $E_i$, or the union $B_{x_2}(p) \cup E_i$.  It follows that this union is contained in $B_j$ or $B_j'$, so is disjoint from $D_j$, and hence from $E_j \subset D_j$.  Thus, the $E_i$'s are disjoint.

Now we recall that $\cap_i L_i$ contains a vertex, say $y \in \mathscr{C}^1(G)$, and hence $B_{m_0'}(y)$ intersects each $E_i$.  Since these regions are disjoint, $B_{m_0'}(y)$ must contain a collection of vertices in bijection with $\{ L_i\}$.  But $G$ is finitely generated, hence $B_{m_0'}(y)$ has only finitely many vertices, and we have reached a contradiction.  
\end{proof}


\begin{prop}\label{LnearZ}
Let $L$ be a $3$-parting $(\phi, N)$ quasi-line in the Cayley graph of a one-ended, finitely presented group $G$, and suppose that $L$ satisfies $iness(m_1)$ for some $m_1$.  Then there is some subgroup $H \cong \mathbb{Z}$ of $G$ such that $d_{Haus}(L,H)< \infty$.
 \end{prop}

\begin{proof} 

Let $L$ be as in the statement of the proposition, and let $x_1$ be as in Lemma \ref{seq_of_qlines_is_fin}, defined with the parameters of $L$.
In case 1 of section 6 of \cite{papa_qlines}, Papasoglu makes the following construction.

Fix some $y \in L$, and choose a sequence $\{y_i \} \subset L$ such that $d(y,y_i) \to \infty$.  Let $g_i$ be such that $g_iy_i = y$, and, by passing to a subsequence, we may assume that, for all $i>j$, 
$$g_jL \cap B_j(y) = g_iL \cap B_j(y).$$
If there is some $i$ such that $d_{Haus} (g_iL, g_jL) $ is less than or equal to any fixed constant for infinitely many $g_j$, then it is shown in \cite{papa_qlines} that 
there is some $g$ contained in the subgroup generated by these $g_j$ such that $\langle g \rangle \cong \mathbb{Z}$, and $g_iL$ is a finite Hausdorff distance from $\langle g \rangle$.
Thus $d_{Haus}(L, g_i^{-1}\langle g \rangle) < \infty$.  Since $d_{Haus}(g_i^{-1}\langle g \rangle, g_i^{-1}\langle g \rangle g_i)$ is bounded by the word length of $g_i$, it follows that $L$ is a finite Hausdorff distance from $g_i^{-1}\langle g \rangle g_i \cong \mathbb{Z}$.

So, by passing to a subsequence, we may assume that, for each $i$ and $j$, $d_{Haus}(g_iL, g_jL)>x_1$.  
It follows that this infinite subsequence of $\{ g_iL\}$ satisfies the hypotheses of Lemma \ref{seq_of_qlines_is_fin}, which is a contradiction.
\end{proof}


Next, we will give the definition of coends.  We will see that Theorem \ref{existence_qi_inv} follows quickly from Proposition \ref{LnearZ} and a few basic facts about coends.

Given a group $G$ with a subgroup $H$ and a subset $Y$, we say that $Y$ is {\em $H$-finite} if $Y$ is contained in finitely many cosets $Hg$ of $H$.
In \cite{KropRol}, Kropholler and Roller defined 
$$\tilde{e}(G,H) = {\rm dim}_{\mathbb{F}_2} (\mathcal{P}G / \mathcal{F}_HG)^G,$$
where $\mathcal{P}G$ is the power set of all subsets of $G$, and $\mathcal{F}_HG$ is the set of all $H$-finite subsets of $G$.  
The quotient set $\mathcal{P}G/\mathcal{F}_HG$ forms a vector space over $\mathbb{F}_2$, the field with two elements, under the operation of symmetric difference.  Thus a subset $X$ of $G$ represents an element of $(\mathcal{P}G / \mathcal{F}_HG)^G$ if and only if the symmetric difference $X+Xg$ is $H$-finite for all $g \in G$.  

Following Bowditch \cite{BowJSJ}, we shall call $\tilde{e}(G,H)$ the number of {\em coends} of $H$ in $G$.  
(Kropholler and Roller called $\tilde{e}(G,H)$ the number of relative ends of $H$ in $G$, and we note that this is also sometimes referred to as the number of filtered ends of $H$ in $G$.)

If $X$ is a subset of $G$, then we can think of $X$ as a subset of the vertex set of $\mathscr{C}^1(G)$, and thus $\delta X$, the coboundary of $X$, is the set of edges in $\mathscr{C}^1(G)$ that have exactly one vertex contained in $X$.
It is a fact that $X$ represents an element of $(\mathcal{P}G / \mathcal{F}_HG)^G$ exactly when $\delta X$ is $H$-finite.
(See Cohen \cite{Cohen_cohom_dim_one} for a proof of this in the case when $H$ is trivial.)

The following lemma shows that we can characterize the number of coends of a two-ended subgroup in terms of essential components:

\begin{lem}\label{coends_and_ess_comps}
Let $G$ be a one-ended, finitely generated group with two-ended  subgroup $H$, and let $n<\infty$.  Then $\tilde{e}(G,H)\geq n$ if and only if there is some $R>0$ such that $N_R(H)$ is a quasi-line in $\mathscr{C}^1(G)$ that is $n$-parting.

Moreover, $\tilde{e}(G,H)=\infty$ if and only if, for each $n< \infty$ there is some $R = R(n)$ such that $N_R(H)$ is an $n$-separating quasi-line.
\end{lem}

\begin{proof}
A subset $X$ of $G$ represents an element in the $\mathbb{F}_2$-vector space $(\mathcal{P}G/\mathcal{F}_HG)^G$ if and only if $\delta X$ is an $H$-finite set of edges in $\mathscr{C}^1(G)$.
Note that this happens precisely when $\delta X$ is contained in a finite neighborhood of $H$ in $\mathscr{C}^1(G)$.

Essential components of the complement of any quasi-line of the form $N_R(H)$ naturally correspond to elements of $(\mathcal{P}G/\mathcal{F}_HG)^G$:  let $\hat{Y}$ be an essential component of the complement of $N_R(H)$, and let $Y$ denote the vertex set of $\hat{Y}$.  Then for any $\epsilon >0$, the boundary of $\hat{Y}$ is contained in $N_{R+\epsilon}(H)$, hence $\delta Y \subset N_{R+1}(H)$, thus $Y$ represents an element of $(\mathcal{P}G/\mathcal{F}_HG)^G$.
Note that $Y$ is not $H$-finite, so the element it represents must be nontrivial in $(\mathcal{P}G/\mathcal{F}_HG)^G$.

By Lemma \ref{2ended_subgps_are_qlines}, we can fix $R>0$ such that $N_R(H)$ is a quasi-line.
Suppose that $N_R(H)$ be $n$-parting, and let $Y_1, \ldots, Y_{n}$ be essential components of the complement of $N_R(H)$.  They are disjoint, hence represent independent elements of $(\mathcal{P}G/\mathcal{F}_HG)^G$, and thus $\tilde{e}(G,H) \geq n$.

If $\tilde{e}(G,H)\geq n$ for some $n<\infty$, then we can find representatives $X_1, \ldots , X_n$ of elements of a basis for $(\mathcal{P}G/\mathcal{F}_HG)^G$.  Thus there is some $R>0$ such that, in $\mathscr{C}^1(G)$, $\delta X_i \subset N_R(H)$, for all $i$.  Then note that each $X_i$ is equivalent in $(\mathcal{P}G/\mathcal{F}_HG)^G$ to a union of components of $\mathscr{C}^1(G)-N_R(H)$.  Recall from Lemma \ref{NRH_sats_iness} that, for some $m_1>0$, $N_R(H)$ satisfies $iness (m_1)$, hence each $X_i$ is equivalent to a union of essential components of $\mathscr{C}^1(G)-N_R(H)$.
Since the $X_i$'s are independent, $n$ of these essential components must be disjoint, so the complement of $N_R(H)$ has at least $n$ distinct essential components, i.e. $N_R(H)$ is $n$-parting.

Now suppose that $\tilde{e}(G,H)= \infty$, and fix any $n < \infty$.  Then in particular $\tilde{e}(G,H)\geq n$ so, by the previous paragraph, there is some $R = R(n)$ such that $N_R(H)$ is an $n$-parting quasi-line.

Lastly, suppose now that $H$ is such that, for any $n<\infty$ there exists some $R(n)$ such that $N_R(H)$ is an $n$-parting quasi-line, and let $Y_1^n, Y_2^n, \ldots , Y_n^n$ denote the essential components of the complement of $N_R(H)$.

Fix any sequence $n_1, n_2, n_3, \ldots $ such that $R(n_i) < R(n_{i+1})$ for all $i$.
Then we note that there are indices $j_i \neq k_i$ such that $1 \leq j_i, k_i \leq n_i$ and such that $Y_{j_i}^{n_i} \subset Y_{k_l}^{n_l}$ for all $i > l$, and hence $\{ Y_{j_i}^{n_i}\}_{i=1}^\infty$ are a disjoint collection of representatives of elements of $(\mathcal{P}G / \mathcal{F}_H(G))^G$.
It follows that $\tilde{e}(G,H)=\infty$.
\end{proof}

\begin{lem}\label{ess_comps_go_to_ess_comps}
Let $f\co \mathscr{C}^1(G) \to \mathscr{C}^1(G')$ be a $(\Lambda, C)$ quasi-isometry between the Cayley graphs of one-ended, finitely presented groups $G$ and $G'$, and let $L$ be a $(\phi, N)$ quasi-line in $\mathscr{C}^1(G)$ satisfying $iness(m_1)$.

Then there is some $R'=R'(\Lambda, C, \phi, N, m_1) >0$ such that, if $L$ is $n$-parting, then $N_{R'}(f(L))$ is also $n$-parting.
\end{lem}

\begin{proof}
Lemma \ref{imgs_of_qlines_are_qlines} shows that we can find some $R''>0$ such that $N_{R''}(f(L))$ is a quasi-line.  Thus so is $N_{R_0'}(f(L))$ for any $R_0' \geq R''$, and, by Lemma \ref{iness_preserved_under_qi}, we also have that $N_{R_0'}(f(L))$ satisfies $iness(m_1')$ for some $m_1'$ (depending on $R_0'$).
For any such $R_0'$,  the image under $f$ of any component of $\mathscr{C}^1(G)-L$ will be contained in the union of $N_{R_0'}(f(L))$ and components of its complement.
As in the proof of Lemma \ref{iness_preserved_under_qi}, there is some $R' \geq R''$ such that the images under $f$ of distinct components of $\mathscr{C}^1(G)-L$ do not meet the same components of $\mathscr{C}^1(G') - N_{R'}(f(L))$.  
Let $L' = N_{R'}(f(L))$.

As $f$ is coarsely surjective and $L'$ satisfies $iness(m_1')$, the image of any essential component in the complement of $L$ meets an essential component in the complement of $L'$.  As no two components of the complement of $L$ meet the same components of the complement of $L'$, it follows that the complement of $L'$ contains at least as many essential components as the complement of $L$. 
\end{proof}

We now can prove the following:

\begin{thm}\label{existence_qi_inv}
Let $f\co \mathscr{C}^1(G) \to \mathscr{C}^1(G')$ be a quasi-isometry between the Cayley graphs of one-ended, finitely presented groups $G$ and $G'$, and assume that $G$ contains a 2-ended subgroup $H$ that has $n$ coends in $G$, for $n \in \{ 3, 4, \ldots \} \cup \{ \infty\}$.
Then there is a two-ended subgroup $H'$ of $G'$ that has $n$ coends in $G'$, and furthermore 
$$d_{Haus}(f(H), H')<\infty.$$
\end{thm}

\begin{proof}
Suppose first that $n<\infty$.
By Lemmas \ref{2ended_subgps_are_qlines} and \ref{NRH_sats_iness}, for any $R$ such that $N_R(H)$ is connected, we ahve that $N_R(H)$ is a $(\phi, N)$ quasi-line that satisfies $iness(m_1)$, where $\phi, N$, and $m_1$ all depend on $R$.
As $\tilde{e}(G,H)= n$, it follows from Lemma \ref{coends_and_ess_comps} that we can further choose $R$ so that $N_R(H)$ is $n$-parting.
Let $L = N_R(H)$ for some such $R$.

Then, by Lemmas \ref{imgs_of_qlines_are_qlines} and  \ref{iness_preserved_under_qi}, and \ref{ess_comps_go_to_ess_comps}, there is some $R'$ such that $N_{R'}(f(L))$ is a quasi-line satisfying $iness(m_1')$, and $N_{R'}(f(L))$ is $n$-parting.  Let $L'$ denote $N_{R'}(f(L))$ for some such $R'$.

Proposition \ref{LnearZ} implies that there is some $H' \cong \mathbb{Z}$ that is a finite Hausdorff distance from $L'$.  Let $L'' = N_{R''}(H')$, with $R''>0$ such that $L''$ contains $L'$.  Then $L''$ is $n$-parting, so, by Lemma \ref{coends_and_ess_comps}, $\tilde{e}(G', H') \geq n$.

If $\tilde{e}(G',H')>n$, then Lemma \ref{coends_and_ess_comps} and Lemma \ref{ess_comps_go_to_ess_comps}, applied to a quasi-inverse of $f$, implies that there is some quasi-line that is a finite Hausdorff distance from $H$ and is $m$-parting for some $m>n$.  There is a neighborhood of $H$ contains this quasi-line, thus is $m$-parting, so by Lemma \ref{coends_and_ess_comps}, $\tilde{e}(G,H) \geq m$, a contradiction.

Thus $\tilde{e}(G', H')=n$, so $H'$ is the desired subgroup.

Suppose then that $\tilde{e}(G,H)=\infty$.  
Then $\tilde{e}(G,H) \geq m$ for any $m < \infty$, and the above argument shows that there is some $R''>0$ such that $N_{R''}(H')$ is an $m$-parting quasi-line.
Thus Lemma \ref{coends_and_ess_comps} implies that $\tilde{e}(G', H')=\infty$.
\end{proof}


\section{The quasi-isometry invariance of commensurizer subgroups}

In this section, we will see that commensurizers of two-ended subgroups with at least three coends are invariant under quasi-isometries.

More specifically, we saw in the last section that, if $f \co \mathscr{C}^1(G) \to \mathscr{C}^1(G')$ is a quasi-isometry between the Cayley graphs of one-ended, finitely presented groups $G$ and $G'$, and $H$ is a two-ended subgroup of $G$ with at least three coends, then there is a two-ended subgroup $H'$ with at least three coends, that is a finite Hausdorff distance from $f(H)$ in $\mathscr{C}^1(G')$.
We will now see that in fact Comm$_{G'}(H')$ is a finite Hausdorff distance from the image under $f$ of Comm$_G(H)$ in $\mathscr{C}^1(G')$ (Theorem \ref{Cqiinv}).  

We first observe the geometric structure of commensurizers: 

\begin{lem}\label{commcharacterization}
If $G$ is a finitely generated group with subgroup $H$, then 
$${\rm Comm}_G(H) = \{ g \in G : d_{Haus}(H,gH)<\infty \}.$$
\end{lem}

\begin{proof}
Let $l(g)$ be the minimal word length of representatives for $g \in G$, with respect to the given finite generating set for $G$.  Then note that, for all $x, g \in G$, $d(x,xg) =d(e,g) = l(g)$.
Thus $d_{Haus}(gH, gHg^{-1}) \leq l(g^{-1})$, so it suffices to show that $g \in Comm_G(H)$ if and only if $d_{Haus}(H, gHg^{-1})<\infty$.

Let $H^g$ denote $gHg^{-1}$.  If $d_{Haus}(H, H^g)=M<\infty$, then, for any $x \in H$, there is some $y \in H^g$ such that $d(x,y) \leq M$, i.e. $d(y^{-1}x, e) = l(y^{-1}x)\leq M$.
Let $L(M) = \{ k \in G : l(k) \leq M\}$.  It follows that
\begin{equation}\label{comm1}
H \subset \cup_{k\in L(M)}H^gk
\end{equation}
and similarly that
\begin{equation}\label{comm2}
H^g \subset \cup_{k\in L(M)}Hk.
\end{equation}
Observe that in fact (\ref{comm1}) and (\ref{comm2}) are equivalent to having $d_{Haus}(H, H^g)\leq M$.

If $H$ meets $H^gk$, then there is some $h_1 \in H$ with $H^gk=H^gh_1$.
$G$ is finitely generated, so $L(M)$ is finite, and it follows that (\ref{comm1}) implies that there are finitely many elements $h_1, \ldots , h_n$ in $H$ such that
$$H \subset \cup_{i=1}^n H^gh_i.$$
Thus $H = \cup_{i=1}^n (H \cap H^g)h_i$, i.e. $(H \cap H^g)$ is of finite index in $H$.
Similarly $(H \cap H^g)$ is of finite index in $H^g$, so $H$ and $H^g$ are commensurable, hence $g \in $ Comm$_G(H)$.

Conversely, if $g \in $ Comm$_G(H)$, then there are elements $h_1, \ldots , h_n$ in $H$ such that $H = \cup_{i=1}^n (H \cap H^g)h_i$, and elements $h_1', \ldots , h_n'$ in $H^g$ such that $H^g = \cup_{i=1}^{n'} (H \cap H^g)h_i'$.
In particular, (\ref{comm1}) and (\ref{comm2}) hold if we take $M$ to be the maximal word length of the $h_i$'s and $(h_i')$'s.

Thus $d_{Haus}(H, H^g)\leq M$, so we have shown the lemma.
\end{proof}

\begin{rmk}\label{char_comm}
As we saw in the proof of Theorem \ref{existence_qi_inv}, if $H$ is a two-ended subgroup of $G$ with at least $n$ coends, $n < \infty$, then there is some $R$ such that $N_R(H)$ is an $n$-parting $(\phi, N)$ quasi-line satisfying $iness(m_1)$, for some $\phi, N$, and $m_1$.
Thus, by Lemma \ref{commcharacterization} and since $G$ acts on its Cayley graph by isometries on the left, 
$$N_R({\rm Comm}_G(H)) = \cup_{g \in {\rm Comm}_G(H)}N_R(gH) = \cup_{g \in {\rm Comm}_G(H)}g\cdot N_R(H) $$
is a union of isometric copies of $N_R(H)$ that are pairwise of finite Hausdorff distance from one another.  Hence we may think of Comm$_G(H)$ as a collection of ``parallel'' $n$-parting $(\phi, N)$ quasi-lines that satisfy $iness(m_1)$.
\end{rmk}

Consider the following.

\begin{prop}\label{fin_H_dist} 
Let $G$ be a one-ended, finitely presented group with a two-ended subgroup $H$ that has at least three coends, and let $C= $ Comm$_G(H) = \coprod_i g_iH$.  
Given quasi-line parameters $\phi$ and $N$, and a function $m_1$, there exists a constant $x = x(\phi, N, m_1, H)$ such that, if $L$ is a $3$-parting $(\phi, N)$ quasi-line in $\mathscr{C}^1(G)$ satisfying $iness(m_1)$ and such that $d_{Haus}(L,H)<\infty$, then, for some $i$, $d_{Haus}(L, g_iH)<x$.
\end{prop}


Assuming this proposition for the moment, we shall see how it implies the invariance of these commensurizer subgroups under quasi-isometries.

Suppose that $f \co \mathscr{C}^1(G) \to \mathscr{C}^1(G')$ is a quasi-isometry between the Cayley graphs of one-ended, finitely presented groups $G$ and $G'$, that $H$ is a two-ended subgroup of $G$ with at least $n$ coends, for some $3 \leq n < \infty$, and that $C = $ Comm$_{G}(H)$.  
Then, by Remark \ref{char_comm}, we have that some neighborhood $N_R(C)$ of $C$ is a union of pairwise finite Hausdorff distance, $n$-parting $(\phi ,N)$ quasi-lines $\{ L_i\}$, all of which satisfy $iness(m_1)$ for some $m_1$.  

By Theorem \ref{existence_qi_inv}, there is a two-ended subgroup $H'$ of $G'$ such that $\tilde{e}(G,H) = \tilde{e}(G', H')$ and $d_{Haus}(f(H), H')$ is finite.  
Let $C' = $ Comm$_{G'}(H')$.
By Lemma \ref{imgs_of_qlines_are_qlines}, there exists $R'$ such that the $R'$-neighborhood of each $f(L_i)$ is a $(\phi', N')$ quasi-line, for some $\phi'$ and $N'$ depending on $R'$.  
By Lemma \ref{ess_comps_go_to_ess_comps}, we can choose $R'$ so that each $N_{R'}(f(L_i))$ is $3$-parting.  
By Lemma \ref{iness_preserved_under_qi} and Remark \ref{iness_preserved}, we can further suppose that each $N_{R'}(f(L_i))$ satisfies $iness(m_1')$, for some fixed $m_1'$.
Thus we may apply Proposition \ref{fin_H_dist} to get some $x$ such that each $N_{R'}(f(L_i))$ is contained in $N_x(C')$.  
It follows that $N_{R'}(f(N_R(C))) \subset N_x(C')$, i.e. that $f(C)$ is contained in a finite neighborhood of $C'$.

As was the case for $C$, recall that a neighborhood of $C'$ is a union of quasi-lines as above.
Thus, by running the same argument on a quasi-inverse to $f$, it follows that $d_{Haus}(f(C), C')<\infty$.  Hence we have the following.


\begin{thm} \label{Cqiinv}
Let $f\co \mathscr{C}^1(G) \to \mathscr{C}^1(G')$ be a $(\Lambda, K)$-quasi isometry between finitely presented, one-ended groups, and suppose that $H$ is a two-ended subgroup of $G$ with $n$ coends in $G$, for $n \geq 3$.
Then there is a two-ended subgroup $H'$ of $G'$ such that $H'$ has $n$ coends in $G'$ and there exists some constant $y=y(G, H, \Lambda, K)$ such that 
$$d_{Haus}(f({\rm Comm}_{G}(H)), {\rm Comm}_{G'}(H'))< y.$$
\end{thm}

\begin{proof}[Proof of Proposition \ref{fin_H_dist}]
Let $\mathscr{L}$ be the set of $3$-parting $(\phi, N)$ quasi-lines in $\mathscr{C}^1(G)$ that satisfy $iness(m_1)$, and are a finite Hausdorff distance from $H$.  If $\mathscr{L}$ is finite, then we are done, so assume that $\mathscr{L}$ is infinite, and that no such $x$ exists.  Then we can find a sequence $\{L_i\}$ of elements of $\mathscr{L}$ such that
$$ \min_{g \in C}  d_{Haus}(L_i, gH) \to \infty,$$
as $i \to \infty$.

Let $c_{i}=g \in C$ realize the minimum above for $L_i$, and fix $x_1=x_1(G,\phi, N, m_1)$ from Lemma \ref{seq_of_qlines_is_fin}.
Then we can pass to a subsequence so that, for all $j>i$, 
\begin{equation}
d_{Haus}(L_j, c_{j}H)>d_{Haus}(L_i, c_{i}H)+x_1.
\label{assumption}
\end{equation}
Then, by the following argument we will have that, for all $g,g' \in G$ and $i \neq j$, we have
\begin{equation}
d_{Haus} (gL_i, g'L_j)>x_1.
\label{eq1}
\end{equation}
Firstly, note that it suffices to show that $d_{Haus}(L_i, gL_j)>x_1$, for any $g \in G$ and $i<j$.  If $g \notin C$, then $d_{Haus}(H, gH) = \infty$.  But $d_{Haus}(L_i, H)$ and $d_{Haus}(gL_j, gH)$ are finite, so $d_{Haus}(L_i, gL_j) = \infty$.

Assume then that $g \in C$, and $d_{Haus} (L_i, gL_j)\leq x_1$.  Then 
$$d_{Haus}(gL_j, c_{i}H) \leq d_{Haus}(gL_j, L_i) + d_{Haus}(L_i, c_iH) \leq x_1 + d_{Haus}(L_i, c_{i}H).$$ 
Thus $$d_{Haus}(L_j, g^{-1}c_{i}H) = d_{Haus}(gL_j, c_{i}H) \leq x_1 + d_{Haus}(L_i, c_{i}H).$$  But note that $d_{Haus}(L_j, c_{j}H) \leq d_{Haus}(L_j, g^{-1}c_{i}H)$ by the definition of $c_j$, so we have that
$$d_{Haus}(L_j, c_{j}H)  \leq x_1+d_{Haus}(L_i, c_{i}H),$$ 
contradicting (\ref{assumption}).
Thus (\ref{eq1}) holds for all $g \in G$.

By translating the $L_i$'s, we can obtain a new set of quasi-lines that each contain $e \in G$, and for which (\ref{eq1}) holds for all $g \in G$, though the quasi-lines may no longer be a finite Hausdorff distance from $H$.  
This new sequence of quasi-lines satisfies the hypotheses of Lemma \ref{seq_of_qlines_is_fin}.  This leads to a contradiction, since we had assumed $\mathscr{L}$ to be infinite.
\end{proof}

\section{Commensurizers of type $F_n$}

Recall that a group $C$ is of type $F_n$ if there is a $K(C,1)$ with finite $n$-skeleton.  
Being of type $F_1$ is equivalent to being finitely generated, and being of type $F_2$ is equivalent to being finitely presented.

We shall prove the following result in this section:

\begin{thm}\label{fn_thm}
Let $f \co \mathscr{C}^1(G) \to \mathscr{C}^1(G')$ be a quasi-isometry between finitely generated groups, and suppose that $C$ is a subgroup of $G$, $C'$ is a subgroup of $G'$, and that $d_{Haus}(f(C), C')<\infty$.
Then $C$ is of type $F_n$ if and only if $C'$ is of type $F_n$, for $n \geq 1$.
\end{thm}

In light of Theorem \ref{Cqiinv}, we have, as an immediate corollary:

\begin{cor}\label{fn_cor}
Let $f, G, G', H$ and $H'$ be as in Theorem \ref{Cqiinv}.  Then Comm$_G(H)$ is of type $F_n$ if and only if Comm$_{G'}(H')$ is of type $F_n$.
\end{cor}

We can prove Theorem \ref{fn_thm} in the case that $n=1$ with a short and simple argument, using a coarse geometric characterization for a subgroup to be finitely generated (Lemma \ref{char_fg_subgps}).
To prove the theorem for $n>1$, we will introduce some new terminology and basic facts about ``uniformly distorting'' maps and ``coarse isometries''.
In \cite{Kap_ggt_book}, Kapovich gives a proof that being of type $F_n$ is a quasi-isometry invariant, and we will note that his arguments go through in the more general setting of coarse isometries.

First, we shall see that Theorem \ref{fn_thm} holds when $n=1$.
Consider the following.

\begin{lem}\label{char_fg_subgps}
Let $C$ be a subgroup of a finitely generated group $G$.  Then $C$ is finitely generated if and only if there exists some $A_0>0$ such that, for any $g,h \in C$, there is some sequence $s_0, s_1, \ldots s_n \subset C$ so that $g=s_0$, $h=s_n$, and for all $i$, $d(s_i, s_{i+1})< A_0$.
\end{lem}

\begin{proof}
Call a sequence $\{ s_i\}$ as in the statement of the lemma an {\it $A_0$-chain from $g$ to $h$}.  If $C$ is finitely generated, then fix a generating set $S_C$ for $C$, and note that the generators of $C$ have word length in $\mathscr{C}^1(G)$ less than some constant $A_0$.  For any $g,h \in C$, we can represent $g^{-1}h$ by a word $s_1s_2\cdots s_m$ with each $s_i$  in $S_C$, and then the sequence $e$, $s_1$, $s_1s_2$, $\ldots$, $s_1s_2\cdots s_m=g^{-1}h$ is a $A_0$-chain from $e$ to $g^{-1}h$, and hence $g, gs_1, gs_1s_2, \ldots , gs_1s_2\cdots s_m=h$ is a $A_0$-chain in $C$ from $g$ to $h$.

Assume now that $C$ contains a $A_0$-chain between any two of its elements, for some $A_0$, and let $S_C = C \cap B_{A_0}(e)$.  
Since $G$ is finitely generated, $S_C$ is finite, and we claim that $S_C$ generates $C$.  
Fix any $h \in C$, and let $e=s_0,$ $ s_1,$ $  \ldots$ $ s_{n-1},s_n=h$ be a  $A_0$-chain in $C$ from $e$ to $h$.  
Then $h = s_0(s_0^{-1}s_1)(s_1^{-1}s_2)\cdots$ $(s_{n-2}^{-1}s_{n-1})(s_{n-1}^{-1}s_n)$, with $s_0 = e$ and $(s_{i}^{-1}s_{i+1})$ in $S_C$ for each $i$.  
Thus $S_C$ generates $C$, so we are done.  
\end{proof}

\begin{proof}[Proof of Theorem \ref{fn_thm} in the case that $n=1$]
Recall that $f\co \mathscr{C}^1(G) \to \mathscr{C}^1(G')$ is a quasi-isometry, and that $d_{Haus}(f(C), C') < \infty$.  We shall show that $C$ is finitely generated if and only if $C'$ is.

Let $f$ be a $(\Lambda, \kappa)$-quasi-isometry.  Let $n = d_{Haus}(f(C),C')$, and assume that $C$ is finitely generated.  Since $f$ has a quasi-inverse, it suffices to prove that $C'$ must also be finitely generated.

Fix $g',h' \in C'$, and let $s_0, \ldots s_m$ be a sequence of vertices in $C$ such that $d(f(s_0), g')$ $<n$, $d(f(s_m), h')<n$, and $d(s_i, s_{i+1}) < A_0$ for all $i$.  Let $s_i' = f(s_i)$, and let $s_i'' \in C'$ be such that $d(s_i', s_i'') <n$.  As $d(s_i', s_{i+1}') < \Lambda A_0 +\kappa$, we must have that $d(s_i'', s_{i+1}'') < \Lambda A_0 +\kappa+2n$.  
Then the consecutive terms of the sequence $g', s_0'',s_1'', \ldots , s_n'', h'$ are less than $(\Lambda A_0 + \kappa + 2n)$ apart, thus by the lemma above, $C'$ is finitely generated.  
\end{proof}

Next, we will introduce a few new notions, and then see that the theorem holds for $n>1$.
If $C$ is a finitely generated subgroup of a finitely generated group $G$, then we shall use $d_G$ to denote the metric on $G$, and hence $C$, induced from the finite generating set that is fixed for $G$, and we shall denote by $d_C$ the metric on $C$ that is induced from the finite generating set fixed for $C$.
For any $w \in C$, we shall write $|w|_C$ for $d_C(e,w)$, and we shall write $|w|_G$ for $d_G(e,w)$, for any $w \in G$.
For simplicity, in the remainder of this section we shall work with metric spaces such as $(G, d_G)$, instead of with Cayley graphs.

\begin{defn}
We shall say that a map between metric spaces, $f \co (X,d_X) \to (Y,d_Y)$, is {\em $(\phi, \Phi)$-uniformly distorting}, or {\em $(\phi, \Phi)$-u.d.} if $\phi$ and $\Phi$ are weakly increasing proper maps from $\mathbb{R}_{\geq 0}$ to $\mathbb{R}_{\geq 0}$, or from $Im(d_X)$ to $Im(d_Y)$, such that, for any $x, x' \in X$ and any $r$,
\begin{enumerate}
\item{if $d_X(x,x') \geq r$ then $d_Y(f(x), f(x')) \geq \phi (r)$, and}
\item{if $d_X(x,x') \leq r$ then $d_Y(f(x), f(x')) \leq \Phi(r)$.}
\end{enumerate}
We will say that $f$ is {\em u.d.} if $f$ is $(\phi, \Phi)$-u.d. for some $\phi$ and $\Phi$.
\end{defn}

Note that the composition of u.d. maps is u.d.

\begin{convention}
The metric spaces that we are interested in are groups with word metrics, hence all distance functions below will take on only integer values.  
Therefore we will only consider $(\phi, \Phi)$-u.d. maps where we shall take the domain and range of $\phi$ and $\Phi$ to be $\mathbb{Z}_{\geq 0}$.
\end{convention}

We note the following fact:

\begin{lem}\label{inclusion_is_wup}
Let $G$ be a finitely generated group with finitely generated subgroup $C$.  Then the indentity map $i_C \co (C, d_C) \to (C, d_G)$ is u.d.
\end{lem}

\begin{proof}
We shall see that the geometric action of $C$ on itself ensures that all metric distortion is uniform.

Let 
$$\phi (r) = \min \{ |c|_G : c \in C, |c|_C \geq r\},$$
and let 
$$\Phi(r) = \max \{ |c|_G : c \in C, |c|_C \leq r\}.$$
As $C$ is finitely generated, $\Phi$ is finite valued, and note that both functions are weakly increasing.
In addition, for any $c_1, c_2 \in C$, let $c = c_1^{-1}c_2$ and note that $d_C(c_1, c_2) = |c|_C \geq r$ implies $d_G(c_1, c_2) = |c|_G \geq \phi(r)$ and $d_C(c_1, c_2) = |c|_C \leq r$ implies $d_G(c_1, c_2) = |c|_G \leq \Phi(r)$.
As $G$ is finitely generated, hence locally finite, it follows that $\lim_{r \to \infty} \phi(r) = \lim_{r \to \infty} \Phi(r) = \infty$, and hence that $\phi$ and $\Phi$ are proper.
\end{proof}

We note, though we shall not make use of this fact, that in the above proof, the function $\Phi$ is bounded above by a linear function.
For let $S_C$ denote the finite generating set for $C$ and let 
$$L = \max_{s \in S_C} |s|_G.$$
It follows that, for any $c \in C$, $|c|_G \leq L|c|_C$, hence $\Phi(r) \leq Lr$.

\begin{defn}
If $f \co (X, d_X) \to (Y,d_Y)$ is a map between metric spaces and $t \geq 0$, then we will say that $f$ is {\em $t$-onto} if the $t$-neighborhood of $Im(f)$ in $Y$ is equal to $Y$.
If $f$ is $t$-onto for some $t$, then we will say that $f$ is {\em coarsely onto}.

If $f$ is both u.d. and coarsely onto, then we shall say that $f$ is a {\em coarse isometry}.
\end{defn}

We note that any quasi-isometry is a coarse isometry.

\begin{defn}
We say that a function $f_1 \co (X, d_X) \to (Y,d_Y)$ has {\em finite distance } from a function $f_2 \co (X, d_X) \to (Y,d_Y)$ if
$$\sup_{x \in X} d_Y(f_1(x), f_2(x)) < \infty.$$
\end{defn}

Justification for the terminology ``coarse isometry'' is in the following fact:

\begin{lem}\label{coarse_inverse}
If $f \co (X, d_X) \to (Y,d_Y)$ is a coarse isometry between metric spaces, then there is a coarse isometry $f' \co (Y,d_Y) \to (X,d_X)$ such that $f'f$ and $ff'$ have finite distances from $id_X$ and $id_Y$ respectively.
\end{lem}

\begin{defn}
We shall call any function $f'$ satisfying the conclusion of the above lemma a {\em coarse inverse} to $f$.
\end{defn}

\begin{proof}[Proof of Lemma \ref{coarse_inverse}]
Let $\phi, \Phi$ and $t$ be such that $f$ is $t$-onto and $(\phi, \Phi)$-u.d., and note that we can assume that $\Phi(0)= 0$.
Let $\pi$ denote nearest point projection from $Y$ to $Im(f)$, and define $f' \co Y \to X$ to take any $y \in Y$ to a point $x \in X$ such that $f(x) = \pi(y)$.

Let 
$$\phi'(r) = \min \{ s \in \mathbb{Z}_{\geq 0} : \Phi(s) \geq r\}$$
and let 
$$\Phi'(r) = \max \{ s \in \mathbb{Z}_{\geq 0} : \phi (s) \leq r\}.$$
Note that both $\phi'$ and $\Phi'$ are weakly increasing, and that $\phi'$ is proper.
As $\phi$ is a proper map, it follows that $\Phi'$ is as well.

Suppose that $y_1, y_2 \in Y$ and $r \geq 0$ are such that $d(y_1,y_2) \geq r$.  Then we have that $d(\pi(y_1), \pi(y_2)) \geq r-2t$, and it follows that $d(f'(y_1), f'(y_2)) \geq \phi' (r-2t)$.
Similarly if $d(y_1, y_2) \leq r$, then $d(\pi(y_1), \pi(y_2)) \leq r+2t$ and hence $d(f'(y_1), f'(y_2)) \leq \Phi'(r+2t)$.
Hence if we let $\phi''(r) = \phi'(r-2t)$ (taking $\phi'$ to be zero on the negative integers) and $\Phi''(r) = \Phi'(r+2t)$, then we have that $f'$ is $(\phi'', \Phi'')$-u.d.

To see that $f'$ is coarsely onto, note that if $f(x) = f(x')$ then $d(x,x') \leq \Phi'(0)$, and hence for any $y \in Y$, $diam(f^{-1}(y)) \leq \Phi'(0)$.
Note that $Im(f')$ meets $f^{-1}(y)$ for each $y \in Y$, and it follows that $f'$ is $\Phi'(0)$-onto, and hence a coarse isometry.

The above argument also implies that $\sup_{x \in X} d(f'f(x), x) \leq \Phi'(0)$, so the composite map $f'f$ is a finite distance from $Id_X$.
On the other hand, for any $y \in Y$, $ff'(y) = \pi(y)$, hence $\sup_{y \in Y} d(ff'(y), y) \leq t$, so $ff'$ is a finite distance from $Id_Y$ as desired.
\end{proof}

We note also the following, the proof of which is left to the reader:

\begin{lem}
Suppose that $f \co X \to Y, g \co Y \to Z$ are coarse isometries.  Then $g f \co X \to Z$ is also a coarse isometry.
\end{lem}

We can now show the following, which explains our interest in coarse isometries:

\begin{prop}\label{cc'_ci}
Let $f \co \mathscr{C}^1(G) \to \mathscr{C}^1(G')$ be a quasi-isometry between finitely generated groups, and suppose that $C$ is a subgroup of $G$, $C'$ is a subgroup of $G'$, and that $d_{Haus}(f(C), C')<\infty$.
Then there is a coarse isometry between $(C, d_C)$ and $(C', d_{C'})$.
\end{prop}

\begin{proof}
Let $i_C \co (C, d_C) \to (C, d_G)$ and $i_{C'} \co (C', d_{C'}) \to (C', d_{G'})$ denote the identity maps on $C$ and $C'$. 
As we saw in Lemma \ref{inclusion_is_wup}, both $i_C$ and $i_{C'}$ are u.d. and hence coarse isometries.
Let $i_{C'}'$ be a coarse inverse to $i_{C'}$.

Let $t$ be such that $d_{Haus}(f(C), C')\leq t$.  Then we can extend $i_{C'}'$ to a coarse isometry $j_{C'}$ from the $t$-neighborhood of $C'$ in $G'$ to $(C', d_{C'})$ by defining the projection map $\pi$ to take each point $n$ in the neighborhood to a point $c' \in C'$ such that $d_{G'}(n,c') \leq t$, and then setting $j_{C'}(n) = i_{C'}'(c')$.
We have that $j_{C'}$ is a coarse isometry, hence so is $j_{C'} \circ f \circ i_C \co (C, d_C) \to (C', d_{C'})$.
\end{proof}

For any discrete metric space $(X,d_X)$ and $d \geq 0$, we let $Rips_d(X)$ denote the $d$-Rips complex of $X$, i.e. the simplicial complex whose vertex set is equal to $X$, and such that any finite collection $X_0$ of vertices spans a simplex if and only if $d_X(x, x') \leq d$ for all $x,x' \in X_0$.
It is immediate that the proof of Lemma 2.9 of \cite{Kap_ggt_book} extends to the following:

\begin{lem}\label{simpl_maps}
Let $f \co (X, d_X) \to (Y,d_Y)$ be $(\phi, \Phi)$-u.d.  Then $f$ induces a simplicial map $Rips_d(X) \to Rips_{\Phi(d)}(Y)$ for each $d \geq 0$.
\end{lem}

We recall Definition 2.10 of \cite{Kap_ggt_book}:

\begin{defn}
A metric space $X$ is said to be {\em coarsely $n$-connected} if, for each $r \geq 0$ there exists some $R \geq r$ such that the map $Rips_r(X) \to Rips_R(X)$ induces the trivial maps on $i^{th}$ homotopy groups, for every $0 \leq i \leq n$.
\end{defn}

Corollary 2.15 of \cite{Kap_ggt_book} shows that coarse $n$-connectedness is a quasi-isometry invariant.  
In light of Lemma \ref{simpl_maps}, a minor alteration of that proof shows the following.

\begin{thm}\label{cnc_is_ci_inv}
Coarse $n$-connectedness is a coarse isometry invariant.
\end{thm}

Thus we have

\begin{thm}\label{fn_is_ci_inv}
The property of a finitely generated group being of type $F_n$, $n > 1$, is a coarse isometry invariant.
\end{thm}

\begin{proof}
Suppose that $C$ and $C'$ are finitely generated groups, with $C$ of type $F_n$, and that there is a coarse isometry between $(C,d_C)$ and $(C', d_{C'})$.  
Then Theorem 2.17 of \cite{Kap_ggt_book} implies that $(C,d_C)$ is coarsely $(n-1)$-connected, hence, by Theorem \ref{cnc_is_ci_inv}, so is $(C', d_{C'})$.  It is shown in the proof of Theorem 2.21 of \cite{Kap_ggt_book} that each coarsely $(n-1)$-connected group has type $F_n$, so the theorem follows.
\end{proof}

Thus, Theorem \ref{fn_thm} in the case that $n>1$ is immediate from Proposition \ref{cc'_ci} and Theorem \ref{fn_is_ci_inv}.


\section{The quasi-isometry invariance of the vertices of commensurizer type of the Scott-Swarup JSJ decomposition}

The proof of Theorem \ref{Cqiinv} was motivated by the goal of showing that certain vertex groups of the Scott-Swarup JSJ decomposition for finitely presented, one-ended groups are invariant under quasi-isometries.  
We will see that our result follows immediately from Theorems \ref{existence_qi_inv} and \ref{Cqiinv}, given the theory of Scott and Swarup.

We shall begin with some basic facts about the Scott-Swarup decomposition, and then discuss Papasoglu's results from \cite{papa_qlines} that show the invariance under quasi-isometry of certain parts of the decomposition.
Next we will discuss Scott and Swarup's theory in more detail, and see that our invariance results follow from Theorems \ref{existence_qi_inv} and \ref{Cqiinv}.

For an introduction to splittings and graphs of groups decompositions of groups, and group actions on trees, the reader is referred to \cite{ScottWall}.
In \cite{SS}, Scott and Swarup construct a canonical JSJ decomposition $\Gamma_1(G) $ of any one-ended, finitely presented group $G$, in which the vertex groups ``enclose'' all splittings of $G$ over two-ended subgroups, and moreover, enclose all nontrivial almost invariant subsets of $G$ over two-ended subgroups (see Definition \ref{def_enclosing2}).  

If $v$ is a vertex of $\Gamma_1(G)$, then we shall denote its vertex group by $G(v)$, which is defined up to conjugacy.  Similarly if $e$ is an edge of $\Gamma_1(G)$, then we shall let $G(e)$ denote its edge group.

$\Gamma_1(G)$ is a regular neighborhood, as defined in \cite{SS}, of all the nontrivial almost invariant subsets of $G$ over two-ended subgroups.  
Thus $\Gamma_1(G)$ is a bipartite graph of groups with fundamental group $G$, and with the vertices in the complementary subsets called {\it $V_0$-vertices} and {\it $V_1$-vertices}.  
If $v$ is a vertex of $\Gamma_1(G)$, then $G(v)$ is said to be either a {\it $V_0$-} or {\it $V_1$-vertex group}, depending on whether $v$ is a $V_0$- or $V_1$-vertex.

Furthermore, each nontrivial almost invariant subset of $G$ over a two-ended subgroup is ``enclosed'' by some $V_0$-vertex.
We shall define nontrivial almost invariant sets and the notion of enclosure below, but in the case that such an almost invariant set is associated to a splitting of $G$, this means that the enclosing $V_0$-vertex group admits a splitting that is compatible with $\Gamma_1(G)$.  
Moreover, when $\Gamma_1(G)$ is refined by this splitting, the added edge is associated to the given splitting of $G$.
Each $V_0$-vertex of $\Gamma_1(G)$ encloses at least one such splitting of $G$ over a two-ended subgroup.

Each $V_0$-vertex $v$ is one of three types:
\begin{enumerate}
  \item{$v$ is isolated}
  \item{$v$ is of Fuchsian type, or}
  \item{$v$ is of commensurizer type.}
\end{enumerate}

If $v$ is isolated, then $v$ is of valence two.  Moreover, if we let $e_1$ and $e_2$ denote the edges incident to $v$, then the inclusions of $G(e_1)$ and $G(e_2)$ into $G(v)$ are isomorphisms, and all three subgroups are two-ended. 

If $v$ is of Fuchsian type, then $G(v)$ is finite-by-Fuchsian, where the Fuchsian group is a discrete group of isometries of the hyperbolic plane or of the Euclidean plane, but is not finite nor two-ended.  Associated to each peripheral subgroup of $G(v)$ there is exactly one corresponding edge $e$ incident to $v$, and $G(e)$ is conjugate to that subgroup.

Lastly, if $v$ is of commensurizer type, then $v$ is not isolated nor of Fuchsian type, and there is a two-ended subgroup $H$ of $G$ with $\tilde{e}(G,H) >3$ such that $G(v) = $ Comm$_G(H)$.  Only in this case is it possible that the subgroups carried by the edges incident to $v$ are not two-ended, and in fact they may not even be finitely generated.  It follows that the $V_1$-vertex groups of $\Gamma_1(G)$ may not be finitely generated either.

We say that a subgroup $C$ of $G$ is a vertex group of isolated, Fuchsian or commensurizer type respectively if $C$ is the vertex group of a vertex of $\Gamma_1(G)$ of isolated, Fuchsian or commensurizer type respectively.
\\

It is natural to ask if $\Gamma_1(G)$ is somehow invariant under quasi-isometries.
While the underlying graph of $\Gamma_1(G)$ need not be invariant, one could ask whether or not the existence of vertex groups of certain types is invariant under quasi-isometries. 
If the answer to this is `yes', then one could ask if the locations of these vertex groups is also invariant under quasi-isometries, in the sense of a quasi-isometry being forced to take a vertex group to within a finite Hausdorff distance of a vertex group of the same type.
Papasoglu has addressed these questions for the Dunwoody-Sageev JSJ decomposition of one-ended, finitely presented groups.

The JSJ decomposition of a group $G$ as given by Dunwoody and Sageev in \cite{DunSa} is a graph of groups decomposition of $G$, say $\Gamma_{DS}(G)$, which is bipartite.  
Call the two types of vertex groups white and black, and then all the black vertex groups are either of Fuchsian type or of isolated type (see above).
$\Gamma_{DS}(G)$ describes all the splittings of $G$ over two-ended subgroups, in the sense that if $G$ splits over a two-ended subgroup $C$, either as $A*_CB$ or $A*_C$, then $C$ is conjugate into a vertex group of $\Gamma_{DS}$, has a  finite index subgroup which is contained in a black vertex group, and each white vertex group is conjugate into $A$ or $B$.

All of the edge groups of $\Gamma_{DS}(G)$ are two-ended, and it is this that Papasoglu exploits in \cite{papa_qlines} to prove the quasi-isometry invariance of this JSJ decomposition.  Specifically, the author proves the following.

\begin{thm} {\rm \cite{papa_qlines}} \    
\label{papa_jsj}
Let $G$ and $G'$ be one-ended, finitely presented groups. 
Suppose that 
$f \co \mathscr{C}^1(G) \to \mathscr{C}^1(G')$ is a quasi-isometry.  
Then there is a constant $C>0$ such that if $A$ is a subgroup of $G$ conjugate to a vertex group, a vertex group of Fuchsian type, or an edge group of the graph of groups $\Gamma_{DS}(G)$, then $f(A)$ has Hausdorff distance less than or equal to $C$ from a subgroup of $G'$ conjugate to, respectively, a vertex group, a vertex group of Fuchsian type, or an edge group of $\Gamma_{DS}(G')$.
\end{thm}

Given any one-ended, finitely presented group $G$, $\Gamma_{DS}(G)$ differs from $\Gamma_1(G)$ as follows.  The Fuchsian type vertex groups of $\Gamma_{DS}(G)$ and the Fuchsian type vertex groups of $\Gamma_1(G)$ are the same (up to conjugacy), and have the same edge groups.   
Also, the isolated vertex groups of $\Gamma_1(G)$ are vertex groups of $\Gamma_{DS}(G)$, and have the same edge groups.  
Thus $V_1$-vertices adjacent only to Fuchsian and isolated vertices of $\Gamma_1(G)$ are the same as the corresponding white vertex groups of $\Gamma_{DS}(G)$. 
So $\Gamma_1(G)$ differs from $\Gamma_{DS}(G)$ only at the vertices of commensurizer type, and the adjacent edges and $V_1$-vertices.

Thus Theorem \ref{papa_jsj} shows the invariance under quasi-isometries of the $V_0$-vertex groups of $\Gamma_1(G)$ of isolated and Fuchsian types, as well as those $V_1$-vertices that are adjacent only to the isolated and Fuchsian $V_0$-vertices.
It does not, however, answer the question of the invariance of the vertex groups of commensurizer type.

In fact the vertices of commensurizer type are invariant under quasi-isometries, and this fact is an immediate corollary to Theorems \ref{existence_qi_inv} and \ref{Cqiinv}, in light of the following fact about $\Gamma_1(G)$:

\begin{thm}\label{SSJSJ_comm_char}
Let $G$ be a one-ended finitely presented group, and let $H$ be a two-ended subgroup of $G$.
Then Comm$_G(H)$ is a vertex group of $\Gamma_1(G)$ of commensurizer type if and only if $\tilde{e}(G,H) \geq 4$.
\end{thm}

This theorem is not stated explicitly in \cite{SS}, so we digress in order to explain how it follows from that work.
\\

We will need a more detailed description of $\Gamma_1(G)$ in order to do this. 
We begin with some definitions from \cite{SS}.

Let $H$ be a subgroup of $G$, let $X$ be a subset of $G$ and let $X^*$ denote the complement of $X$ in $G$.
Recall that $X$ is said to be $H$-finite if it is contained in the union of finitely many cosets $Hg$ of $H$.

\begin{defn}\label{def_ai}
$X$ is an {\em $H$-almost invariant subset of $G$}, or an {\em almost invariant subset of $G$ over $H$}, if $HX=X$ and the symmetric difference of $X$ and $Xg$ is $H$-finite for all $g \in G$.

We say that an $H$-almost invariant set $X$ is {\em nontrivial} if neither $X$ nor $X^*$ is $H$-finite.
\end{defn}

(We note that any $H$-almost invariant subset of $G$ represents an element of $(\mathcal{P}G/\mathcal{F}_HG)^G$ from Section \ref{section3}, though a representative of an element of $(\mathcal{P}G/\mathcal{F}_HG)^G$ need not be fixed by the left action of $H$.)

\begin{rmk} \label{splitting_gives_ai_set}
$H$-almost invariant subsets of $G$ generalize splittings of $G$ over $H$, for there is a natural way to associate to any splitting of $G$ over $H$ an $H$-almost invariant set as follows.  

Suppose that $G$ admits a splitting over $H$, and let $T$ be the associated tree.
Let $V(T)$ denote the vertex set of $T$, fix a basepoint $w \in V(T)$ and let $e$ be an oriented edge of $T$ with stabilizer $H$.
Then $e$ determines a partition of $V(T)$:  consider the two subtrees of $T$ resulting from the removal of the interior of $e$.  
Let $Y_e$ denote the vertex set of the subtree containing the terminal vertex of $e$ and let $Y_e^*$ denote the vertex set of the other subtree.
Let $\varphi \co G \to V(T)$ be defined by setting $\varphi (g) = g\cdot w$, let $Z_e = \varphi ^{-1}(Y_e)$ and let $Z_e^* = \varphi^{-1}(Y_e^*)$.

Lemma 2.10 of \cite{SS} shows that $Z_e$ and $Z_e^*$ are $H$-almost invariant subsets of $G$.
Moreover, these sets are canonically associated to the given splitting, up to complementation and $H$-finite symmetric difference.
(A different choice of basepoint, for instance, would result in an $H$-almost invariant set that has $H$-finite symmetric difference with $Z_e$.)  See section 2 of \cite{SS} for more details.

By looking at the translates $gZ_e$ of $Z_e$ by elements of $G$, one can recover the action of $G$ on $T$, but in general one does not get a tree action by looking at translates of almost invariant sets.
\end{rmk}

There is a notion of almost invariant sets crossing:

\begin{defn}\label{def_crossing}
Let $X$ be an $H$-almost invariant subset of $G$ and let $Y$ be a $K$-almost invariant subset of $G$.
We say that {\em $Y$ crosses $X$} if none of the four sets $X \cap Y$, $X^* \cap Y$, $X \cap Y^*$ and $X^* \cap Y^*$ is $H$-finite.
\end{defn}

Scott shows in \cite{Scott_intersection_numbers} that the crossing of nontrivial almost invariant sets is symmetric:

\begin{thm}{\rm \cite{Scott_intersection_numbers}} \ 
Let $G$ be a finitely generated group with subgroups $H$ and $K$, let $X$ be a nontrivial $H$-almost invariant subset of $G$ and let $Y$ be a nontrivial $K$-almost invariant subset of $G$.

Then $Y$ crosses $X$ if and only if $X$ crosses $Y$.
\end{thm}

Recall that, if $X$ is a subset of $G$ then we may think of $X$ as a subset of the vertex set of $\mathscr{C}^1(G)$, hence we may think of the coboundary of $X$, $\delta X$, as a collection of edges in $\mathscr{C}^1(G)$.

In \cite{SS}, notions of strong and weak crossings play an important role.

\begin{defn}
Let $X$, $Y$ be as in Definition \ref{def_crossing}.
Then we say that {\em $Y$ crosses $X$ strongly} if $\delta Y \cap X$ and $\delta Y \cap X^*$ project to infinite sets in $H \backslash \mathscr{C}^1(G)$.

If $Y$ crosses $X$, but not strongly, then we say that {\em $Y$ crosses $X$ weakly}.
\end{defn}

\begin{rmk}
If $Y$ crosses $X$ strongly, then $Y$ crosses $X$.

Whether or not $Y$ crosses $X$ strongly does not depend on our choice of finite generating set for $G$.

In general, strong crossing is not symmetric:  it is possible to have $Y$ cross $X$ strongly  and $X$ cross $Y$ weakly (see Example 2.26 of \cite{SS}).  However, if $H$ and $K$ are both two-ended, then the notion is symmetric (\cite{SS}, Proposition 7.2).
\end{rmk}


Finally, we must discuss the notion of enclosing.  Recall that, given a based $G$-tree $T$ with an edge $e$, we defined almost invariant sets $Z_e$ and $Z_e*$ in Remark \ref{splitting_gives_ai_set}.

\begin{defn}\label{def_enclosing2}
Let $X$ be a nontrivial $H$-almost invariant subset of a group $G$, let $T$ be a $G$-tree and let $\Gamma$ denote the graph of groups decomposition of $G$ associated to $T$.
Pick a basepoint in $V(T)$, so that, for any oriented edge $e$ of $T$, we can define the sets $Z_e$, $Z_e*$ as in Remark \ref{splitting_gives_ai_set}.

Suppose that $u \in V(T)$ is such that, for all edges $e$ of $T$ that are incident to $u$ and directed towards $u$, $X \cap Z_e*$ and $X^* \cap Z_e*$ are $H$-finite.
Then we say that the vertex $v = G\backslash u$ in $\Gamma$ {\em encloses} $X$.
\end{defn}

Suppose for a moment that $X$ is associated to a splitting $\sigma$ of $G$ (see Remark \ref{splitting_gives_ai_set}).
In this special case, a vertex $v$ in $\Gamma$ enclosing $X$ is equivalent to $\Gamma$ and $\sigma$ having a common refinement $\Gamma'$, that differs from $\Gamma$ only in that $v$ is replaced by an edge, and the splitting of $G$ associated to that edge is $\sigma$.

The reader is referred to Section 4 of \cite{SS} for more detail and motivation for the concept of enclosing.

We have defined everything we need in order to discuss $\Gamma_1(G)$ in enough detail to explain how Theorem \ref{SSJSJ_comm_char} follows from \cite{SS}.
All almost invariant subsets of $G$ discussed in the remainder of this section are over two-ended subgroups.

Recall that $\Gamma_1(G)$ is such that its $V_0$-vertices enclose all nontrivial almost invariant subsets of $G$, and that each $V_0$-vertex encloses at least one such subset.


If $v$ is isolated, then the only almost invariant sets enclosed by $v$ are those from the splitting of $G$ associated to the edges incident to $v$, hence $v$ does not enclose any crossing almost invariant sets.  Conversely, if a $V_0$-vertex $v$ does not enclose any crossing almost invariant sets over two-ended subgroups, then $v$ is isolated.
Moreover, if $X$ is an $H$-almost invariant set of $G$ which is enclosed by $v$, then $\tilde{e}(G,H)$ must be $2$ or $3$ (see part 1 of Theorem 1.9 from \cite{SS_PDn_pairs}).


If $v$ is of Fuchsian type, then $v$ is not isolated, and any almost invariant sets enclosed by $v$ that cross do so strongly.  (See Propositions 7.2, 7.4 and 7.5 of \cite{SS}.)
Also, Theorem 7.8 of \cite{SS}, tells us that, if $X$ is an $H$-almost invariant set that is enclosed by a vertex $v$ of Fuchsian type, then we have that either $\tilde{e}(G,H) = 2$, or $X$ is associated to the splitting given by an edge incident to $v$.


If $v$ is of commensurizer type, then any two almost invariant sets enclosed by $v$ that cross do so weakly.  Moreover, if $X$ and $Y$ are almost invariant sets over subgroups $H$ and $K$ that are enclosed by $v$, then $H$ and $K$ are commensurable (see Propositions 7.3 and 7.5 of \cite{SS}), $G(v)=$ Comm$_G(H)=$ Comm$_G(K)$, and $\tilde{e}(G,H) = \tilde{e}(G,K)$ is at least 4.
Conversely, if $H$ is a two-ended subgroup of $G$ such that $\tilde{e}(G,H) \geq 4$ and there exists a nontrivial $H$-almost invariant subset of $G$, then there is a commensurizer vertex group of $\Gamma_1(G)$ that is equal to Comm$_G(H)$.
(See part 1 of Theorem 1.9 from \cite{SS_PDn_pairs}.) 
We note that if $\tilde{e}(G,H) \geq 4$ and there are no nontrivial $H$-almost invariant subsets of $G$, then Lemma 2.40 of \cite{SS} implies that $H$ contains a (finite index) subgroup $H'$ such that $\tilde{e}(G,H') = \tilde{e}(G,H)$ and there exists a nontrivial $H'$-almost invariant subset of $G$.  
Hence in this case $\Gamma_1(G)$ contains a commensurizer vertex group equal to Comm$_G(H') = $ Comm$_G(H)$.

Thus if $H$ is a two-ended subgroup of $G$, then Comm$_G(H)$ is a vertex group of $\Gamma_1(G)$ of commensurizer type if and only if $\tilde{e}(G,H) \geq 4$.
This proves Theorem \ref{SSJSJ_comm_char}, and thus we have the following immediate corollary to Theorem \ref{existence_qi_inv}:

\begin{cor}
Let $f \co \mathscr{C}^1(G) \to \mathscr{C}^1(G')$ be a quasi-isometry between the Cayley graphs of one-ended, finitely presented groups $G$ and $G'$.  
Then $\Gamma_1(G)$ has a vertex group of commensurizer type if and only if $\Gamma_1(G')$ does.
\end{cor}

Moreover, the next fact follows immediately from Theorem \ref{Cqiinv}.

\begin{cor}\label{SSJSJ_cor}
Let $f \co \mathscr{C}^1(G) \to \mathscr{C}^1(G')$ be a quasi-isometry between the Cayley graphs of one-ended, finitely presented groups $G$ and $G'$, and suppose that $C$ is a vertex group of $\Gamma_1(G)$ of commensurizer type.
Then there is a vertex group, $C'$, of $\Gamma_1(G')$ of commensurizer type such that 
$$d_{Haus}(f(C), C')<\infty.$$
\end{cor}

In addition, ``small'' and ``large'' vertex groups of commensurizer type are invariant under quasi-isometries.
We recall from \cite{SS} that a vertex group $C=$ Comm$_G(H)$ of $\Gamma_1(G)$ of commensurizer type is said to be {\it small} if $H$ is of finite index in $C$, and otherwise, $C$ is said to be {\it large}.  
Thus $C$ is small if and only if $e(C)=2$.  

Hence Theorem \ref{Cqiinv} also implies the following.

\begin{cor}
If $f \co \mathscr{C}^1(G_1) \to \mathscr{C}^1(G_2)$ is a quasi-isometry between the Cayley graphs of finitely presented, one-ended groups $G_1$ and $G_2$, then $\Gamma_1(G_1)$ has a vertex group of small commensurizer type if and only if $\Gamma_1(G_2)$ does, and $\Gamma_1(G_1)$ has a vertex group of large commensurizer type if and only if $\Gamma_1(G_2)$ does.
\end{cor}


\section{On the quasi-isometry invariance of the topological JSJ decomposition}

We recall that any orientable Haken 3-manifold $M$ with incompressible boundary has a JSJ decomposition, and the ``characteristic'' pieces of this decomposition essentially make up the characteristic submanifold of $M$, $V(M)$.
(See below for a description of $V(M)$.)
In this section, we shall discuss how Scott and Swarup's theory of JSJ decompositions of groups, together with Corollary \ref{SSJSJ_cor}, imply the invariance under quasi-isometries of the Seifert fibered components of $V(M)$ that meet $\partial M$.

We first remark that M. Kapovich and Leeb have used different methods to prove a stronger result that implies this one.
In Theorem 1.1 of \cite{KapLeeb_Haken}, the authors prove the quasi-isometry invariance of all components of characteristic submanifolds of Haken manifolds with zero Euler characteristic.
We also note that, in the earlier work \cite{KapLeeb_SFS}, the authors proved the quasi-isometry invariance of the existence of (not necessarily peripheral) Seifert fibered components of characteristic submanifolds for Haken manifolds with zero Euler characteristic that are not Nil nor Sol.

Our main result in this section is the following.

\begin{cor}\label{3man_app}
Let $M$ and $M'$ be connected orientable Haken 3-manifolds with incompressible boundary, and with $f\co \pi_1(M) \to \pi_1(M')$ a quasi-isometry.  Suppose that $N$ is a nonexceptional Seifert fibered component of the characteristic submanifold of $M$ that meets the boundary of $M$.

Then there is a nonexceptional Seifert fibered component, $N'$, of the characteristic submanifold of $M'$ that meets the boundary of $M'$.
Moreover, if $C$ denotes the subgroup of $\pi_1(M)$ induced by the inclusion of $N$ into $M$ and $C'$ denotes the subgroup of $\pi_1(M')$ induced by the inclusion of $N'$ into $M'$, then
$$d_{Haus}(f(C), C')<\infty.$$
\end{cor}

This corollary will follow immediately from Corollary \ref{SSJSJ_cor} and the following proposition.

\begin{prop} \label{3man_prop}
Let $M$ be a connected orientable Haken 3-manifold with incompressible boundary and let $G=\pi_1(M)$.
Then there is a one-to-one correspondence between the nonexceptional Seifert fibered components of $V(M)$ that meet $\partial M$ and the commensurizer vertex groups of $\Gamma_1(G)$, given by taking a Seifert fibered component $N$ to the subgroup of $G$ given by the inclusion of $\pi_1(N)$ into $\pi_1(M)$.
\end{prop}

Before proving Proposition \ref{3man_prop}, we recall some basic definitions.
We say that a 3-manifold $M$ is {\em irreducible} if every 2-sphere in $M$ bounds a 3-ball.
We call a map of a surface $S$ into $M$ {\em proper} if the map takes $\partial S$ into $\partial M$.
A proper embedding of an orientable surface $S$ that is not the disk or 2-sphere into $M$ is said to be {\em incompressible} if it induces an injection on fundamental groups. 
An embedding of the 2-sphere into $M$ is {\em incompressible} if the image does not bound a 3-ball.
We say that $M$ has {\em incompressible boundary} if the inclusion of $\partial M$ into $M$ induces an injection on fundamental groups.
We say that $M$ is an {\em orientable Haken} 3-manifold if $M$ is compact, orientable, irreducible, and contains an incompressible surface.

A map of the torus into $M$ is said to be {\em essential} if it is incompressible and not homotopic into $\partial M$, and a proper map of the annulus into $M$ is said to be {\em essential} if it is incompressible and is not properly homotopic into $\partial M$.

Following \cite{SS}, we shall say that an embedded essential annulus or torus $S$ in $M$ is {\em canonical} if any essential map of the annulus or torus into $M$ can be properly homotoped until it is disjoint from $S$.
We shall say that a submanifold $N$ of $M$ is {\em simple} if any essential map of an annulus or torus into $M$ with image in $N$ can be properly homotoped into the frontier of $N$.

Jaco and Shalen \cite{JacoShalen} and Johannson \cite{Johannson} proved that there is a unique finite collection $\mathcal{T}$ of disjoint canonical annuli and tori in $M$ such that $\mathcal{T}$ contains one representative from each isotopy class of canonical annuli and tori in $M$.
These authors also showed that the pieces obtained by cutting $M$ along $\mathcal{T}$ are $I$-bundles over surfaces, Seifert fibered, or simple; we shall consider these pieces to be submanifolds of $M$.

We define the characteristic submanifold of $M$, $V(M)$, to be the collection of $I$-bundle and Seifert fibered submanifolds as above, except that if two such submanifolds meet one another at some surface $S \in \mathcal{T}$, then we shall remove a regular neighborhood of $S$ from $V(M)$.  
Also if two simple submanifolds meet at some $S \in \mathcal{T}$, then we shall add a regular neighborhood of $S$ to $V(M)$.

Thus $V(M)$ is a submanifold of $M$, each component of which is a regular neighborhood of an annulus or a torus, an $I$-bundle over a surface, or is Seifert fibered.
We shall say that a component of $V(M)$ is {\em exceptional} if it is a solid torus with frontier 3 annuli of degree 1, or 1 annulus of degree 2 or 3, or is a twisted I-bundle over the Klein bottle as described in \cite{SS}. 
A component of $V(M)$ is called {\em peripheral} if it meets $\partial M$.

We can now show the correspondence between the nonexceptional peripheral Seifert fiber components of $V(M)$ and the commensurizer vertex groups of $\Gamma_1(\pi_1(M))$.

\begin{proof}[Proof of Proposition \ref{3man_prop}]
Let $M$ be a connected orientable Haken 3-manifold with incompressible boundary and let $G=\pi_1(M)$.
For any subgroup $H$ of a group $G$, we shall denote by $N_G(H)$ the normalizer of $H$ in $G$.

The proposition is vacuously true if $G$ is finite, so it suffices to take $M$ such that $G$ is infinite, hence torsion-free.  
Thus the two-ended subgroups of $G$ are infinite cyclic.

If $N$ is a Seifert fiber space with infinite fundamental group, then we have the following short exact sequence:
$$1 \to \mathbb{Z} \to \pi_1(N) \to \pi_1(B) \to 1,$$
where the $\mathbb{Z}$ is generated by a regular fiber of $N$ and $B$ denotes the base 2-orbifold of $N$ (see, for example, \cite{Scott_geometries}).

In fact, \cite{Waldhausen_gruppen} and \cite{Tollefson_involutions_of_SFSs} imply that if $N$ is an orientable Haken 3-manifold with infinite fundamental group, then the converse holds:  if $\pi_1(N)$ has a normal infinite cyclic subgroup then $N$ is Seifert fibered. 
We will be interested in manifolds $N$ that are orientable, irreducible and with nonempty boundary.
Any such $N$ is Haken, so this result will apply.

Consider again $M$ and $G$, and suppose that $H = \langle h \rangle \subset G$ is infinite cyclic.
Recall that $N_G(H) \subset $ Comm$_G(H)$; it follows from work of Jaco and Kropholler that Comm$_G(H) = N_G(\langle h^m \rangle)$ for some $m \geq 1$.
To see this, we have that Jaco showed in \cite{Jaco_roots} that any $g \in $ Comm$_G(H)$ is contained in $N_G(\langle h^n \rangle)$ for some $n$, thus any finitely generated subgroup of Comm$_G(H)$ is contained in $N_G(\langle h^n \rangle)$, for some $n$ depending on the subgroup.
In \cite{Kropholler_note}, Kropholler showed that ascending chains of centralizers in $G$ must terminate, and thus ascending chains of normalizers of infinite cyclic subgroups must also terminate.
Since Comm$_G(H)$ can be exhausted by finitely generated subgroups, it follows that Comm$_G(H) = N_G(\langle h^m \rangle)$ for some $m\geq 1$, and that Comm$_G(H)$ is finitely generated.

We note that if $M$ is Seifert fibered and $H \cong \mathbb{Z}$ denotes the subgroup of $\pi_1(M)$ that is carried by a regular fiber of $M$, then $H$ has infinite index in $\pi_1(M)$, and has more than three coends in $G$ if and only if $M$ is a nonexceptional Seifert fiber space with nonempty boundary.
Thus the proposition follows if $M$ is Seifert fibered or $G$ is itself of commensurizer type.

We shall now prove the proposition, assuming neither of these are the case.
Suppose that $N$ is a Seifert fibered component of $V(M)$ and let $C$ denote the subgroup of $G$ that is carried by $N$.
Thus $N$ is orientable, Haken and with boundary, and $C$ must be finitely generated.
Let $H = \langle h \rangle$ denote the subgroup of $G$ generated by a regular fiber of $N$, so that $H$ is normal in $C$.
As we noted above, $H$ is of infinite index in $C$, and $N$ is peripheral and nonexceptional (hence a peripheral component of $V(M)$) exactly when $\tilde{e}(G,H)>3$.

In this case, $C$ is contained in a commensurizer vertex group $C' = $ Comm$_G(H)$ of $\Gamma_1(G)$, with $C' = N_G(\langle h^m \rangle)$ finitely generated and $m \geq 1$.
Suppose that this containment is proper.

Consider the cover of $M$ with fundamental group $C'$, which we shall call $M_{C'}$, and denote by $\rho$ the projection from $M_{C'}$ to $M$.  
By the theory of Scott \cite{Scott_compact_core}, $M_{C'}$ contains a compact submanifold, let's say $N'$, with fundamental group $C'$.  
Thus $N'$ is a Seifert fiber space, and it follows from \cite{JacoShalen} or \cite{Johannson} that $\rho$ can be homotoped so that $\rho(N')$ is contained in a Seifert fibered component of $V(M)$.

As we have assumed that $C$ is properly contained in $C'$, this component must be different from $N$.  Let us call it $S$, and note that $S$ is nonexceptional and peripheral, for $\tilde{e}(G, \langle h^m \rangle ) > 3$ and hence the subgroup carried by a regular fiber of $S$ will also have more than three coends in $G$.
It follows that there is a collection of canonical annuli and tori that separates $N$ from $S$ in $M$, which we shall call $\Sigma_1, \ldots , \Sigma_k$.  
Hence $G$ has a graph of groups decomposition over the surface groups $\pi_1(\Sigma_i)$, with $C$ and $C'$ contained in distinct vertex groups.  
But $C \subset C'$, so $C$ must be contained, up to conjugacy, in an edge group $\pi_1(\Sigma_i)$.
This is not possible, hence we must have that $C = C'$.

Now suppose that $C' = $ Comm$_G(H)$ is a commensurizer vertex group of $\Gamma_1(G)$, so $C' = N_G(\langle h^m \rangle)$ for some $m \geq 1$ and $C'$ is finitely generated.
Consider the cover of $M$ with fundamental group $C'$, which we shall call $M_{C'}$, and denote by $\rho$ the projection from $M_{C'}$ to $M$.  
As we saw above, $M_{C'}$ contains a Seifert fiber space $N'$, and $\rho$ can be homotoped so that $\rho(N')$ is contained in a Seifert fibered component $S$ of $V(M)$.
Let $D$ denote the inclusion of $\pi_1(S)$ into $G$, and we have that $C' \subset D$ and $D = N_G(H')$, where $H'$ is a finite index subgroup of $\langle h^m \rangle$.
We note that $\tilde{e}(G,H') > 3$, so $S$ is nonexceptional peripheral.
But $N_G(H') \subset $ Comm$_G(H') = $ Comm$_G(H)$, so $D=C'$, and the proposition follows.
\end{proof}


\section{Application to the groups $QI(G)$}\label{qi(g)_section}

Proposition \ref{fin_H_dist} gives us some insight into the structure of groups of quasi-isometries of one-ended, finitely presented groups.
In this section, we provide an analogue (Corollary \ref{qi1}) to a result of Souche and Wiest, who investigate $QI(T\times \mathbb{R}^n)$ for infinite trees $T$ in \cite{SoucheWiest}.
In addition, we note Corollary \ref{qi2}, which is a weaker, but far more general result.

We first introduce the notion of the group of quasi-isometries of a group.
Given metric spaces $X$ and $Y$, one may consider all quasi-isometries from $X$ to $Y$, modulo the relation that $f \sim f'$ when 
\begin{equation}\label{fin_sup_dist}
\sup_{x \in X}d(f(x), f'(x))<\infty.
\end{equation}
We shall denote this set by $QI(X,Y)$.  
It is standard to denote $QI(X , X)$ by $QI(X)$, and $QI(\mathscr{C}^1(G))$ by $QI(G)$, for any finitely generated group $G$.

These latter sets form groups, and are generally very complicated--for instance, Sankaran showed in \cite{Sankaran} that $QI(\mathbb{Z})$ contains Thompson's group F and the free group of continuous rank.

Our first corollary is about the quasi-isometries of commensurizer groups.
Suppose that $G$ is finitely presented, one-ended, and equal to the commensurizer of a two-ended subgroup $H$ such that $\tilde{e}(G,H) \geq 3$.
Then recall that we can think of $H$ as a subset of the vertex set of $\mathscr{C}^1(G)$.
The vertex set of $\mathscr{C}^1(G)$ is equal to the disjoint union of the translates (that is, cosets) of $H$, and any two translates $gH, g'H$ are of finite Hausdorff distance from one another (see Lemma \ref{commcharacterization}).

Thus we can define a metric on $G/H$ such that the distance between $gH$ and $g'H$ is equal to the Hausdorff distance between $gH$ and $g'H$ in $\mathscr{C}^1(G)$.
Note that if $H$ happens to be a normal subgroup of $G$, then this recovers the metric on the vertex set of the Cayley graph for $G/H$, with respect to the given generating set for $G$.

If $f$ is a quasi-isometry from $G$ to itself, then, by Theorem \ref{existence_qi_inv} and Proposition \ref{fin_H_dist}, there is some infinite cyclic $H' \subset G$ with $\tilde{e}(G,H')\geq 3$ and a constant $y$ such that, for each $g \in G$, $d_{Haus}(f(gH), g'H')<y$ for some $g' \in G$.
In addition, it follows from Theorem \ref{2.1} that $H$ and $H'$ are a finite Hausdorff distance from one another in $\mathscr{C}^1(G)$.  Hence there is some constant $z \geq y$ such that, for each $g \in G$, $d_{Haus}(f(gH), g'H)<z$ for some $g' \in G$.

Thus the quasi-isometry $f \co \mathscr{C}^1(G) \to \mathscr{C}^1(G)$ induces a map from $G/H$ to itself, that takes any $gH$ to some point $g'H$ such that $d_{Haus} (f(gH), g'H)$ $< z$.
As $f$ is a quasi-isometry, it follows that this new map is as well.

We note moreover that any quasi-isometry of $G/H$ to itself induces a natural quasi-isometry of $G$ to itself.  
Hence we have the following.

\begin{cor}\label{qi1}
Suppose that $G$ is a one-ended, finitely presented group such that $G = $ Comm$_G(H)$ for a two-ended subgroup $H$ of $G$ that has at least three coends in $G$.
Consider $G/H$, with a metric defined by setting the distance between $gH$ and $g'H$ to equal the Hausdorff distance between the vertex sets $gH$ and $g'H$ in $\mathscr{C}^1(G)$.

Then there is a canonical map $QI(G) \to QI(G/H)$ that is surjective.
\end{cor}

\begin{rmk}
We note that the kernel of this map $QI(G) \to QI(G/H)$ is exactly the set of equivalence classes of quasi-isometries $f \co \mathscr{C}^1(G) \to \mathscr{C}^1(G)$ for which the distances $d_{Haus}(f(gH), gH)$ are uniformly bounded.  
\end{rmk}

We can generalize this idea to get a weaker result for general one-ended finitely presented groups $G$.
Suppose that $G$ contains at least one two-ended subgroup $H$ with $\tilde{e}(G,H) \geq 3$, and for any such $H$, let $M_H$ denote the metric space with underlying set equal to Comm$_G(H)/H$ and the distance between any two points $gH$ and $g'H$ defined to equal the Hausdorff distance between those sets in $\mathscr{C}^1(G)$.

Fix any such $H$, and any quasi-isometry $f \co \mathscr{C}^1(G) \to \mathscr{C}^1(G)$.
Then Theorem \ref{existence_qi_inv} shows that there is some two-ended subgroup of $G$, $H'$, (possibly equal to $H$) such that $d_{Haus}(f(H), H') < \infty$ and $\tilde{e}(G,H) = \tilde{e}(G,H')$.

Proposition \ref{fin_H_dist} implies that $f$ induces a quasi-isometry from $M_H$ to $M_{H'}$ as in the above argument.
Thus, not only do we get a natural map taking $f$ into $QI(M_H, M_{H'})$, but we also get a natural map of $f$ into the symmetric group of a suitable collection of two-ended subgroups of $G$ with a fixed number of coends.

We shall fix some more notation so that we can say this more carefully.
For each $n \in \{ 3,4,5,\ldots \} \cup \{\infty\}$, let $K_n$ denote a maximal collection of two-ended subgroups of $G$ that have $n$ coends in $G$ and are of pairwise infinite Hausdorff distance in $\mathscr{C}^1(G)$.
Thus, for any $n$, $f$ induces an element of $\prod_{H \in K_n} QI(M_H, M_{\sigma(H)})$, for some $\sigma$ in the symmetric group on $K_n$.

In fact, we can say a bit more.  
Note that by Theorem \ref{Cqiinv}, any groups $H$ and $\sigma(H)$ as above must have quasi-isometric commensurizers.
Thus if we let $\{K_n^j\}_{j \in J_n}$ denote the partition of $K_n$ into collections of subgroups with quasi-isometric commensurizers, i.e. $H, H' \in K_n^j$ for some $j \in J_n$ if and only if Comm$_G(H)$ and Comm$_G(H')$ are quasi-isometric, then the permutation $\sigma$ above must be contained in $\prod_{j \in J_n}Sym(K_n^j)$, where $Sym (K_n^j)$ denotes the symmetric group on $K_n^j$.
Moreover, $f$ induces such a map for all $n$.

Thus we have the following corollary:

\begin{cor}\label{qi2}
Let $G$ be a one-ended finitely presented group and let $N= \{ 3,4,5,\ldots\} \cup \{ \infty\}$.  
For any $n \in N$, let $K_n$ be a maximal collection of two-ended subgroups of $G$ with $n$ coends that have mutually infinite Hausdorff distance, and let $\{ K_n^j\}_{j \in J_n}$ be the partition of $K_n$ into sets of subgroups of quasi-isometric commensurizers.

Then there is a canonical map 
$$QI(G) \to \left\{ \prod_{n \in N} \prod_{H \in K_n} QI(M_H, M_{\sigma(H)}) \ : \ \sigma \in \prod_{n \in N} \prod_{j \in J_n} Sym(K_n^j)\right\} .$$
\end{cor}

\begin{rmk}
In contrast to Corollary \ref{qi1}, we expect that the map given in Corollary \ref{qi2} will typically not be surjective.  
\end{rmk}

\begin{rmk}
As was the case in Corollary \ref{qi1}, the kernel of the map given in Corollary \ref{qi2} is not hard to describe at the following level:  a quasi-isometry $f \co \mathscr{C}^1(G) \to \mathscr{C}^1(G)$ is in this kernel if and only if, for each two-ended subgroup $H$ with at least three coends in $G$, there is a constant $y \geq 0$ such that, for all $g \in$ Comm$_G(H)$, $d_{Haus}(f(gH), gH)<y$.
\end{rmk}


\section{When $F_n \rtimes \mathbb{Z}$ is quasi-isometric to $F_n \times \mathbb{Z}$} \label{semi_section}

Our final application of Theorem \ref{Cqiinv} characterizes when a semi-direct product of a free group with $\mathbb{Z}$ is quasi-isometric to the direct product of the two groups.  
Specifically, we shall prove the following.

\begin{cor}\label{semi}
Let $F_n$ be the free group on $n$ generators for any $n \geq 1$.
Then $F_n \rtimes \mathbb{Z}$ is quasi-isometric to $F_n \times \mathbb{Z}$ if and only if it is virtually $F_n \times \mathbb{Z}$.
\end{cor}

Related results have been proven by Bridson and his coauthors.
In Proposition 3.7 of \cite{AlonsoBridson}, Alonso and Bridson show that a group $G$, that is an extension of the form $1 \to H \to G \to Q \to 1$ defined by a function $\varphi \co Q \to Aut(H)$ and a cocycle $f \co Q \times Q \to H$, is quasi-isometric to $H \times Q$ if $\varphi$ and $f$ have finite images.

In \cite{Bridson_optimal_isoper_ab_by_free}, Bridson considers groups of the form $A \rtimes F$ for $A$ finitely generated and abelian, and $F$ finitely generated and free.  
By characterizing when the Dehn functions of these groups are polynomial and determining their degrees in those cases, Bridson is able to give necessary conditions for these groups to be quasi-isometric to one another.
In \cite{BridsonGersten_optimal_isoper_torus_bundles}, Bridson and Gersten further analyze the groups of the form $\mathbb{Z}^n \rtimes \mathbb{Z}$, and find further necessary (and in special cases, sufficient) conditions for two such groups to be quasi-isometric (see section 5 of \cite{BridsonGersten_optimal_isoper_torus_bundles}).

We shall begin the proof of Corollary \ref{semi} with a few standard lemmas.

\begin{lem}\label{free_gp_lemma}
Let $F_n$ be the free group on $n$ generators, $n \geq 1$, and let $H$ be an infinite cyclic subgroup of $F_n$.  Then Comm$_{F_n}(H)$ is the unique maximal infinite cyclic subgroup of $F_n$ that contains $H$.
\end{lem}

\begin{proof}
Let $\mathscr{C}^1(F_n)$ be the Cayley graph of $F_n$ with respect to the standard presentation, thus $\mathscr{C}^1(F_n)$ is a regular $2n$-valent tree on which $F_n$ acts simplicially, freely and with quotient a bouquet of $n$ copies of $S^1$.

Let $h$ denote a generator of $H$.  As $h$ acts freely on $\mathscr{C}^1(F_n)$, $h$ has an axis, which we shall denote by $\gamma$.  Let $H'$ denote the subgroup of $F_n$ that preserves $\gamma$, so $H \subseteq H'$.  As $F_n$ is torsion free, it follows that $H'$ is infinite cyclic.  
Note that any infinite cyclic subgroup of $G$ containing $H$ must preserve $\gamma$ (see Proposition II.6.2(2) of \cite{BridsonHaefliger}), hence $H'$ must be the unique maximal infinite cyclic subgroup of $F_n$ that contains $H$.

Let $h'$ denote a generator of $H'$, and fix $k \in \mathbb{Z}$ such that $h = (h')^k$.  It is clear that $h'$ normalizes $H$, and hence that $H' \subseteq $ Comm$_{F_n}(H)$.

On the other hand, suppose that $g \in $ Comm$_{F_n}(H)$, hence there exist nonzero integers $n,m$ such that $h^n = g h^m g^{-1}$.  
It follows that the axis of $gh^mg^{-1}$ equals $\gamma$.  
On the other hand, the axis of $gh^mg^{-1}$ must be $g \gamma$ (see again Proposition II.6.2(2) of \cite{BridsonHaefliger}), thus $g\gamma = \gamma$ so $g \in H'$.  Hence $H' = $ Comm$_{F_n}(H)$.
\end{proof}

\begin{lem}\label{roots_in_fn}
Let $F_n$ be the free group on $n$ generators with $n \geq 1$, and suppose that $f_1, f_2 \in F_n$ and $k$ is an integer, $k \neq 0$, such that $f_1^k = f_2^k$.  Then $f_1 = f_2$.
\end{lem}

\begin{proof}
We may assume that $k>0$.
For each $i$, note that $f_i$ can be expressed uniquely as a minimal word of the form $u_iw_iu_i^{-1}$, where $w_i$ is cyclically reduced.  Thus, for any positive integer $j$, $f_i^j$ is represented by the minimal word $u_iw_i^ju_i^{-1}$, and thus the word length of $f_i^j$ (with respect to the standard generating set for $F_n$) is equal to $l(u_i) + j\cdot l(w_i) + l(u_i^{-1}) = 2l(u_i) + j\cdot l(w_i)$.  

If $f_1^k = f_2^k$ then, for any integer $m$, $f_1^{km} = f_2^{km}$.  Hence
$$2l(u_1) + mk\cdot l(w_1) = 2l(u_2) + mk\cdot l(w_2)$$
for any integer $m$, and it follows that $l(w_1) = l(w_2)$ and $l(u_1) = l(u_2)$.  As each $u_iw_i^ku_i^{-1}$ is a reduced word, it follows that $u_1 = u_2$, and $w_1 = w_2$, thus $f_1 = f_2$.
\end{proof}

We can now prove the main fact needed for Corollary \ref{semi}:

\begin{prop}\label{infinite_prop}
Let $n \geq 1$, let $\alpha \co \mathbb{Z} \to Aut(F_n)$ be a homomorphism, let $\pi \co Aut(F_n) \to Out(F_n)$ be the quotient homomorphism, let $G$ denote the semidirect product $F_n \rtimes_\alpha \mathbb{Z}$ and let $H$ be an infinite cyclic subgroup of $G$.

If $\pi \alpha$ has infinite image, then Comm$_G(H)$ is of infinite index in $G$.
\end{prop}

\begin{proof}
The proposition is vacuously true in the case that $n=1$, so suppose that $n>1$.

Let $S_F = \{ f_1, f_2, \ldots , f_n\}$ be the standard generating set for $F_n$.  Let $A$ denote the infinite cyclic subgroup of $G$ such that $G = F_n \rtimes_\alpha A$, and let $t$ denote a generator of $A$.
For any $b \in A$, let $\alpha_b$ denote the image of $b$ under $\alpha$.

Recall that $G$ is generated by $S_F \cup A$, and that, for any $a \in A$ and $f \in F_n$, we have that 
$$af = \alpha_a(f)a \  {\rm and } \  fa=a\alpha_{a^{-1}}(f).$$
Thus any $g \in G$ can be written uniquely in the form $fa$, for some $f \in F_n$ and $a \in A$.

Let $H = \langle h \rangle$ be an infinite cyclic subgroup of $G$.  If $h \in F_n$, then Lemma \ref{free_gp_lemma} implies that Comm$_{F_n}(H)$ is an infinite cyclic subgroup of $F_n$.  Thus $[F_n:$ Comm$_{F_n}(H)] = \infty$.  

Note that Comm$_{F_n}(H) = $ Comm$_G(H) \cap F_n$, hence in this case Comm$_G(H)$ is of infinite index in $G$.

Suppose next that $h = a \in A$.  If there is some  $g $ contained in $F_n \cap $Comm$_G(H)$, then there must be nonzero integers $k, k'$ such that
$$a^k = ga^{k'}g^{-1}.$$
Then $a^k = g\alpha_{a^{k'}}(g^{-1})a^{k'}$.  As $g\alpha_{a^{k'}}(g^{-1}) \in F_n$, $a^k, a^{k'} \in A$ and $F_n \cap A = \o$, it follows that $k=k'$, and that $g = \alpha_{a^k}(g)$.

Note that $|$Im$(\pi \alpha )|= \infty$ implies that $|$Im$(\alpha)|=\infty$, and we claim that $|$Im$(\alpha)|=\infty$ implies that there exists some $g_0 \in F_n$ such that $\alpha_b(g_0) \neq g_0$ for all nontrivial $b \in A$.  For if no such $g_0$ did exist, then, for each $i=1, 2, \ldots n$, there would be some positive integer $k_i$ such that $\alpha_{t^{k_i}}(f_i) = f_i$.  Let $k$ equal the least common multiple of the $k_i$'s, and it follows that $\alpha_{t^k}$ fixes each generator of $F_n$.  Thus $\alpha_{t^k} = id$ and $|$Im$(\alpha)| \leq k$, a contradiction.

Hence there is some $g_0 \in F_n$ that is such that $\alpha_b(g_0) \neq g_0$ for all nontrivial $b \in A$, and hence $g_0 \notin $ Comm$_G(H)$.  
It follows from Lemma \ref{roots_in_fn} that, for each nonzero integer $r$ and for all nontrivial $b \in A$, $\alpha_b(g_0^r) \neq g_0^r$, and hence $g_0^r \notin $ Comm$_G(H)$.
Thus the elements $g_0^r$, $r \in \mathbb{Z}$, are contained in distinct cosets of Comm$_G(H)$ in $G$, and so $[G:$ Comm$_G(H)]=\infty$ in this case as well.

The general case remains:  suppose that $h$ is neither contained in $F_n$ nor contained in $A$.  Then there are nontrivial $f \in F_n, a \in A$ such that $h = fa$.
Suppose that there is some element $g \in F_n \cap$Comm$_G(H)$, so there exist nonzero integers $k,k'$ such that 
$$(fa)^k = g(fa)^{k'}g^{-1}.$$
Then there exist $f', f'' \in F_n$ such that the above equation becomes
$$f'a^k = f''a^{k'},$$
and hence $k=k'$ as in the previous case.

For any integer $i$, define $m_i = (fa)^ia^{-i}$, and note that $m_ia^i = (fa)^i$, so in particular, $m_i \in F_n$ and, when $i$ is nonzero, $m_i$ is nontrivial.
We claim that, for any $g \in F_n$, $(fa)^k = g(fa)^{k}g^{-1}$ if and only if $\alpha_{a^k}(g) = m_k^{-1}gm_k$.
To see this, note that we have: 
$$(fa)^k = g(fa)^{k}g^{-1}$$
$$\Leftrightarrow m_ka^k = gm_ka^kg^{-1} = gm_k\alpha_{a^k}(g^{-1})a^k$$
$$\Leftrightarrow m_k = gm_k\alpha_{a^k}(g^{-1})$$
$$\Leftrightarrow \alpha_{a^k}(g) = m_k^{-1}gm_k$$
as desired.

Now suppose that there is some nontrivial $g \in F_n$ that is not contained in Comm$_G(H)$.  Then there is no nonzero $k$ such that $\alpha_{a^k}(g) = m_k^{-1}gm_k$.
Suppose that some nontrivial power of $g$, $g^r$, was contained in Comm$_G(H)$.  Then, for some nonzero value of $k$, we would have
$$\alpha_{a^k}(g^{r}) = m_k^{-1}g^rm_k$$
and hence
$$[\alpha_{a^k}(g)]^{r} = [m_k^{-1}gm_k]^r.$$
By Lemma \ref{roots_in_fn}, $\alpha_{a^k}(g) = m_k^{-1}gm_k$, so we have reached a contradiction.

Thus if $g \notin$ Comm$_G(H)$, then for each nonzero integer $r$, $g^r \notin $ Comm$_G(H)$.  As in the previous case, it follows that the elements $g^r$ are contained in distinct cosets of Comm$_G(H)$ in $G$ and $[G:$ Comm$_G(H)]=\infty$ as desired.

It remains to consider the case that $F_n \subset $ Comm$_G(H)$.
In this case, we have that in particular $f_i \in $ Comm$_G(H)$ for each $1 \leq i \leq n$; let $k_i$ be a nonzero integer such that 
$$(fa)^{k_i} = f_i(fa)^{k_i}f_i^{-1}.$$
Note that we can take each $k_i$ to be positive.

Let $k$ equal the least common multiple of the $k_i$'s.  Thus, for each $i$, 
$$(fa)^{k} = f_i(fa)^{k}f_i^{-1}$$
and hence $\alpha_{a^k}(f_i) = m_k^{-1}f_im_k$ for each $i$.  It follows that $\alpha_{a^k}$ acts on $F_n$ by conjugation by $m_k$ and hence that $|$Im$(\pi \alpha)|$ is finite.
It follows that $|Im(\pi \alpha)|=\infty$ implies that $[G:$ Comm$_G(H)]=\infty$ for any infinite cyclic subgroup $H$ of $G$.
\end{proof}

We can now prove the main result in this section.
\\

\noindent {\bf Corollary \ref{semi}.}{\it
Let $F_n$ be the free group on $n$ generators for any $n \geq 1$.
Then $F_n \rtimes \mathbb{Z}$ is quasi-isometric to $F_n \times \mathbb{Z}$ if and only if it is virtually $F_n \times \mathbb{Z}$.
}
\\

\begin{proof}
Let $\alpha \co \mathbb{Z} \to Aut(F_n)$ be the homomorphism such that $F_n \rtimes \mathbb{Z}=F_n \rtimes_\alpha \mathbb{Z}$.
Let $A = \langle t \rangle$ denote the infinite cyclic subgroup in the statement of the corollary, let $F_n = \langle f_1, \ldots , f_n\rangle$, let $G_\rtimes$ denote $F_n \rtimes_\alpha A$ and let $G_\times$ denote $F_n \times A$.

Certainly if $G_\rtimes$ is virtually $G_\times$, then the groups are quasi-isometric.
We claim that $|$Im$(\pi \alpha)|<\infty$ implies that $G_\rtimes$ is virtually $G_\times$, where $\pi$ denotes the quotient homomorphism from $Aut(F_n)$ to $Out(F_n)$ as in the previous proposition.

For suppose that $|$Im$(\pi \alpha)|<\infty$, so that there is some nonzero $k$ such that $\alpha_{t^k}$ is an inner automorphism of $F_n$, with $m \in F_n$ such that $\alpha_{t^k}$ takes any $f \in F_n$ to $mfm^{-1}$.
Let $G_k$ denote the subgroup of $G_\rtimes$ that is generated by $f_1, \ldots , f_n,$ and $t^k$, and we have that $G_k \cong G_\times$, by an isomorphism from $G_\times$ to $G_k$ that acts as the identity on $F$, and takes $t$ to $m^{-1}t^k$.

To see that $G_k$ is of finite index in $G_\rtimes$, recall that any element of $G_\rtimes$ can be written uniquely in the form $ft^l$ for some $f \in F_n$ and $l \in \mathbb{Z}$.  It follows that any $g \in G_\rtimes$ is of the form $(ft^{mk})(t^i)$, for some $m \in \mathbb{Z}$ and $0 \leq i < k$.  Thus $G_\rtimes = \cup_{i=0}^{k-1}G_kt^i$, so $G_k$ is of index no more than $k$ in $G_\rtimes$, as desired.

Thus it remains to show that $G_\rtimes$ and $G_\times$ are {\em not} quasi-isometric in the case that $|Im(\pi \alpha)|=\infty$.
Recall from Proposition \ref{infinite_prop} that in this case, for each subgroup $H \cong \mathbb{Z}$ of $G_\rtimes$, $[G_\rtimes:$ Comm$_{G_\rtimes}(H)]=\infty$.

It is clear that the quotient by $A$ of the Cayley graph of $G_\times$ has infinitely many ends, and hence $\tilde{e}(G_\times, A) = \infty$.  Note also that $G_\times =$ Comm$_{G_\times}(A)$.
Thus if there were a quasi-isometry from $G_\times$ to $G_\rtimes$, then it would follow from Theorem \ref{Cqiinv} that there is some $z>0$ and some infinite cyclic subgroup $H$ of $G_\rtimes$ such that the $z$-neighborhood of Comm$_{G_\rtimes}(H)$ in the Cayley graph of $G_\rtimes$ is equal to the whole Cayley graph.  This, however, would imply that Comm$_{G_\rtimes}(H)$ is of finite index in $G_\rtimes$, a contradiction.
\end{proof}

\appendix 

\section{Appendix}

This appendix contains a proof of Theorem \ref{2.1}, which is provided to clarify the argument given in \cite{papa_qlines}.
A few definitions are required first.

Let $(X,d)$ be a metric space.  If $A, B \subset X$ then let $d_{inf}(A,B)$ denote 
$$\inf_{a \in A, b \in B} d(a,b).$$
We note that when $A$ and $B$ are not single points, this function does not necessarily obey the triangle inequality.
Nonetheless, this notation will be useful to us.

If $A \subset X$, let $fr(A)$ denote the frontier of $A$ in $X$, so $fr(A) = \overline{A} \cap \overline{X-A}$.  

Let $A \subset B \subset X$, with $B$ path connected, and we will let $CH(A,B)$ denote the convex hull of $A$ in $B$, with respect to the path metric in $B$ induced by the geometry of $X$. 

Fix $A_0>0$, and recall from the proof of Lemma \ref{char_fg_subgps} that a sequence $s_0, s_1, \ldots , s_n$ of points in $X$ is an $A_0$-chain from $s_0$ to $s_n$ if $d(s_i, s_{i+1})<A_0$ for all $i$.

Let $L'$ be a $(\phi, N)$ quasi-line with associated line $l'$.  
Following \cite{papa_qlines}, if $x = l'(x_0)$ and $y = l'(y_0)$, then we will denote by $[x,y]_{l'}$ the segment $l'([x_0, y_0])$ in $L'$, and we will write $x<y$ if $x_0 < y_0$.
Moreover, for arbitrary $x,y \in L'$, we shall denote by $[x,y]_{L'}$ the ``thickened segment'' $\{ z \in L' : d(z, [x_0, y_0]_{l'} \leq N\}$, where $x_0$ ($y_0$ respectively) is a point in $l'$ that can be connected to $x$ ($y$ respectively) by a path in $L'$ of length no more than $N$.
Let length$([x,y]_{L'})$ denote length$([x_0, y_0])$.

Recall that we say that two points $a$ and $b$ are {\it $K$-separated} by a quasi-line $L_1$ if $B_K(a)$ and $B_K(b)$ are in different components of $X-L_1$.
\\

\noindent {\bf Theorem \ref{2.1}} {\it
\cite{papa_qlines} 
Let $G$ be a one-ended, finitely presented group, and let $L, L_1$ be $(\phi',N')$ quasi-lines in $\mathscr{C}^1(G)$ that satisfy $iness(m_1')$.
Suppose that $L$ is 3-parting.

Then there is some $K = K(G, \phi', N', m_1')$ such that no two points $a,b \in L$ are $K$-separated by $L_1$.
}
\\

\begin{proof}
Let $\mathscr{C}^1 = \mathscr{C}^1(G)$ and let $\mathscr{C}^2 = \mathscr{C}^2(G)$.
Recall that we may think of $\mathscr{C}^1$ as the 1-skeleton of $\mathscr{C}^2$, and that there is a uniform bound to the size of the 2-cells of $\mathscr{C}^2$.
We will begin by showing that it suffices to work in $\mathscr{C}^2$, instead of $\mathscr{C}^1$.

We note that there exist $\Lambda \geq 1,C \geq 0$, depending only on $G$ (and its associated finite presentation), such that the inclusion of $\mathscr{C}^1$ into $\mathscr{C}^2$ is a $(\Lambda, C)$ quasi-isometry.

Suppose that $L'$ is a $(\phi', N')$ quasi-line in $\mathscr{C}^1$ that satisfies $iness(m_1')$ and is $n$-parting in $\mathscr{C}^1$, for some $n \geq 1$, and consider $L'$ as a subset of the 1-skeleton of $\mathscr{C}^2$.
By Lemmas \ref{iness_preserved_under_qi} and \ref{ess_comps_go_to_ess_comps}, there is a neighborhood of $L'$ in $\mathscr{C}^2$, which we shall call $\tilde{L}'$, that is an $n$-parting $(\phi'', N'')$ quasi-line satisfying $iness(m_1)$, where $\phi'',N''$ and $m_1$ depend only on $G, \phi', N'$ and $m_1'$.
We can further assume that $\tilde{L}'$ is a subcomplex of $\mathscr{C}^2$.

Now recall the quasi-lines $L$, $L_1$ in $\mathscr{C}^1$.
Suppose that there is some $K$ such that no two points in $\tilde{L}$ can be $K$-separated by $\tilde{L}_1$ in $\mathscr{C}^2$.
Recall that inclusion is a $(\Lambda, C)$ quasi-isometry from $\mathscr{C}^1$ to $\mathscr{C}^2$.
Thus no two points in $L$ can be $(\Lambda K +C + N'')$-separated by $L_1$ in $\mathscr{C}^1$, so the theorem will follow, with $(\Lambda K +C + N'')$ replacing $K$.
Thus we shall no longer work with $\mathscr{C}^1$, but with $\mathscr{C}^2$ instead, together with $\tilde{L}$ and $\tilde{L}_1$, which are both $(\phi'', N'')$ quasi-lines.
\\

We shall next reduce to the case that $\tilde{L}$ is simply connected.
Let $\Delta'$ denote a collection of regular polygons attached to $\mathscr{C}^2$ along all simple closed edge paths of $\tilde{L}$ of length no more than $[\phi'' ( 2N''+1) + 2N'' + 1]$.
Then the methods from the proof of Lemma \ref{L''_simp_conn} show that $\tilde{L} \cup \Delta'$ is simply connected in $\mathscr{C}^2 \cup \Delta'$.
Let $X$ denote $\mathscr{C}^2 \cup \Delta'$, and we note that $X$ is simply connected.

By construction $\tilde{L}$ is $3$-parting in $\mathscr{C}^2$, and it follows that $\tilde{L} \cup \Delta'$ is 3-parting in $X$.
Also we have that $\tilde{L} \cup \Delta'$ is a $(\phi, N)$ quasi-line, for some $\phi$ and $N$ depending on $\phi''$ and $N''$.

Consider the union of $\tilde{L_1}$ with any cells of $\Delta'$ that meet it.  
This union has at least as many essential complementary components in $X$ as $L_1$ does in $\mathscr{C}^1$, and is also a $(\phi, N)$ quasi-line.
By abuse of notation, we shall refer to this union containing $\tilde{L_1}$ as $L_1$, and we shall refer to $\tilde{L} \cup \Delta'$ as $L$ throughout the following.

Thus $L$ and $L_1$ are both $(\phi, N)$ quasi-lines, and $L$ is 3-parting.
We note that both $L$ and $L_1$ satisfy $iness(m_1)$.
Moreover, as $L$ is simply connected and $X$ is the union of a Cayley complex of $G$ and some 2-cells of bounded size, the methods in the proof of Lemma \ref{compl_has_fin_many_comps} show that $L$ and $L_1$ satisfy $ess(m_0)$ for some $m_0 \geq 0$.

We claim that it suffices to consider $L$ and $L_1$ in $X$ to prove the theorem.
For the inclusion of $\mathscr{C}^2$ into $X$ is a $(\Lambda', C')$ quasi-isometry for some $\Lambda' \geq 1, C' \geq 0$ depending on $G$ and $[\phi'' ( 2N''+1) + 2N'' + 1]$.
So if we can find a value of $K$ such that no two points in $L$ can be $K$-separated by $L_1$ in $X$, then it follows that no two points of $\tilde{L}$ can be $(\Lambda' K + C')$-separated by $\tilde{L}_1$ in $\mathscr{C}^2$, so the theorem holds.

We shall make one final reduction before beginning our argument.
Let $l$ be the line associated to $L$ and let $l_1$ be the line associated to $L_1$.  
If we show that no two points $a$ and $b$ of $L$ can be $K$-separated by $L_1$ in the case that $a$ and $b$ are vertices of $l$, then our result follows: in order to get the constant $K$ for arbitrary $a, b \in L$, it suffices to add $(N+1)$ to the constant we find, since any point in $L$ is a distance of less than $(N+1)$ from a vertex of $l$.

Thus we shall prove that, given the new quasi-lines $L$ and $L_1$ defined above in the 2-dimensional CW complex $X$, there exists some $K$ such that no two vertices $a,b$ in $l$ can be $K$-separated by $L_1$.
\\

We shall make use of winding numbers (of curves about points in the disk) in our argument.  See, for instance, Chapter 10 of \cite{Munkres}.
In particular, we will need the following fact.

\begin{lem}
Suppose that $\alpha, \beta, \gamma$ are oriented curves in a 2-disk $D^2$, and let $-\gamma$ denote the curve $\gamma$ with the opposite orientation.
Suppose further that $\alpha \cup \beta$, $\alpha \cup \gamma$, and $\beta \cup -\gamma$ are closed oriented curves, and that $v$ is a point in $D^2$ that is not met by $\alpha, \beta$ or $\gamma$.
For any oriented closed curve $\delta \subset D^2-\{ v\}$, let $w_v(\delta)$ denote the winding number of $\delta$ about $v$.

Then 
$$w_v (\alpha \cup \beta) = w_v(\alpha \cup \gamma) + w_v(\beta \cup -\gamma).$$
\end{lem}

We will also make use of the next lemma.
Stated in the setting of $\mathscr{C}^1$, Lemma 1.9 of \cite{papa_qlines} may be restated as the following.  
\begin{lem}\label{1.9} \cite{papa_qlines}
Let $G$ be a finitely presented group and let $L'$ be a 1-parting $(\phi, N)$ quasi-line in $\mathscr{C}^2(G)$ that satisfies $ess(m_0)$.  
Given any $r_1>0$, there is some $r_2 = r_2(G, \phi, N, r_1, m_0) >\max \{r_1, m_0\}$ such that, for any vertices $a< b$ in $L'$ with length$([a,b]_{L'})>2r_2$, 
and for any essential component $Y$ of $\mathscr{C}^2(G)-L'$, there is a simplicial path $p$ joining $a$ to $b$ in $Y \cup L'$, such that
\begin{enumerate}
\item{$p \cap N_{r_1}( [a+r_2, b-r_2]_{L'}) =$ \o, and }
\item{$p \subset N_{r_2}( [a,b]_{L'} ) $.}
\end{enumerate}
\end{lem}

As $X$ is a Cayley complex, together with additional 2-cells of bounded size, we note that there is an increasing function $i \co \mathbb{R}_+ \to \mathbb{R}_+$ such that, for any simplicial path in $X$ of length less than or equal to $r$, there is a simplicial null-homotopy of that path that is contained in the $i(r)$-neighborhood of the path.
The methods of Lemma \ref{L''_simp_conn}, together with this observation, imply the following lemma.

\begin{lem}\label{coarsely_sc}
There exists a constant $M = M(X, \phi'', N'', i)$ such that, for any closed curve $\gamma$ in $L_1$, there is a null-homotopy of $\gamma$ that is contained in $N_M(\gamma) \subset N_M(L_1)$.
\end{lem}

We shall now set up some constants so that we can show that no two vertices in $l$ can be $K$-separated by $L_1$, for a value of $K$ to follow.
Let $R > (\frac{1}{2}\phi (2N) + N + M)$, 
let $K_1 > [\frac{1}{2} \phi (2(N+R+1)) + (2N+R+1)+N]$, let $r_1$ be larger than $K_1+N$, and let $r_2$ be as in Lemma \ref{1.9}, with respect to $r_1$ and the other data we are working with.
Let $K>\max \{ \frac{1}{2}[\phi(K_1 + N+r_2) + K_1 + N + r_2], r_2 + R + 1 \}$.


In summary, we have the following conditions on constants:
\begin{itemize}
\item{$R> \frac{1}{2}\phi(2N)+N+M$}     
\item{$K_1 > \frac{1}{2} \phi (2(N+R+1)) + (2N+R+1)+N$}
\item{$r_1 > K_1+N$}
\item{$r_2$ is from Lemma \ref{1.9} and depends on $G, \phi, N, r_1$ and $m_0$, with $r_2>\max\{ r_1, m_0\}$}
\item{$K> \max \{ \frac{1}{2}[\phi(K_1 + N+r_2) + K_1 + N + r_2], r_2 + R + 1 \}$}
\end{itemize}
Recall also that $\phi \co \mathbb{R}_{\geq 0} \to \mathbb{R}_{\geq 0}$ is such that $\phi(t) \geq t$ for all $t \in \mathbb{R}_{\geq 0}$.

Now suppose that $a < b \in l$ are $K$-separated by $L_1$.
Let $X_1, X_2, X_3$ denote essential complementary components of $L$, and, for $i=1,2,3$, let $p_i$ denote a path from $a$ to $b$ in $X_i \cup L$ from Lemma \ref{1.9}, with respect to the constants $r_1$ and $r_2$.

We can alter each $p_i$ so that it is a simple (simplicial) path, by deleting any subpaths that begin and end at the same point.  Note that this does not alter any of the properties from Lemma \ref{1.9} that are satisfied by $p_i$.

As $a$ and $b$ are in distinct components of the complement of $L_1$ and each $p_i$ is a path from $a$ to $b$, it follows that $(p_i \cap L_1)$ is nonempty.
Let $x \in (p_i \cap L_1)$ for any $i$, and we shall show that $d_{inf}(x,L) > K_1$. 
It suffices to show that, for any point $c \in l$, $d(x, c)>K_1+N$.

As $r_1 > K_1+N$ and $x \in p_i$, if $c \in [a+r_2, b-r_2]_l$ then it is clear that $d(x, c)>K_1+N$ from the construction of $p_i$.
If $c \in [(B_{r_2}(a)\cup B_{r_2}(b)) \cap [a,b]_l]$, then note that $K > \frac{1}{2}[\phi(K_1 + N+r_2) + K_1 + N + r_2] \geq \frac{1}{2}[(K_1 + N+r_2) + K_1 + N + r_2] = K_1+N+ r_2$, and recall that $a$ and $b$ are $K$-separated by $L_1$.
Thus 
$$d (x,c) \geq d_{inf}(L_1,c) \geq d_{inf}(L_1, B_{r_2}(a) \cup B_{r_2}(b)) > K_1+N,$$
so $d(x,c) >K_1+N$.

It follows from the construction of $p_i$ that $x$ is of distance no more than $r_2$ from a point $d$ in $[a,b]_l$.
Now suppose that $c \in (-\infty , a-r_2]_l$.
If $d(x,c) \leq K_1+N$, then $d(d,c) \leq d(d,x) + d(x,c) \leq K_1 + N + r_2$, and hence length$[d,c]_l \leq \phi (K_1+N+r_2)$.
On the other hand, $x \in L_1$, so $x \notin B_K(a)$ and hence 
$$K \leq d(a,x) \leq d(a,c)+d(c,x) \leq {\rm length}[a,c]_l + (K_1 + N)$$
and
$$K \leq d(a,x)\leq d(a,d)+d(d,x) \leq {\rm length}[a,d]_l + r_2.$$
As $a \in (d,c)_l$, ${\rm length}[c,d]_l = {\rm length}[a,c]_l+{\rm length}[a,d]_l$, hence the addition of the above equations yields
$$2K \leq {\rm length}[c,d]_l + K_1 + N + r_2.$$
Combining this with the above observation that length$[d,c]_l \leq \phi (K_1+N+r_2)$, we have that
$$2K-K_1-N-r_2 \leq {\rm length}[a,d]_l \leq \phi(K_1 + N+r_2).$$

Thus $K \leq \frac{1}{2}[\phi(K_1 + N+r_2) + K_1 + N + r_2]$.  
But this contradicts our choice of $K$.
Thus $d(x,c) > K_1 + N$, and similarly if $c \in [b+r_2, \infty )_l$, then $d(x,c) > K_1+N$ as well.
Hence $d_{inf}(x, L) > K_1$.  
\\

Note that, if $i \neq j$, then $(p_i \cap p_j) \subset L$.
It follows from this and our previous observation that $p_i$ and $p_j$ will not meet in $L_1$.

For each $i$, let $q_i \co S^1 \to (X_i \cup L)$ traverse the closed simplicial path $p_i \cup [a,b]_{l}$.  
As $L$ is simply connected, Van Kampen's Theorem implies that each $X_i \cup L$ is simply connected.  Thus there is a map $g_i \co (D_i, \partial D_i) \to (X_i \cup L, q_i(S^1))$, where $D_i$ denotes a copy of the 2-disk $D^2$, and $g_i|_{\partial D_i} = q_i$.  

We shall denote by $\hat{p}_i$ the subpath of $\partial D_i$ mapped homeomorphically onto $p_i$ by $g_i$. 
Similarly let $\hat{[a,b]}$ denote the subpath of $\partial D_i$ mapped homeomorphically onto $[a,b]_{l}$, so $\partial D_i = \hat{p}_i \cup \hat{[a,b]}$. 
We may resize $D_i$ and assume that the restriction of $g_i$ to the domains $\hat{p}_i$ and $\hat{[a,b]}$ has unit speed.

Let $\hat{D} = D_1 \coprod D_2 \coprod D_3 / \sim$, where $\sim$ denotes the canonical identification of the subpaths $\hat{[a,b]}$ in $D_1, D_2, D_3$.
We shall from now on consider $D_1, D_2,$ and $D_3$ as subsets of $\hat{D}$.
Let $g$ denote the map from $\hat{D}$ into $X$ that is induced by $g_1, g_2, g_3$.  
The restriction of $g$ to $\hat{[a,b]}$ is a unit speed homeomorphism onto $[a,b]_l$; let $\hat{a}$ denote the preimage of $a$ under this restriction of $g$, and let $\hat{b}$ denote the preimage of $b$.
The restriction of $g$ to each $\hat{p}_i$ is also a unit speed homeomorphism.
As $a$ and $b$ are in different components of the complement of $L_1$, note that the preimage of $L_1$ under $g$ is a (possibly disconnected) subset of $\hat{D}$ that separates $\hat{a}$ from $\hat{b}$.

Note that each $p_i$ is simplicial and $L$ is a subcomplex of $X$, hence $\partial \hat{D}$ meets $g^{-1}(L)$ in only finitely many components.

As $L_1$ $K$-separates $a$ and $b$, and $K>R$, 
it follows that there exist points $e_1< e_2 \in (a,b)_l$ such that $B_{R}(e_1) \cap L_1 = $ \o $  = B_{R}(e_2) \cap L_1 $, and $e_1$ and $e_2$ are in different components of $X-L_1$, say $Y_1$ and $Y_2$ respectively.
Furthermore, we may choose $e_1$ and $e_2$ such that, fixing $\epsilon \ll 1 < d(e_1, e_2)$ (recalling that the edges of $X$ are of length 1), we have that for any $q \in [e_1+\epsilon, e_2-\epsilon]_l$, $B_{R}(q)$ meets $L_1$.
Let $B_i$ denote $B_{R}(e_i)$ for $i=1,2$.

As $g$ is a homeomorphism from $\hat{[a,b]}$ to $[a,b]_l$, for each $k=1,2$, $g^{-1}(e_k) \cap \hat{[a,b]}$ is one point - denote it by $\hat{e}_k$.

For any $1 \leq i < j \leq 3$, let $D_{ij}$ denote $(D_i \cup D_j) \subset \hat{D}$.  
Note that $D_{ij}$ is a copy of the disk and $\partial D_{ij} = \hat{p}_i \cup \hat{p}_j$.
Recall that $g$ was defined to take $\hat{D}$ into $X$; from now on, we shall have $g$ denote the restriction of this map to $D_{ij}$.
Note that the image of this restriction is contained in $(X_i \cup X_j \cup L)$.

Let $\hat{B_k}$ denote the connected component of $g^{-1}(B_k)$ that contains $\hat{e_k}$.
As $L_1$ separates $B_1$ from $B_2$, it follows that $g^{-1}(L_1)$ separates $\hat{B_1}$ from $\hat{B_2}$.
Let $A_{ij}$ denote the connected component of $g^{-1}(Y_1)$ in $D_{ij}$ that contains $\hat{B_1}$.  
Thus there is some connected component, $\hat{\delta}_{ij}$, of the frontier of $A_{ij}$ that separates $\hat{B_1}$ from $\hat{B_2}$.
Note that $\hat{\delta}_{ij}$ is contained in $g^{-1}(L_1)$.

We may further take $\hat{\delta}_{ij}$ to be a simple curve by removing subpaths that begin and end at the same point, while maintaining that $\hat{\delta}_{ij}$ separates $\hat{B_1}$ from $\hat{B_2}$.

A priori, $\hat{\delta}_{ij}$ may or may not be a closed curve.

\begin{lem}\label{scc_case}
$\hat{\delta}_{ij}$ is not a closed curve, but rather must be an arc.

\end{lem}

The proof of this lemma follows the completion of this argument.


Thus $\hat{\delta}_{ij}$ is a curve with two distinct endpoints in $\hat{p}_i \cup \hat{p}_j$.
Let $\hat{x}, \hat{y}$ denote these endpoints.
We may assume that these endpoints are exactly the intersection of $\hat{\delta}_{ij}$ with $\hat{p}_i \cup \hat{p}_j$.

For any $c_1, c_2 \in l$ such that $a \leq c_1 \leq c_2 \leq b$, let $\hat{[c_1, c_2]}$ denote the arc contained in $\hat{[a,b]}$ that is mapped homeomorphically onto $[c_1, c_2]_l$ by $g$, and define the notation $\hat{[c_1, c_2)}, \hat{(c_1, c_2]},$ and $\hat{(c_1, c_2)}$ simiarly.
Thus $\hat{[e_1, e_2]}$ is a path in $D_{ij}$ from $\hat{e_1}$ to $\hat{e_2}$, and hence $\hat{\delta}_{ij}$ must meet $\hat{[e_1, e_2]}$.

Let $\hat{\gamma_x}$ be the subpath of $\hat{\delta}_{ij}$ from $\hat{x}$ to $g^{-1}([e_1, e_2]_l)$, that does not meet $g^{-1}([e_1, e_2]_l)$ in its interior, and define $\hat{\gamma_y}$ similarly.
Let $x$ denote $g(\hat{x})$, let $y$ denote $g(\hat{y})$, let $\gamma_x$ denote the path $g(\hat{\gamma_x})$ and let $\gamma_y$ denote $g(\hat{\gamma_y})$.
Note that $\gamma_x$ ($\gamma_y$ respectively) is a path from $x$ ($y$ respectively) to $[e_1, e_2]_l$, that does not meet $[e_1, e_2]_l$ except at one endpoint.

The path $\gamma_x$ is contained in $(X_i \cup X_j \cup L)$, since the image of $g$ is contained in this union.
As $\hat{\delta}_{ij} \subset g^{-1}(L_1)$, $\gamma_x$ is contained in $L_1$.

Moreover as $L_1 \cap B_1 = $ \o, $\gamma_x$ is contained in $X-B_1$.
Thus $\gamma_x$, and similarly $\gamma_y$, is contained in $[L_1\cap (X_i \cup L \cup X_j)\cap (X-B_1)]$.

Recall that, by definition, the $R$-ball about each point in $[e_1+ \epsilon, e_2-\epsilon]_l$ meets $L_1$.
Thus the $(R+\epsilon)$-ball about each point of $[e_1, e_2]_l$ meets $L_1$, and the $(N+R+\epsilon)$-ball about each point of $[e_1, e_2]_l$ meets $l_1$.
Let $\pi_\chi \co [e_1, e_2]_l \to l_1$ denote nearest point projection.
Then Im$(\pi_\chi)$ contains a $2(N+R+\epsilon)$-chain from $\pi_\chi(e_1)$ to $\pi_\chi (e_2)$, with consecutive points in the chain connected in $l_1$ by paths of length no more than $\phi (2(N+R+\epsilon))$.
Let $\chi$ denote $CH ($Im$(\pi_\chi),l_1)$, and it follows that $\chi$ is contained in the $[\frac{1}{2} \phi (2(N+R+\epsilon)) + (N+R+\epsilon)]$-neighborhood of $[e_1, e_2]_l$.
Let $A$ denote $[\frac{1}{2} \phi (2(N+R+\epsilon)) + (N+R+\epsilon)]$.

As $x,y \in (p_i \cup p_j) \cap L_1$, we have from an argument above that $d_{inf}(x,L), d_{inf}(y,L)>K_1$.
Let $\pi \co L_1 \to l_1$ denote nearest point projection.
Since $L_1$ is a $(\phi, N)$ quasi-line, $d(z, \pi(z)) \leq N$ for any $z \in L_1$.
We claim that $\pi(x), \pi(y) \notin \chi$.
To see this, we have that
$$d_{inf}(\pi(x), \chi) \geq d_{inf}(B_N(x), N_A(L)) \geq K_1 - N - A.$$
As $K_1 > N+A$, we have that $d_{inf}(\pi(x), \chi)>0$.
Similarly $d_{inf}(\pi (y), \chi) >0$, and thus $\pi(x), \pi(y) \notin \chi$. 

Let $\gamma_x'$ denote the component of $CH(\pi (\gamma_x),l_1)-\chi$ that contains $\pi(x)$, so $\gamma_x'$ is an arc in $l_1$ from $\pi (x)$ to $\chi$.
Define $\gamma_y'$ simiarly.

A less immediate result is the following.

\begin{lem}\label{second_lemma}
$\pi(x) \notin \gamma_y'$ and $\pi(y) \notin \gamma_x'$.

\end{lem}

A proof for this lemma is given after the completion of this argument.

So $\gamma_x', \gamma_y'$ are segments in $l_1$ that both meet $\chi$ in precisely one point.
The subspace $l_1 - \chi$ consists of two components, and it follows from Lemma \ref{second_lemma} that $\gamma_x'$ is contained in the closure of one and $\gamma_y'$ in the closure of the other.

Our argument up to this point has been on the restriction of $g$ to $D_{ij}$.  
Note that, while $\hat{\delta}_{ij}$ depended in $i,j$, the points $\hat{e_1}$ and $\hat{e_2}$ did not, nor did the segment $\chi \subset l_1$.

Let $k \in \{ 1,2,3\}$, $k \neq i,j$, and we can run the same argument in $D_{ik}$.
This will result in points $z,w \in (X_i \cap p_i) \cup (X_k \cap p_k)$, and segments $\gamma_z', \gamma_w'$ from $\pi(z)$, $\pi(w)$ respectively, to $\chi$, with $\gamma_z'$ and $\gamma_w'$ contained in the closures of different components of $l_1-\chi$.

Without loss of generality, we may suppose that $\gamma_x'$ and $\gamma_z'$ are contained in the closure of the same component of $l_1-\chi$, and $\gamma_y'$ and $\gamma_w'$ are contained in the closure of the other component.
We would like to say that $x$ and $z$, or $y$ and $w$, are contained in distinct elements of $\{ X_1, X_2, X_3\}$.
If this is not the case, then we must have that all four points are contained in $X_i$.
Then we will run our above argument for $D_{jk}$.
This will result in another pair of points, say $\overline{z}, \overline{w}$, and segments $\gamma_{\overline{z}}', \gamma_{\overline{w}}'$.
Without loss of generality, we shall suppose that $\gamma_x'$ and $\gamma_{\overline{z}}'$ are contained in the closure of the same component of $l_1-\chi$, and we must have that $x$ and $\overline{z}$ are contained in different $X_l$'s.

Thus, without loss of generality, we shall assume that $x$ and $z$ are contained in distinct components of the complement of $L$, say with $x \in X_l$ and $z \in X_{l'}$.

Note that $\gamma_x'$ is made up of a $2N$-chain in $\pi(\gamma_x) \subset l_1$, together with connecting segments in $l_1$ of length no more than $\phi (2N)$.
It follows that $\gamma_x'$ is contained in the $(\frac{1}{2}\phi (2N) + N)$-neighborhood of $\gamma_x$.
Similarly $\gamma_{z}'$ is contained in the $(\frac{1}{2}\phi (2N) + N)$-neighborhood of $\gamma_{z}$.
As $\gamma_x \subset (X_l \cup L)$, we have that $\gamma_x' \subset N_{\frac{1}{2}\phi(2N) + N}(X_l \cup L)$, and as $\gamma_{z} \subset (X_{l'} \cup L)$, we have that $\gamma_{z}' \subset N_{\frac{1}{2}\phi(2N) + N}(X_{l'} \cup L)$.

Recall that $d_{inf}(x,L)>K_1$, and we also have that $d_{inf}(z, L)>K_1$.
Thus $d_{inf}(\pi(x), L)>(K_1-N)$ and $d_{inf}(\pi (z), L)>(K_1-N)$.
Recall that $\pi(x) \in [\gamma_x' \cap X_l]$, and $\pi (z) \in [\gamma_z' \cap X_{l'}]$, and hence
$$d_{inf}(\pi(x), \gamma_z') \geq d_{inf}([X_l - N_{K_1-N}(L)], N_{\frac{1}{2}\phi(2N) + N}(X_{l'} \cup L))$$
and
$$d_{inf}(\pi(z), \gamma_x') \geq d_{inf}([X_{l'} - N_{K_1-N}(L)], N_{\frac{1}{2}\phi(2N) + N}(X_l \cup L)).$$
As $K_1-N > (\frac{1}{2}\phi(2N) + N)$, we have that $d_{inf}(\pi(x), \gamma_z'), d_{inf}(\pi(z), \gamma_x') >0$, thus  $\pi(x) \notin \gamma_{z}'$ and $\pi(z) \notin \gamma_x'$. 
But $\pi(x) \in \gamma_x'$, $\pi(z_s) \in \gamma_{z_s}'$, and both $\gamma_x'$ and $\gamma_{z_s}'$ are segments in the same component of $\overline{l_1-\chi}$, and both contain the endpoint of that component.

As $l_1$ is an embedded copy of the real line, this situation is impossible.
Thus we have reached a contradiction, so the conclusion of the Theorem follows.
\end{proof}

It remains to prove Lemmas \ref{scc_case} and \ref{second_lemma}.

\noindent {\bf Lemma \ref{scc_case}.
{\it
Let $g$ be a continuous map from $D_{ij}$ to $X$ as defined in the proof of Theorem \ref{2.1}, with $A_{ij}$ the component of $g^{-1}(Y_1)$ that contains $\hat{B_1}$ and $\hat{\delta}_{ij}$ the component of the frontier of $A_{ij}$ that separates $\hat{B_1}$ from $\hat{B_2}$, made to be simple by the removal of loops that do not change that $\hat{\delta}_{ij}$ separates $\hat{B}_1$ from $\hat{B}_2$.

Then $\hat{\delta}_{ij}$  not a closed curve.
}

\begin{proof}
Suppose that $\hat{\delta}_{ij}$ is a closed curve.

Then $\hat{\delta}_{ij}$ is a curve about $\hat{B_1}$ or $\hat{B_2}$; without loss of generality, let's say it is about $\hat{B_1}$.
Let $\delta_{ij}$ denote the closed (not necessarily embedded) curve $g(\hat{\delta}_{ij})$, which is contained in $[L_1 \cap (X_i \cup X_j \cup L)]$.  

By Lemma \ref{coarsely_sc}, $\delta_{ij}$ admits a null-homotopy in $X$ that is contained in the $M$-neighborhood of $L_1$.
Let $\mathscr{D}$ denote a copy of the 2-disk, with $h \co \mathscr{D} \to X$ representing this null-homotopy.
Let $k \co \partial \mathscr{D} \to \hat{\delta}_{ij}$ be a homeomorphism, and further choose $h$ and $k$ such that $h|_{\partial \mathscr{D}} = g \circ k|_{\partial \mathscr{D}}$.

Recall that $L$ is a $(\phi, N)$ quasi-line.
We claim that, if $R'$ is any constant greater than $\frac{1}{2}\phi(2N) + N$, then the $R'$-ball about $e_1$ separates $L$ into two infinite components, with one contained in $N_N((-\infty, e_1)_l)$ and the other contained in $N_N((e_1, \infty)_l)$.

To see this, suppose for a contradiction that there is a point $q \in (L-B_{R'}(e_1)) \cap N_N((-\infty, e_1)_l) \cap N_N((e_1, \infty)_l)$.
Then there must be points $q^- \in (-\infty, e_1)_l$ and $q^+\in (e_1, \infty)_l$ that do not meet $B_{R'-N}(e_1)$, that are both of distance no more than $N$ from $q$, and hence $d(q^-, q^+) \leq 2N$.
It follows that $length[q^-, q^+]_l \leq \phi (2N)$ and hence that one of $q^-, q^+$ is of distance no more than $\frac{1}{2}\phi (2N)$ from $e_1$.
But $(R'-N)>\frac{1}{2}\phi(2N)$, so neither $q^-$ nor $q^+$ was contained in $B_{R'-N}(e_1)$, a contradiction.
It follows that $B_{R'}(e_1)$ separates $L$ as desired.

Now let $R'$ denote $R-M$, and note that $R'> \frac{1}{2}\phi (2N)+N$.
Let $l^+$ denote $(e_1, \infty)_l$, let $l^-=(-\infty, e_1)_l$, let $L^+$ denote $(N_N(l^+)\cap L)-B_{R'}(e_1)$ and let $L^-$ denote $(N_N(l^-) \cap L)-B_{R'}(e_1)$.
Thus the regions $L^+, L^-$ and $(B_{R'}(e_1)\cap L)$ may not be connected subsets of $L$, but they are disjoint, and their union is $L$.

Note that, while Im$(h)$ is contained in $N_{M}(L_1)$, it need not be contained in $(X_i \cup X_j \cup L)$ nor in $L_1$.
However, as $L_1$ does not meet $B_1 = B_R(e_1)$, it follows that Im$(h)$ does not meet $B_{R'}(e_1)$.
Thus Im$(h) \cap L$ is contained in $L^+ \cup L^-$.

We shall want to consider $h^{-1}(L) \subset \mathscr{D}$.  
Let $\mathscr{C}$ denote the set of components of $h^{-1}(L) \cap \partial \mathscr{D}$.  
Note that, as $h = g \circ k$ on $\partial \mathscr{D}$, as $g$ is cellular and as $L$ is a subcomplex of $S$, we have that $\mathscr{C}$ is finite.

Let $PM = \{ +, -\}$, so each component of $h^{-1}(L)$ is contained in exactly one $h^{-1}(L^\rho)$, with $\rho \in PM$.
Thus the same is true for each element of $\mathscr{C}$.
Let $\mathscr{C}^\rho$ denote $\mathscr{C} \cap h^{-1}(L^\rho)$ for each $\rho \in PM$.
It follows that no component of $\mathscr{C}^+$ is connected to any component of $\mathscr{C}^-$ through $h^{-1}(L)$.

Recall that $h(\partial \mathscr{D}) \subset $ Im$(g)$, so we have that $h(\partial \mathscr{D}) \subset (X_i \cup L \cup X_j)$.  
Hence, for each $c \in \mathscr{C}$, either $c$ meets in $\partial \mathscr{D}$ only $h^{-1}(X_i)$, only $h^{-1}(X_j)$, or meets both $h^{-1}(X_i)$ and $h^{-1}(X_j)$.
Let $\mathscr{C}'$ denote the collection of components $c \in \mathscr{C}$ that, in $\partial \mathscr{D}$, meet both $h^{-1}(X_i)$ and $h^{-1}(X_j)$.  

Recall the identification $k \co \partial \mathscr{D} \to \hat{\delta}_{ij}$, where $h|_{\partial \mathscr{D}} = g \circ k|_{\partial \mathscr{D}}$.
Let $k(\mathscr{C}') = \{ k(c) : c \in \mathscr{C}'\}$, and 
thus the regions in $k(\mathscr{C}')$ are exactly those components of $g^{-1}(L) \cap \hat{\delta}_{ij}$ that meet both $g^{-1}(X_i)$ and $g^{-1}(X_j)$ in $\hat{\delta}_{ij}$.

Let us return to the disk $\mathscr{D}$, and consider again the full components of $h^{-1}(L)$.  Consider only those components that meet $\mathscr{C}'$, so only those components that separate a component of $h^{-1}(X_i) \cap \partial \mathscr{D}$ from a component of $h^{-1}(X_j) \cap \partial \mathscr{D}$.
Let $\mathscr{C}''$ denote the set of these components.


We first claim that $\mathscr{C}''$ is nonempty.  For this, it suffices to note that $\mathscr{C}'$ is nonempty, which will follow if we know that $\hat{\delta}_{ij}$ meets both $g^{-1}(X_i)$ and $g^{-1}(X_j)$.

Thus, we shall shift our attention back to $D_{ij}$.  
Suppose that this is not the case, so without loss of generality $\hat{\delta}_{ij}$ is contained in $g^{-1}(X_i) \cup g^{-1}(L)$.  

Recall that $[g^{-1}(L_1) \cap \hat{(a,e_1)}] \subset [g^{-1}(X-B_1) \cap \hat{(a,e_1)}]$ must be contained in $g^{-1}(L^-)$, and $[g^{-1}(L_1) \cap \hat{(e_1, b)}]$ is contained in $g^{-1}(L^+)$.
Recall that $\hat{\delta}_{ij}$ is a simple closed curve about $\hat{e_1}$, hence it must contain an arc that connects $\hat{(a,e_1)}$ to $\hat{(e_1, b)}$ within $D_j$.
Thus this arc must connect $g^{-1}(L^+)$ to $g^{-1}(L^-)$, and must be entirely contained in $g^{-1}(L)$, since we've assumed that $\hat{\delta}_{ij}\subset [g^{-1}(X_i) \cup g^{-1}(L)]$.
It follows that this arc must meet $g^{-1}(B_{R'}(e_1)) \subset g^{-1}(B_1)$.
But $\hat{\delta}_{ij}$ is contained in $g^{-1}(L_1)$, so this would imply that $g^{-1}(B_1)$ meets $g^{-1}(L_1)$, and hence that $B_1$ meets $L_1$, a contradiction.
Hence $\mathscr{C}''$ must be nonempty.


Note that we can see from the above paragraph that $\hat{\delta}_{ij}$ must meet both $g^{-1}(L^+)$ and $g^{-1}(L^-)$, and any arc in $\hat{\delta}_{ij}$ connecting $g^{-1}(L^+)$ to $g^{-1}(L^-)$ must lie in $g^{-1}(X_i)$ or $g^{-1}(X_j)$.
Furthermore, there must be at least one such arc in $g^{-1}(X_i)$ and at least one in $g^{-1}(X_j)$.

It follows that $\mathscr{C}' \cap \mathscr{C}^-$ and $\mathscr{C}' \cap \mathscr{C}^+$ must both be nonempty.  
Any component of $\mathscr{C}'$ is contained in a component of $\mathscr{C}''$, so it follows that $\mathscr{C}'' \cap h^{-1}(L^+)$ and $\mathscr{C}'' \cap h^{-1}(L^-)$ are both nonempty.
But $h^{-1}(L^+)$ and $h^{-1}(L^-)$ do not meet, so $|\mathscr{C}''| \geq 2$.


Thus there exists a component $\eta$ of $\partial \mathscr{D} - \mathscr{C}''$ that meets a component $C^+$ of $(\mathscr{C}'' \cap h^{-1}(L^+))$ and a component $C^-$ of $(\mathscr{C}'' \cap h^{-1}(L^-))$.
Suppose that $|\mathscr{C}''| = 2$, so $\mathscr{C}'' = \{ C^+, C^-\}$.

Then the component of $\mathscr{D}-(C^+ \cup C^-)$ that contains $\eta$ meets exactly one other component of $\partial \mathscr{D}-\mathscr{C}''$, say $\eta'$.
Let $p$ be a point in the interior of $\eta$, fix $p'$ in the interior of $\eta'$, and consider the two components of $\partial \mathscr{D}-\{ p, p'\}$.

One of these components must contain $C^+ \cap \partial \mathscr{D}$ and be disjoint from $C^- \cap \partial \mathscr{D}$ and the other must contain $C^- \cap \partial \mathscr{D}$ and be disjoint from $C^+\cap \partial \mathscr{D}$.  Let $\gamma^+$ denote the former component, and $\gamma^-$ the latter.

We can moreover choose $p$ and $p'$ so that $h(p)$ and $h(p')$ are both contained in $X_i$ or both contained in $X_j$.  
Without loss of generality, let's suppose both are contained in $X_i$. 

Consider again $D_{ij}$.  
Recall that $\hat{[a,b]}$ is a simple curve that contains $\hat{e_1}$ and separates $D_{ij}$ into two components, and note that $k(p), k(p') \in \hat{\delta}_{ij}$ are contained in the same component of $D_{ij}-\hat{[a,b]}$.

It may be the case that $k(\gamma^-)$ meets $\hat{(e_1, b]}$, but as $\hat{(e_1,b]} \subset g^{-1}(L^+)$, and $\gamma^-$ does not meet $\mathscr{C}^+$, note that $k(\gamma^-)$ will not cross from $g^{-1}(X_i)$ to $g^{-1}(X_j)$ through $\hat{(e_1, b]}$.
Similarly $k(\gamma^+)$ does not cross from $g^{-1}(X_i)$ to $g^{-1}(X_j)$ through $\hat{[a,e_1)}$.
Thus $\hat{\delta}_{ij} = [k(\gamma^+) \cup k(\gamma^-) \cup k(p) \cup k(p')]$ is the union of two segments that meet each other in the same component of $D_{ij}-\hat{[a,b]}$, and each component only crosses from $g^{-1}(X_i)$ to $g^{-1}(X_j)$ through $\hat{[a,e_1)}$ or $\hat{(e_1,b]}$.
As $g^{-1}(B_1)$ separates $g^{-1}(L)$ and contains $\hat{e_1}$, and $\hat{\delta}_{ij}$ does not meet $g^{-1}(B_1)$, it follows that $\hat{\delta}_{ij}$ has winding number zero about $\hat{e_1}$, a contradiction.  
Hence $|\mathscr{C}''| > 2$.

Before moving on to the general case, we note that we did not need such strong conditions on $\mathscr{D}$ for the above argument to work.
Suppose still that $h$ maps $\mathscr{D}$ into $N_{M}(L_1)$, and $h$ maps $\partial \mathscr{D}$ into $(X_i \cup X_j \cup L) $, so that we may define $\mathscr{C}, \mathscr{C}', \mathscr{C}''$ as before.
Suppose that $k$ is an embedding of $\partial \mathscr{D}$ into $D_{ij}$.
Retain $\hat{\delta}_{ij}$ and the map $g$ as was defined earlier.

Suppose however that $h$ is not equal to $g \circ k$ when restricted to a region of $\partial \mathscr{D}$, and moreover, that the image under $k$ of this region need not be in $\hat{\delta}_{ij}$.

In particular, let $c \in \mathscr{C}-\mathscr{C}'$ be this region and suppose that $g \circ k = h$ except on $c$.

Let $\iota_1$ and $\iota_2$ denote the endpoints of $c$ in $\partial \mathscr{D}$, let $\hat{c}$ denote the arc in $\hat{\delta}_{ij}$ that connects $k(\iota_1)$ to $k(\iota_2)$ and is not $k(\partial D - c)$.
As $c \notin \mathscr{C}'$, $k(\iota_1)$ and $k(\iota_2)$ must both be in $D_i$ or both be in $D_j$.
Hence we may choose a path $\iota$ that connects $k(\iota_1)$ to $k(\iota_2)$ and is contained entirely in $D_i$ or $D_j$, and does not meet $\hat{[a,b]}$, except possibly at its endpoints, if $k(\iota_1)$ or $k(\iota_2)$ is in $\hat{[a,b]}$.

Consider the closed curve attained by replacing $\hat{c}$ in $\hat{\delta}_{ij}$ with $\iota$, and the closed curve $\iota \cup \hat{c}$.
As we are still assuming that $|\mathscr{C}''| = 2$, note that our previous argument may be applied to show that the former curve has winding number zero about $\hat{e_1}$.
Thus if $\hat{c}$ is such that $\iota \cup \hat{c}$ also has winding number zero about $\hat{e_1}$, then it would follow that $\hat{\delta}_{ij}$ has winding number zero about $\hat{e_1}$, a contradiction.

Moreover, if there were more than one region like $c$ in $\mathscr{C}-\mathscr{C}' \subset \partial \mathscr{D}$ on which $\partial \mathscr{D}$ and $\hat{\delta}_{ij}$ did not correspond via $k$, and where each corresponding pair of paths $\iota$ and $\hat{c}$ made a curve with winding number zero about $\hat{e_1}$, then a contradiction would also follow.


To complete the proof that $\hat{\delta}_{ij}$ cannot be a closed curve, we shall induct on $|\mathscr{C}''|$ to get a contradiction in every case.
Let $D_{ij}$, $g$ and $\hat{\delta}_{ij}$ be defined as above.

We shall work with exclusively with the following type of situation.
Let $\mathscr{D}_0$ be a disk, let $h_0$ be any continuous map of $\mathscr{D}_0$ into $N_{M}(L_1)$, with $h_0$ taking $\partial \mathscr{D}_0$ into $(X_i \cup X_j \cup L)$, and let $k_0 \co \partial \mathscr{D}_0 \to D_{ij}$ be an embedding.
Parallel with our definitions of the sets associated to $\mathscr{D}$ and the map $h$,
\begin{itemize}
\item{let $\mathscr{C}_0$ denote the set of components of $h_0^{-1}(L) \cap \partial \mathscr{D}_0$,}

\item{let $\mathscr{C}_0'$ denote the subset of components $c \in \mathscr{C}_0$ that meet both $h_0^{-1}(X_i)$ and $h_0^{-1}(X_j)$ in $\partial \mathscr{D}_0$,}
\item{let $\mathscr{C}_0''$ denote the components of $h_0^{-1}(L)$ in $\mathscr{D}_0$ that meet $\mathscr{C}_0'$, and}

\item{for each $\rho \in PM$, let $\mathscr{C}_0^\rho$ denote $\mathscr{C}_0 \cap h_0^{-1}(L^\rho)$.}
\end{itemize}

We will assume that $|\mathscr{C}_0''|$ is finite and greater than 2.

Let $\mathscr{S}_0 \subset \mathscr{C}_0-\mathscr{C}'_0$ and suppose that the restriction of $k_0$ to $\partial \mathscr{D}_0-\mathscr{S}_0$ is an embedding into $D_{ij}$ with image contained in $\hat{\delta}_{ij}$.
Suppose further that, on the restricted domain $\partial \mathscr{D}_0-\mathscr{S}_0$, $h_0 = g \circ k_0$.

For each $s \in \mathscr{S}_0$, let $\hat{s}$ denote the arc that is the component of $\hat{\delta}_{ij}$ minus the image under $k_0$ of the endpoints of $s$, that does not meet $k_0(\partial D - s)$.
Let $\hat{\mathscr{S}_0} = \{ \hat{s} : s \in \mathscr{S}_0\}$.
We shall assume that $k_0|_{\partial \mathscr{D}_0-\mathscr{S}_0}$ does not interchange the ``order'' of segments inherited from $\partial \mathscr{D}_0$ and $\hat{\delta}_{ij}$, in the sense that there exists a homotopy in $D_{ij}$ rel $k_0(\partial \mathscr{D}_0-\mathscr{S}_0)$ that takes each $\hat{s}$ to $k_0(s)$.

Note that since $s \notin \mathscr{C}'_0$, the two components of $\partial \mathscr{D}_0-\mathscr{C}_0$ that $s$ meets are both contained in $h_0^{-1}(X_i)$ or are both contained in $h_0^{-1}(X_j)$.
Thus both endpoints of $\hat{s}$ are contained in $D_i$ or $D_j$.
Hence for each $s \in \mathscr{S}_0$, there exists a path $\iota_s$ that connects the endpoints of $\hat{s}$, the interior of which is contained in $D_{ij}-\hat{[a,b]}$, and is such that a small neighborhood $n$ of $\hat{s}$ in $\hat{\delta}_{ij}$ has $(n-\hat{s}) \cup \iota_s$ is a path in $D_i$ or $D_j$.  
(In other words, $\iota_s$ may meet $\hat{[a,b]}$ at its endpoints, but $(n-\hat{s}) \cup \iota_s$ does not cross $\hat{[a,b]}$ from $D_i$ into $D_j$ or vice versa.)

Also assume that the collection of closed curves $\hat{s} \cup \iota_s$ all have winding number zero about $\hat{e_1}$.

We shall refer to the maps and subspaces associated to $\mathscr{D}_0$ and $h_0$ defined and satisfying the hypotheses above as the {\em data} associated to $\mathscr{D}_0$.
In the case that $\mathscr{D}_0$ and its data satisfies all the hypotheses above, we shall say that $\mathscr{D}_0$ (with data implicit) {\em has property P}.

Note that our original disk $\mathscr{D}$, with $\mathscr{S} =$ \o, has property P.

Fix $n>2$, and our induction hypothesis is the following.
For any disk $\mathscr{D}_0$ with property P, if $|\mathscr{C}''_0| = (n-1)$, then the closed curve attained from $\hat{\delta}_{ij}$ by replacing each $\hat{s} \in \hat{\mathscr{S}_0}$ with $\iota_s$ has winding number zero about $\hat{e_1}$.
As we have assumed that each curve $\hat{s} \cup \iota_s$ has winding number zero about $\hat{e_1}$ as well, it follows that $\hat{\delta}_{ij}$ has winding number zero about $\hat{e_1}$.
But this contradicts the construction of $\hat{\delta}_{ij}$, and thus no $\mathscr{D}_0$ with such data can exist.

Assume now that $\mathscr{D}_0$ has property P, and that $|\mathscr{C}_0''| = n$.  
So $\mathscr{S}_0$ is some (possibly empty) subset of $\mathscr{C}_0-\mathscr{C}'_0$ and the restriction of $k_0$ to $(\partial \mathscr{D}_0 - \mathscr{S}_0)$ has image in $\hat{\delta}_{ij}$, and is such that $h_0$ restricted to $(\partial \mathscr{D}_0 - \mathscr{S}_0)$ is equal to $g \circ k_0$.
In addition, for each $s \in \mathscr{S}_0$, the closed curve $\hat{s} \cup \iota_s$ has winding number zero about $\hat{e_1}$.

We shall show that we can reduce to the $|\mathscr{C}_0''| = (n-1)$ case.
Let $C'' \in \mathscr{C}''_0$ be an element of $\mathscr{C}''_0$ such that, for some arc $c$ of $fr(C'')$, $c$ separates $C''$ from every other component of $\mathscr{C}''_0$.
Let $\iota_1, \iota_2$ denote the endpoints of $c$, and let $d$ denote the component of $\partial \mathscr{D}_0-\{ \iota_1, \iota_2\}$ that does not meet any element of $\mathscr{C}''_0$ other than $C''$.

Now consider the simple closed curve attained from $\partial \mathscr{D}_0$ by replacing $d$ with $c$.  

Let $\mathscr{D}_1$ denote the disk in $\mathscr{D}_0$ that is bounded by this region.
Note that the restriction of $h_0$ to $\mathscr{D}_1$ is a map into $N_{M}(L_1)$, and, as $c \subset h_0^{-1}(L)$, the restriction of $h_0$ to $\partial \mathscr{D}_1$ is a map into $(X_i \cup X_j \cup L) $.
Let $h_1$ denote the restriction of $h_0$ to $\mathscr{D}_1$, so $h_1$ yields data $\mathscr{C}_1, \mathscr{C}_1'$ and $\mathscr{C}''_1$, defined analogously to the data $\mathscr{C}_0$, etc., that was defined with respect to $\mathscr{D}_0$ and $h_0$.
Note that $c \in \mathscr{C}_1 - \mathscr{C}_1'$, hence $\mathscr{C}_1'' = (\mathscr{C}_0'' - C'')$ and $|\mathscr{C}''_1| = (n-1)$.

Let $k_1 \co \partial \mathscr{D}_1 \to D_{ij}$ be equal to $k_0$ on the domain $\partial \mathscr{D}_1 \cap \partial \mathscr{D}_0$, and take $c$ to $k_0(d)$.
Recall that $k_0$ is an embedding, thus $k_1$ is also an embedding.

Thus we shall have that $\mathscr{D}_1$, with $\mathscr{S}_1 = [\mathscr{S}_0  \cap \partial \mathscr{D}_1] \cup \{ c\}$, has property P if we can find a path $\iota_c$ in $D_{ij}$ that connects $k_0(\iota_1)$ with $k_0(\iota_2)$, does not cross $\hat{[a,b]}$, and is such that $\iota_c \cup k_0(d)$ has winding number zero about $\hat{e_1}$.

As $c$ meets only $h_0^{-1}(X_l)$ in $\partial \mathscr{D}_1$, for some $l = i$ or $j$, a small neighborhood of each of $k_0(\iota_1)$ and $k_0(\iota_2)$ in $k_0(\partial \mathscr{D}_1 - c)$ must be contained in $k_0^{-1}(X_l)$.

It follows that there is a path $\iota_c$ in $D_{ij}$ that connects $k_0(\iota_1)$ to $k_0(\iota_2)$ and is such that a small neighborhood of $k_0(c)$ in $k_0(\partial \mathscr{D}_1)$, with $k_0(c)$ replaced by $\iota_c$, does not cross $\hat{[a,b]}$.
Next we will see that $k_0(d) \cup \iota_c$ has winding number zero about $\hat{e_1}$.

The components of $\mathscr{C}_0-\mathscr{C}_0'$ that meet $d$ will complicate our argument slightly.
Let $\mathscr{T}$ denote those components of $(\mathscr{C}_0-\mathscr{C}_0')$ that are contained in $d$.
If $t \in \mathscr{T} \cap \mathscr{S}_0$, then let $\iota_t$ and $\hat{t}$ be the paths in $D_{ij}$ given in the hypotheses for $\mathscr{D}_0$, and recall that by assumption, the closed curve $\iota_t \cup \hat{t}$ has winding number zero about $\hat{e}_1$.

For the remaining $t \in \mathscr{T}$, there is a path $\iota_t$ with interior in $D_{ij}-\hat{[a,b]}$ that connects the image under $k_0$ of the endpoints of $t$, and is such that a small neighborhood of $k_0(t)$ in $k_0(d)$, with $k_0(t)$ replaced by $\iota_t$, does not cross $\hat{[a,b]}$.
As $t \notin \mathscr{S}_0$, $k_0(t) \subset g^{-1}(L^\rho)$ for some $\rho \in PM$.
Thus, for these components $t \in \mathscr{T}-\mathscr{S}_0$, the curve $k_0(t) \cup \iota_t$ crosses only $\hat{[a,e_1)}$ or $\hat{(e_1, b]}$, and thus has winding number zero about $\hat{e_1}$.

Hence, for each $t \in \mathscr{T}$, $\iota_t$ together with $\hat{t}$ (if $t \in \mathscr{S}_0$) or $k_0(t)$ (if $t \notin \mathscr{S}_0$) has winding number zero about $\hat{e_1}$.
So in order to show that $k_0(d) \cup \iota_c$ has winding number zero about $\hat{e_1}$, it suffices to show that $\iota_c$, together with the path $\hat{d}'$, which is attained by starting with $k_0(d)$ and replacing $\hat{t}$ or $k_0(t)$ by $\iota_t$ for each $t \in \mathscr{T}$, has winding number zero about $\hat{e_1}$.

To see this, recall that the components of $(\iota_c \cup \hat{d}')-k_0(d)$ are segments that do not cross $\hat{[a,b]}$ and are connected by arcs in $k_0(d)$.
Also recall that $d$ meets only one component of $\mathscr{C}_0''$, hence these arcs in $k_0(d)$ are all contained in $g^{-1}(L^+)$, or are all in $g^{-1}(L^-)$.
Thus $\iota_c \cup \hat{d}'$ has winding number zero about $\hat{e_1}$.

It follows that $\iota_c \cup k_0(d)$ has winding number zero about $\hat{e_1}$, and hence $\mathscr{D}_1$ has property P.
But $|\mathscr{C}_1''| = (n-1)$, so the closed curve attained by replacing each $\hat{s} \in \hat{\mathscr{S}}_1$ with $\iota_s$ has winding number zero about $\hat{e_1}$, and hence so does $\hat{\delta}_{ij}$.
We have reached a contradiction, and it follows that $|\mathscr{C}_0''|=n$ is an impossibility, so the proof is complete.
\end{proof}

\noindent {\bf Lemma \ref{second_lemma}.}  {\it
Let $D_{ij}$, $\hat{\delta}_{ij}$, $g \co D_{ij} \to (X_i \cup X_j \cup L) \subset X$, $\pi \co L_1 \to l_1$ and $\chi \subset l_1$ be as in the proof of Theorem \ref{2.1}.
Let $\hat{x},\hat{y}$ be the endpoints of $\hat{\delta}_{ij}$, and let $x=g(\hat{x}), y=g(\hat{y})$.
Let $\hat{\gamma}_x$ denote the segment of $\hat{\delta}_{ij}$ from $\hat{x}$ to $g^{-1}([e_1, e_2]_l)$,
let $\gamma_x$ denote the path $g(\hat{\gamma}_x)$ from $x$ to $[e_1, e_2]_l$, and let $\gamma_x'$ denote the component of $CH(\pi (\gamma_x)),l_1)-\chi$ that contains $x$.
Define $\hat{\gamma}_y$ and $\gamma_y'$ similarly.

Then $\pi(x) \notin \gamma_y'$ and $\pi(y) \notin \gamma_x'$.
}

\begin{proof}
We shall prove that $\pi (x)$ is not contained in $\gamma_y'$; that $\pi(y) \notin \gamma_x'$ shall follow analogously.

To prove that $\pi(x) \notin \gamma_y'$, it suffices to show that there is no path in $L_1$ from $x$ to $\gamma_y'$, of length less than or equal to $N$. 
Moreover, it suffices to show that there is no path in $L_1$ from $x$ to $\pi(\gamma_y)$ of length less than or equal to $(\frac{1}{2}\phi(2N) + N)$, as $\gamma_y'$ is comprised of a $2N$-chain in $\pi (\gamma_y)$, together with connecting segments in $l_1$ of length no more than $\phi (2N)$.
Furthermore, it suffices to show that there is no path in $L_1$ from $x$ to $\gamma_y$ of length less than or equal to $(\frac{1}{2}\phi (2N)+2N)$.

Suppose that this is not the case, and we will show that, if $\pi (x) \subset \gamma_y'$, then we can alter $g|_{D_i}$ or $g|_{D_j}$ so that $\hat{\delta}_{ij}$ becomes a closed curve.
By Lemma \ref{scc_case}, this is a contradiction, and thus $\pi (x)$ cannot be contained in $\gamma_y'$.

Let $\alpha$ be a path from $x$ to $\gamma_y$ that is contained in $L_1$ and has length no more than $(\frac{1}{2}\phi (2N) + 2N)$.  
Recall that $d_{inf}(x,L)>K_1$, and $K_1 > (\frac{1}{2}\phi (2N) + 2N)$.
Thus the path $\alpha$ does not meet $L$, and moreover does not meet a small neighborhood of $L$.

Let $\alpha_0$ denote the endpoint of $\alpha$ that is contained in $\gamma_y$, and let $\hat{\alpha}_0$ denote a point in $g^{-1}(\alpha_0) \cap \hat{\gamma_y}$.

Recall that $\hat{x}$ is contained in $\hat{p}_i$ or $\hat{p}_j$ - let's say that, without loss of generality, $\hat{x} \in \hat{p}_i$.  
Thus $\hat{x} \in g^{-1}(X_i)$, and as $\alpha \cap L = $ \o, $\hat{\alpha}_0$ is also in $g^{-1}(X_i) \subset D_i$.

Next we claim that there exists some path $\beta$ from $\hat{x}$ to $\hat{\alpha_0}$ that is contained in $D_i$ and does not meet $g^{-1}([e_1, e_2]_l)$.  
We first claim that $[e_1, e_2]_l$ does not meet $p_i$.

To see this, first recall that by construction, $[a,b]_l$ does not meet $p_i$ outside of $r_2$-neighborhoods of $a$ and $b$.
Also recall that $L_1$ meets the $R$-neighborhood of any point in $[e_1+\epsilon, e_2-\epsilon]_l$, so the $(R+\epsilon)$-neighborhood of each point in $[e_1, e_2]_l$ meets $L_1$, and that $L_1$ $K$-separates $a$ and $b$.
Thus no point in $[e_1,e_2]_l$ is contained in the $(K-R-\epsilon)$-neighborhoods of $a$ or $b$.

Recall that 
$$K>r_2+R+\epsilon,$$
and thus that $r_2 < (K-R-\epsilon)$.
It follows that no point in $[e_1, e_2]_l$ is contained in the $r_2$-neighborhoods of $a$ or $b$, and therefore $[e_1, e_2]_l$ does not meet $p_i$.

Hence in $D_{ij}$, $g^{-1}([e_1,e_2]_l)$ does not meet $g^{-1}(p_i)$, and in particular does not meet $\hat{p}_i$.
Similarly, $g^{-1}([e_1,e_2]_l)$ does not meet $\hat{p}_j$, and thus does not meet $\partial D_{ij}$.

Recall that $\hat{x} \in \hat{p}_i$, and $\hat{\alpha_0}$ is contained in $\hat{\gamma}_y$, which is a path that does not meet $g^{-1}([e_1, e_2]_l)$ but does contain $\hat{y} \in (\hat{p}_i \cup \hat{p}_j)$.
It follows that $g^{-1}([e_1, e_2]_l)$ does not separate $\hat{x}$ from $\hat{\alpha_0}$ in $D_{ij}$.

Moreover, since $g^{-1}([e_1, e_2]_l) \cap \hat{[a,b]} = \hat{[e_1, e_2]}$ and $\partial D_i = \hat{[a,b]} \cup \hat{p}_i$, a similar argument shows that $g^{-1}([e_1, e_2]_l)$ does not separate $\hat{x}$ from $\hat{\alpha_0}$ in $D_i$.
Thus let $\beta$ be a path in $D_i$ connecting $\hat{x}$ to $\hat{\alpha_0}$, that does not meet $g^{-1}([e_1, e_2]_l)$.
One may further assume that $\beta$ does not meet $\partial D_i$.

Let $\hat{\delta}'$ denote the arc in $\hat{\delta}_{ij}$ from $\hat{x}$ to $\hat{\alpha}_0$, so $\hat{\delta}' \cup \beta$ is a closed curve in $D_{ij}$.
We claim that this curve has nonzero winding number about $\hat{e_1}$ or $\hat{e_2}$.
(Since we are only worried about showing that this winding number is nonzero, we need not be careful about curve orientation.)

Let $\hat{\delta}''$ denote $\hat{\delta}_{ij}-\hat{\delta}'$.
Thus $\hat{\delta}'' \subset \hat{\gamma}_y$, so $\hat{\delta}''$ does not meet $g^{-1}([e_1, e_2]_l)$.

Let $\partial_0$, $\partial_1$ denote the two components of $\partial D_{ij}-\{ \hat{x}, \hat{y}\}$, so that, for some $\{ m,m'\} = \{ 1,2\}$, $\hat{\delta}_{ij} \cup \partial_0$ has winding number $\pm 1$ about $\hat{e_m}$ and winding number zero about $\hat{e_{m'}}$, and $\hat{\delta}_{ij} \cup \partial_1$ has winding number $\pm 1$ about $\hat{e_{m'}}$ and winding number zero about $\hat{e_m}$.

Now consider the closed curves $(\partial_k \cup \hat{\delta}'' \cup \beta)$ for $k=0$ and $1$.
These curves do not meet $g^{-1}([e_1, e_2]_l)$, so in particular they do not meet $\hat{[e_1, e_2]_l}$, and hence do not separate $\hat{e}_1$ and $\hat{e}_2$.
As $\partial D_{ij}$ has winding number one about $\hat{e}_1$ and $\hat{e}_2$, one curve $(\partial_k \cup \hat{\delta}'' \cup \beta)$ has winding number zero about $\hat{e_1}$ and $\hat{e_2}$, and the other has winding number $\pm 1$ about both points.
Suppose without loss of generality that $(\partial_1 \cup \hat{\delta}'' \cup \beta)$ has winding number zero about the two points.

As $(\partial_1 \cup \hat{\delta}'' \cup \beta)$ has winding number zero about $\hat{e_{m'}}$ and $(\partial_1 \cup \hat{\delta}_{ji})$ has winding number $\pm 1$ about $\hat{e_{m'}}$, it follows that $(\hat{\delta}' \cup \beta)$ must have winding number $\pm 1$ about $\hat{e_{m'}}$.

Next we shall redefine $g$ on a small neighborhood of $\beta$.
The restriction of $g$ to $(\hat{p}_i \cup \hat{p}_j \cup \hat{[a,b]})$ shall be unchanged.

By a ``small neighborhood'' of $\beta$, we shall mean a small open neighborhood that is contained in the interior of $D_i$ and does not meet $g^{-1}([e_1, e_2]_l)$.  
Such a neighborhood exists since $\beta$ does not meet $g^{-1}([e_1, e_2]_l)$ or $\partial D_i$, and $\beta, g^{-1}([e_1, e_2]_l)$ and $\partial D_i$ are all closed.

Recall that $\beta \subset D_i$, so $g(\beta)$ is a path contained in $X_i \cup L$ from $x$ to $\alpha_0$.
We saw earlier that there is a path $\alpha$ contained in $(L_1 \cap X_i)$ from $x$ to $\alpha_0$.

Since $\pi_1 (X_i \cup L) = 0$, it follows that we can homotope $g(\beta)$ within $(X_i \cup L)$ to $\alpha$.
Homotope the map $g$ on a small neighborhood of $\beta$ so that $g$ now takes $\beta$ to $\alpha$ (and $g$ is not altered on any portion of the neighborhood of $\beta$ that meets $\hat{\delta}''$).

Recall the definition of $\hat{\delta}_{ij}$, defined with respect to the old map $g$, and consider now the simple curve, call it $\hat{\delta}_{ij}'$, defined in the same manner, but with respect to the altered map $g$.
As $\hat{\delta}' \cup \beta$ is a closed curve with winding number $\pm 1$ about $\hat{e_{m'}}$, it follows that $\hat{\delta}_{ij}'$ must be a simple closed curve about $\hat{e_{m'}}$.

By Lemma \ref{scc_case} this cannot happen.
Thus $\pi(x)$ cannot be contained in $\gamma_y'$, as desired.
Similarly $\pi(y)$ cannot be contained in $\gamma_x'$.

\end{proof}

\bibliographystyle{alpha}
\bibliography{my_biblio}

\end{document}